\documentclass[a4paper, 11pt, table]{amsart}
\usepackage{amsmath, amsthm, amscd, amssymb, amsfonts, amsxtra, amssymb, latexsym, mathtools}
\usepackage{faktor}
\usepackage{enumerate}
\usepackage{verbatim}
\usepackage{textcomp}
\usepackage{caption}
\usepackage{multirow}
\usepackage[table]{xcolor}
\usepackage{graphicx, epsfig, tikz, pifont}
\usepackage[utf8x]{inputenc}

\usepackage{float}
\usepackage{pb-diagram}

\usepackage{tikz-cd}

\usetikzlibrary{arrows,%
                petri,%
                topaths}%
\usetikzlibrary{automata}                
\usetikzlibrary{positioning}

\usepackage[justification=centering]{caption}

\newcommand{\Z}{\mathbb{Z}}
\newcommand{\C}{\mathcal{C}}
\newcommand{\D}{\mathbb{D}}
\newcommand{\Q}{\mathbb{Q}}
\newcommand{\N}{\mathbb{N}}
\newcommand{\ff}{\mathbb{F}}

\newcommand{\MG}{\mathcal{M}(G)}

\newcommand{\Prep}{\mathcal{P}}

\newcommand{\sk}{\smallskip}

\newcommand{\Sym}{\mathbb{S}}

\newcommand{\congP}{\sim_{\mathcal{P}}}

\newcommand{\Di}{\mathbb{D}}

\newtheorem{thm}{Theorem}[section]
\newtheorem{prop}[thm]{Proposition}
\newtheorem{lema}[thm]{Lemma}
\newtheorem{coro}[thm]{Corollary}
\theoremstyle{definition}
\newtheorem{rem}[thm]{Remark}
\newtheorem{exam}[thm]{Example}

\theoremstyle{remark}

\theoremstyle{definition}
\newtheorem{defi}[thm]{Definition}

\usepackage{color}

\usepackage{hyperref}
\hypersetup{colorlinks = true,	allcolors  = blue}

\begin{document}
\numberwithin{equation}{section}
\title{Invariant metrics on finite groups}
\author{Ricardo A.\@ Podest\'a, Maximiliano G.\@ Vides}
\dedicatory{\today} 
\keywords{Invariant metric, weight function, finite groups, symmetry groups}
\thanks{2010 {\it Mathematics Subject Classification.} Primary 05E16;\, Secondary 54E45, 05E18, 20B05, 20B35.}
\thanks{Partially supported by CONICET, FONCyT and SECyT-UNC}

\address{Ricardo A.\@ Podest\'a, FaMAF -- CIEM (CONICET), Universidad Nacional de C\'ordoba, \newline
 Av.\@ Medina Allende 2144, Ciudad Universitaria, (5000) C\'ordoba, Rep\'ublica Argentina. \newline
{\it E-mail: podesta@famaf.unc.edu.ar}}

\address{Maximiliano G.\@ Vides 
FaMAF -- CIEM (CONICET), Universidad Nacional de C\'ordoba, \newline
 Av.\@ Medina Allende 2144, Ciudad Universitaria, (5000)   C\' ordoba, Rep\'ublica Argentina. \newline
{\it E-mail: mvides@famaf.unc.edu.ar}}

\begin{abstract}
We study invariant and bi-invariant metrics on groups focusing on finite groups $G$. We show that non-equivalent (bi) invariant metrics on $G$ are in 1-1 correspondence with unitary symmetric (conjugate) partitions on $G$. To every metric group $(G,d)$ we associate to it the symmetry group and the weighted graph of distances. Using these objects we can classify all equivalence classes of invariant and bi-invariant metrics for small groups. 
We then study the number of non-equivalent invariant and bi-invariant metrics on $G$. We give an expression for the number of such metrics in terms of Bell numbers, with closed expressions for certain groups such as abelian, dihedral, quasidihedral and dicyclic groups. We then characterize all the groups (finite or not) in which every invariant metric is also bi-invariant.
We give the number of non-equivalent invariant and bi-invariant metrics for all the groups of order up to 32.
\end{abstract}

\maketitle

\section{Introduction}
In this work we study invariant metrics on groups, focusing on the finite case. So, from now on $G$ will always denote a finite group (unless explicit mention of the contrary), $d$ a metric 
on $G$ and $w$ a weight function on $G$. We now recall these definitions.

A \textit{metric} on a set $X$ is a function $d : X \times X \rightarrow\mathbb{R}_{\geq 0}$ such that for every $x,y,z\in X$  satisfies the three conditions:
\begin{enumerate}[($a$)]
\setlength{\itemsep}{1mm}
\item $d(x,y)\geq 0$ and $d(x,y)=0 \,\Leftrightarrow\, x=y$ (positiveness), 
\item $d(x,y)=d(y,x)$ (commutativity), 
\item $d(x,y)\leq d(x,z) + d(x,z)$ (triangle inequality).
\end{enumerate}
One says that $(X,d)$ is a \textit{metric space}. If $d$ only satisfies ($a$)-($b$), $d$ is a \textit{semimetric} and $(X,d)$ is a \textit{semimetric space}. If $X=G$ is a group, then we will say that $(G,d)$ is a \textit{metric group}, or a \textit{metric semigroup} if condition ($c$) does not hold. 
%Further, if $G$ is abelian we will speak of \textit{abelian metric group} (or \textit{abelian semimetric group}).

A \textit{weight} on $G$ is a function $w : G \rightarrow \mathbb{R}_{\ge 0}$ such that for every $x,y\in G$ satisfies:
\begin{enumerate}[($a$)]
\setlength{\itemsep}{1mm}
  \item $w(x) \geq 0$ (positiveness) and $w(x)=0$ if and only if  $x=e$,
  \item $w(x)=w(x^{-1})$ (symmetry),
  \item $w(xy) \leq w(x) + w(y)$ (triangle inequality).
	\end{enumerate}
In general, with minimum extra conditions, given a metric one can define a weight and conversely. 

The metric $d$ is called \textit{integral} if only takes integer values, that is $d : X \times X \rightarrow \N_{0}$ (equivalently, the associated weight function takes integer values). 
Since our main motivation to study metrics on groups are applications to combinatorics (coding theory, for instance), we will be mainly interested in integral metrics. 

The most common metrics used in coding theory are the Hamming distance for codes over fields and the Lee distance for codes over rings (see the seminal works of Lee \cite{Lee}, Nechaev \cite{Ne} and Hammons et al \cite{HKCSS}), which are integral invariant  metrics. 
Brualdi et al \cite{Brualdi} introduced invariant metrics associated to a partially ordered set, the so called poset metrics. For instance, the Hamming metric is a poset metric for the antichain poset. The Rosembloom-Tsfasman metric or RT-metric \cite{RT-metric} which is also widely used in coding theory is a particular case of poset metric \cite{Brualdi}. Another important invariant metric is the one given by the homogeneous weight on rings \cite{Gre} (when this weight indeed defines a metric). We note that poset metrics are also integral metrics and the homogeneous metrics can be rescaled to become integral.

We are interested in metrics defined over groups. A general family of such metrics is given by word metrics.
A complete treatment of metrics on groups can be found in Chapter 10 of the Encyclopedia of distances \cite{Deza}.
Our approach to study metrics on groups will be basically through (unordered) partitions. 
In \cite{Ba95} the author considered metrics with ordered partitions, however metrics obtained from different orders of the same partition are all $\Gamma$-equivalent (see the comments after \eqref{SGd}) to each other.
A relation between metrics and partitions over additive codes is given in \cite{Gluesing}, mainly for the study of duality and the MacWilliams identities.

\subsubsection*{Outline and results}
We now summarize the main results in the paper. 
In Section \ref{S2} we study invariant metrics on an arbitrary group in terms of its symmetric partitions. 
In Section \ref{S3}, for a finite metric group %$(G,d)$ 
we study their symmetry group and the associated graph of distances. In Section \ref{S4} we give a detailed study of the invariant metrics for small groups. 
In Sections \ref{S5} and \ref{S5b} we count the number of invariant and bi-invariant metrics on a finite group.  In the final section we give the number of invariant and bi-invariant metrics for all the groups of order up to $32$.

More precisely, in Section \ref{S2} we define right-invariant 
%(also left-invariant and bi-invariant) 
metrics on a group $G$ and denote by 
$\mathcal{M}(G)$ the set of all these metrics. The weight $w$ associated to an invariant metric $d$ induces a partition $P(G,d)$ on  $G$ (see \eqref{eq part}). 
We denote by $\mathcal{P}(G)$ the set of all unitary symmetric partitions on $G$ (see the definition before \eqref{PG}). Two metrics are $\mathcal{P}$-equivalent if they define the same unitary symmetric partitions on $G$. In Proposition \ref{prop 1-1} we show that the set $\mathcal{M}(G)/_{\sim \mathcal{P}}$ of non-equivalent invariant metrics on $G$ are in 1-1 correspondence with $\mathcal{P}(G)$.
Also, the non-equivalent bi-invariant (left and right) metrics on $G$ are in a 1-1 correspondence with unitary symmetric conjugate partitions of $G$.

In the next section, given a metric group $(G,d)$, we study the symmetry group $\Gamma(G,d)$ of permutations of $G$ preserving the metric $d$ and its relation with its distance graph. In general, it is difficult to determine the group $\Gamma(G,d)$, however we can give some set contentions. 
In Proposition~\ref{Partsimetrica} we show that if $G$ is abelian then we have $\D(G) \le \Gamma(G,d) \le \Sym_G$, where $\D(G)$ is a generalization of the dihedral group (see Definition \ref{Dihedral}). For a general group, if $d$ is bi-invariant, in Proposition~\ref{Gbiinv} we get that $G \rtimes \operatorname{Aut}_{ecp}(G) \le \Gamma(G,d) \le \Sym_G$, 
where $\operatorname{Aut}_{ecp}(G)$ is the subgroup of extended class preserving automorphisms of $G$ (see the comments before \eqref{automorphisms}). 
Finally, in Theorem~\ref{capcays} we show that the symmetry group of $(G,d)$ can be put in terms of certain Cayley graphs associated to the partition. Namely, if $P_0,P_1,\ldots,P_s$ is the unitary symmetric partition of $G$ associated to $d$, then 
$\Gamma(G,s)$ is the intersection of all the Cayley graphs $Cay(G,P_i)$.

In Section \ref{S4}, using the results of Sections 2 and 3, we study the invariant and bi-invariant metrics in detail for small groups. For each group $G$ of order up to 7 we give all their unitary symmetric (conjugate) partitions $\mathcal{P}(G)$. For the equivalence class of metrics $d$ associated to each partition in $\mathcal{P}(G)$ we give the symmetry group $\Gamma(G,d)$ and the graphs of distances $\mathcal{G}(G,d)$. Also, we indicate which of these metrics is a poset metric \cite{Brualdi}, a chain metric or an extended Lee metric --both defined in \cite{PV}-- or an homogeneous metric on rings \cite{Gre}. 

In Section \ref{S5} we study the number of invariant metrics on a finite group $G$. 
In Proposition~\ref{CardMG} we give an expression for the number $M(G)$ of non-equivalent invariant metrics on $G$ in terms of Bell numbers $B_{k(G)}$ where $k(G)$ is certain %non-negative 
integer associated to $G$. This allows us to compute the number $M(G)$ in closed form for abelian groups $G$ in Proposition~\ref{kGab} and for $G$ being dihedral groups $\D_n$, dicyclic groups $\Q_{4n}$ and quasidihedral groups $Q\D_n^\pm$ in Proposition \ref{prop Ms}. In Proposition \ref{Sn An} we give an expression for the number of invariant metrics for symmetric and alternating groups $\Sym_n$, $\mathbb{A}_n$. We then study the number of invariant metrics for some semidirect products $G\rtimes_\varphi H$ in some particular cases (see \eqref{kGsxH bounds}--\eqref{DG}).

In Section \ref{S5b} we study the number of invariant metrics on a group $G$. 
In Proposition \ref{No bi-inv} we compute the number $M^*(G)$ of bi-invariant metrics of a group $G$ in terms of the number of its conjugacy and real conjugacy classes. In Proposition \ref{prop M*s} we give the precise number of bi-invariant metrics for dihedral, dicyclic and quasidihedral groups. The number of bi-invariant metrics of symmetric groups is considered in Example \ref{ambi3}. In Example~\ref{SL2q} we give the number of invariant and bi-invariant metrics of the special group 
$\mathrm{SL}_2(\ff_q)$ of $2 \times 2$ matrices over a finite field. 
Finally, in Theorem \ref{char bi} --one of the main results-- we characterize all groups $G$ (finite or not) in which every invariant metric is also a bi-invariant one. We prove that either $G$ is abelian or $G=\Q_8 \times H$ where $\Q_8$ is the quaternion group and $H$ is an abelian elementary $2$-group. In particular, if $G$ is finite it must be of the form $G=\Q_8 \times \Z_2^k$ for some $k \in \Z_{\ge 0}$.

Finally, in Section \ref{S6}, by using the results of Section \ref{S5} we count the number of invariant and bi-invariant metrics for  the 144 groups of order up to 32 (see Theorem \ref{allmetrics} and Tables \ref{tablita}--\ref{tablita5}).

\section{Invariant metrics and $\mathcal{P}$-equivalence}  \label{S2}
From now on we assume that $(G,d)$ is a metric group. 
The metric $d$ is called \textit{right translation invariant} or \textit{left translation invariant} if for any $g,g',h$ in $G$ we respectively have 
\begin{equation} \label{invariance}
d(gh,g'h)=d(g,g') \qquad \text{or} \qquad d(hg,hg')=d(g,g').
\end{equation}
If $G$ is abelian both notions coincide and $d$ is called \textit{translation invariant} (or \textit{bi-invariant}).   

Given a metric group $(G,d)$ we have the induced weight function $w : G\rightarrow \mathbb{R}_{\ge 0} $ defined by  
$w(x)=d(x,e)$ for any $x \in G$ where $e$ is the identity element of $G$. 
Conversely, if $(G,w)$ is a weight space, one can define a right (resp.\@ left) translation invariant metric $d$ on $G$ by 
\begin{equation} \label{inv w}
d(x,y) = w(xy^{-1})
\end{equation}
(resp.\@ $d(x,y) = w(y^{-1}x$) for every $x,y \in G$ provided that $w(x^{-1})=w(x)$ for every 
$x \in G$, which are denoted $d(x,y) = w(x-y)$ (resp.\@ $d(x,y) = w(y-x)$) and $w(-x)=w(x)$ if $G$ is abelian. 

It is equivalent to study right-invariant metrics or left-invariant metrics on non-abelian groups. To fix ideas and notations we will always assume that the metrics are right-invariant.

Let $(G,d)$ be a metric group and let $\Sym_{G}$ be the permutation group of $G$.  
Note that $G$ can be seen as a subgroup of $\Sym_G$ by identifying $g\in G$ with $\sigma_g \in \Sym_G$ where $\sigma_g(x) = xg$. We denote by 
$G_R$ the group $G$ seen as the group of right translations inside $\Sym_G$.
\begin{defi} \label{def sym grp}
We say that $(G,d)$ is \textit{$\sigma$-invariant} for $\sigma \in \Sym_G$, denoted $d^\sigma=d$, if 
	$$ d^\sigma(x,y) := d(\sigma (x),\sigma (y)) = d(x,y)$$ 
for all $x, y \in G$.
The \textit{symmetry group} of $(G,d)$ is defined by
$\Gamma(G,d) = \{\sigma \in \Sym_{G}  :  d^\sigma =d \}$.
\end{defi}

Clearly, to say that $(G,d)$ is \textit{$\sigma$-invariant} for every $\sigma \in G_R$ is equivalent to say that $d$ is right translation invariant.
From now on $(G,d)$ will denote a metric group where $d$ is a right translation invariant metric. 
We will use the following set
\begin{equation} \label{MG}
\begin{split}
\mathcal{M}(G) =\{ \text{right-invariant metrics on $G$}\}. 
\end{split}
\end{equation}
Note that $\mathcal{M}(G)$ is infinite; for example, given a metric $d$ on $G$ and $\alpha > 0$, we have that
$d_\alpha = \alpha d \in \MG$, where $d_\alpha$ is defined by 
$d_\alpha (x,y) = \alpha  d(x,y)$
for all $x,y \in G$. Similarly, given $d_1,d_2\in \MG $ then $d_1+d_2\in\MG$, where
$(d_1+d_2 )(x,y)=d_1(x,y)+d_2(x,y)$
for all $x,y \in G$. The addition of metrics is associative and hence $\mathcal{M}(G)$ is a semigroup. 
It is worth noticing that in general the product of metrics 
$(d_1 \cdot d_2 )(x,y) = d_1(x,y) d_2(x,y)$
is not a metric, although it is a semimetric, that is a function that have all the properties of a metric but where the triangle inequality does not necessarily holds.

\subsubsection*{Partitions and $\mathcal{P}$-equivalence of metrics}
Let $(G,d)$ be a metric group (recall that $d$ is right translation invariant) with associated weight function $w$. The \textit{induced partition} of $(G,d)$, denoted $\mathcal{P}=P(G,d)$, is the partition of $G$ determined by the equivalence relation
\begin{equation} \label{eq part}
	g \sim h \quad \Leftrightarrow \quad w(g)=w(h) \quad \Leftrightarrow \quad d(g,e)=d(h,e) 
\end{equation} 
for any $g,h \in G$.

In general we will write $P(G,d)=\{P_i\}_{i\in\mathcal{I}}$ for a partition of $G$, with $0\in \mathcal{I}$, and thus we have  
$P_i=\{g\in G : w(g)=w_i\}$
for $i\in \mathcal{I}$ and some sequence $\{w_i\}_{i\in \mathcal{I}}$ of different non-negative real numbers. We also write $w(P_i)=w_i$ for  $i\in \mathcal{I}$.	
If  the partition $P(G,d)$ is finite, in particular for finite groups, we will denote by $P_0, \ldots, P_s$ the parts of $P(G,d)$, that is $\mathcal{P}=\{P_0,\ldots,P_s\}$. Thus, if $w_1, \ldots, w_s$ are the different nonzero weights of $G$ (i.e.\@ the different real values that $w(g)$ can take for $g\in G$) we will say that $G$ is an 
\textit{$s$-weight} metric group. 

We will say that $\mathcal{P}$ is \textit{unitary} if $\{ e \} \in \mathcal{P}$, in which case we assume that $P_0=\{e\}$, and that it is \textit{symmetric} if $P_i=P_i^{-1}$, i.e.\@ $g \in P_i$ if and only if $g^{-1} \in P_i$, for all $i\in\mathcal{I}$. 
We will use the following set  
\begin{equation} \label{PG}
\Prep(G) =\{ \text{unitary symmetric partitions of $G$}\}.
\end{equation}

We now show the basic fact that any partition induced by a metric on a group is unitary and symmetric.
\begin{lema} \label{USP}
If $(G,d)$ is a metric space then $\mathcal{P}=P(G,d)$ is unitary and symmetric.
\end{lema}

\begin{proof}
	Clearly we have that $d(x,y)=0$ if and only if $x=y$, then $w(x)=0$ if and only if
	$x=e$, where $e$ is the identity of $G$. Thus, $\{e\} \in \mathcal{P}$, that is, $\mathcal{P}$ is unitary.
	Now, we have that 
	$w(x) = d(x,e) = d(e,x) = w(x^{-1})$ 
	for all $x\in G$ and hence $P$ is also symmetric.
\end{proof}

The lemma implies that we have a well defined map 
\begin{equation} \label{Particioninducida}
\pi_G  :  \mathcal{M}(G) 
\longrightarrow  \Prep(G), \qquad  d \: \mapsto \:  P(G,d). 
\end{equation}

\begin{defi} \label{Pequivalent}
We say that two metrics $d_1, d_2$ on $G$ are $\Prep$-\textit{equivalent}, and we denote it by $d_1 \sim_{\Prep} d_2$, 
if they induce the same unitary symmetric partition of $G$. In symbols, 
	\begin{equation} \label{Partitionrelation}
	\qquad d_1 \sim_{\Prep} d_2 \quad \Leftrightarrow \quad P(G,d_1) = P(G,d_2).
	\end{equation}
We will denote by $\widetilde{\mathcal{M}}(G) = \mathcal{M}(G)/_{\sim_{\Prep}}$
the \textit{space of $\Prep$-equivalence classes of metrics} on $G$.
\end{defi}

Obviously, a metric $d$ is $\Prep$-equivalent to any scalar multiple $d_\alpha=\alpha d$ of $d$ with
$\alpha \in \mathbb{R}_{>0}$. However, if $d_1\congP d_1^{\prime}$ and $d_2\congP d_2^{\prime}$ it is not true in general 
that $d_1+d_2\congP d_1^{\prime}+d_2^{\prime}$.
For example, consider the following two metrics on $G=\Z_4$, the Lee metric $d_{Lee}$ given by $w_{Lee}(0)=0$, $w_{Lee}(1)=w_{Lee}(3)=1$ and 
$w_{Lee}(2)=2$, and the metric $d$ induced by the weight function $w(0)=0$, $w(2)=1$ and $w(1)=w(3)=2$.
It is clear that $d_{Lee} \congP d$. 
On the other hand, 
$d_{Lee}+d_{Lee} = 2d_{Lee}\congP d_{Lee}$ and $d_{Lee}+d = 3d_{Ham}\congP d_{Ham}$. 
Since $d_{Lee}\not\congP d_{Ham}$ we have that $d_{Lee}+d_{Lee} \not\congP d_{Lee}+d$.
This shows that it could be not so easy to attach a group structure to the space of equivalent metrics of $G$. 
We pose the question: 
\textit{Which metric, topological or algebraic structure can be given to $\mathcal{M}(G)/_{\sim_{\Prep}}$?}

\goodbreak

The map $\pi_G$ is clearly not injective. We now show that $\pi_G$ in \eqref{Particioninducida} is surjective. 
\begin{lema} \label{surjectivity}
If $\mathcal{P}$ is a unitary symmetric partition of $G$ with $|\mathcal{P}|\leq |\mathbb{R}|$, then there exists a right-invariant metric $d\in\mathcal{M}(G)$ such that $P(G,d)=\mathcal{P}$. Moreover, if $|\mathcal{P}|$ is finite, then $d$ can be taken to be integral.
\end{lema}

\begin{proof}
Let $\mathcal{P}=\{P_i\}_{i\in\mathcal{I}}$ with $P_0=\{e\}$. Consider a sequence of real numbers $\{ a_i \}_{i\in \mathcal{I}} \subset [1,2]$ all mutually different, and define the following weight function on $G$:
\begin{equation} \label{ais}
w(x) = \begin{cases}
0,  & \quad \text{if } \ x=e, \\[.75mm]
a_i, & \quad \text{if } x\neq e,\ x\in P_i ,\; i\in\mathcal{I} \smallsetminus \{0\}. 
\end{cases}
\end{equation}
This defines the right-invariant metric $d$ given by
$d(x,y)=w(xy^{-1})$.
Since $\mathcal{P}$ is unitary, then $d(x,y)=0$ if and only if $x=y$,
and by symmetry of the partition we have that $d(x,y)=w(xy^{-1})=w(yx^{-1})=d(y,x)$.
Also $d(x,y)\geq 0$, for all $x,y\in G$, and $d$ satisfies the triangle inequality since
$ d(x,y)\leq 2\leq d(x,z)+d(z,y)$
for all $x,y, z\in G$ with $z\neq x,y$ (in the cases that $z=x$ or $z=y$ the inequality is trivial).

Finally, we have that $P_0(G,d)=\{ e \} = P_0$ and $P_i(G,d) = \{ x\in G : w(x)=a_i \}=P_i$ for any $i\in\mathcal{I},i\neq 0$, by construction.
Moreover, if $\mathcal{I}$ is finite, the numbers $a_i$ can be chosen to be rational, and hence multiplying them by a suitable integer one can obtain a new weight $w_\N$ function having the same partition as before but taking only integer values.
\end{proof}

We have the following useful relation between invariant metrics and symmetric partitions.
\begin{prop} \label{prop 1-1}
Let $(G,d)$ be a metric group. There is a 1--1 correspondence 
\begin{equation} \label{correspondencia}
 \faktor{\mathcal{M}(G)}{{\sim_\mathcal{P}}} \quad \longleftrightarrow \quad \Prep(G)
\end{equation}
between $\mathcal{P}$-classes of right-invariant metrics of $G$ and unitary symmetric partitions $P$ of $G$ with $|P| \le |\mathbb{R}|$.
\end{prop}

\begin{proof}
Consider the map $\bar \pi_G : \mathcal{M}(G)/_{\sim_\mathcal{P}} \longrightarrow \Prep(G)$ given by $\bar \pi_G([d])=P(G,d)$.
The map $\bar \pi_G$ is well defined and surjective by Lemmas~\ref{USP} and \ref{surjectivity}, and it is clearly injective by definition. 
\end{proof}

We note that there is no canonical way to associate a metric to a given unitary symmetric partition of $G$. 
For example in the finite case, given such a partition $\mathcal{P} = \{P_0,P_1,\ldots,P_n\}$, using the weight defined in the proof of Lemma~\ref{surjectivity} one can obtain a metric, although this depends on the choice of the numbers $a_i$.
For instance, we can take the numbers $a_i=1+\tfrac in$ for $i=1,\ldots,n$ and define the weight $w(x)$ on $G$ as in \eqref{ais}.
Multiplying by $n$ we get an integral weight, but with a gap between $w_0=0$ and the minimum non-zero weight $w_1$.

\subsubsection*{Bi-invariant metrics} A \textit{bi-invariant metric} is both a right-invariant and a left-invariant metric (see \eqref{invariance} and \eqref{inv w}). For abelian groups all these notions (R-invariance, L-invariance and bi-invariance) coincide. 

As one can imagine, bi-invariant metrics are closely related with conjugacy classes.
	Let $(G,d)$ be a metric group with weight $w$. If $d$ is bi-invariant then $w$ is constant on conjugacy classes. That is, for every $x,y \in G$ we have 
	$$w(x)=d(x,e)=d(yx,y)=d(yxy^{-1}, yy^{-1})=d(yxy^{-1},e)=w(yxy^{-1})$$
where in the second and third equalities we have used left and right invariance, respectively.	
Conversely, if $d$ is right (or left) invariant	and $w$ is constant on conjugacy classes then $d$ is bi-invariant.
In fact, for every $x,y,z \in G$ we have 
	$$ d(zx,zy)=w(zx(zy)^{-1})=w(zxy^{-1}z^{-1})=w(xy^{-1})=d(x,y).$$

The partition $P_0,P_1,\ldots,P_s$ associated to a bi-invariant metric of $G$ is unitary symmetric and, by the previous comments, also \textit{conjugate}, i.e.\@ 
$gP_i g^{-1}=P_i$ for all $g\in G$ and $i=0,\ldots,s$.
We point out that the classes of equivalence of bi-invariant metrics are in 1--1 correspondence with the unitary symmetric conjugate partitions $\Prep^c(G)$ of $G$, that is 
\begin{equation} \label{correspondencia2}
\faktor{\mathcal{M}(G)}{{\sim_{\mathcal{P}^c}}} \quad \longleftrightarrow \quad \Prep^c(G)
\end{equation}
where $\sim_{\mathcal{P}^c}$ denotes the equivalence relation on invariant metrics given by identifying those metrics having the same unitary symmetric conjugate partition of $G$.
In fact, if $G$ has a unitary symmetric and conjugate partition $\mathcal{P}'$, then there is a bi-invariant metric $d$ on $G$ such that $P(G,d)=\mathcal{P}'$. In fact, since $\mathcal{P}'$ is unitary and symmetric then, by Lemma \ref{surjectivity}, there is a right-invariant metric $d$ such that $P(G,d)=\mathcal{P}'$. Since $\mathcal{P}'$ is conjugate, then $d$ is also left-invariant.

\section{Symmetry groups and distance graphs} \label{S3}
In this section we give a relation between the symmetry group and the distance graph of a metric group $(G,d)$.

\subsection*{The symmetry group of a metric group}
From Definition \ref{def sym grp}, the symmetry group of a metric group $(G,d)$ is the set of all the permutations of $G$ preserving the metric $d$, that is
\begin{equation} \label{SGd} 
\Gamma(G,d) = \{ \sigma \in \Sym_G : d(\sigma(x),\sigma(y))=d(x,y)\}.
\end{equation}
We have the following equivalence relation in $\mathcal{M} (G)$,
$d_1\sim_{\Gamma} d_2$ if and only if $\Gamma (G,d_1)=\Gamma (G,d_2)$.
In this case, we will say that $d_1$ y $d_2$ are \textit{$\Gamma$-equivalent}.
Note that $\mathcal{P}$-equivalence of metrics implies $\Gamma$-equivalence of metrics.

To study these symmetry groups we will need the following generalization of dihedral groups.
\begin{defi} \label{Dihedral}
Let $G$ be an abelian group and $i: G\rightarrow G$ the inversion, given by $i(g)=g^{-1}$. 
We define the \textit{dihedral group} of $G$ by
$\Di(G)= G\rtimes \langle i \rangle$. 
\end{defi}

If $G$ is the cyclic group of $n$ elements, then $\Di (\Z_n) \simeq \Di_n$, the classic dihedral group, for $n\geq 3$ while 
$\Di(\Z_2)\simeq \Z_2$. However, $\Di(\Z_2^n) = \Z_2^n$. That is, $\Di(G)$ can be abelian or not. Also, 
$G \subseteq \Di(G)$ but the equality may happen. Moreover, if $G$ is abelian and $G\ne \Z_2^k$ then $\D(G)=G\rtimes \Z_2$.

By definition, $\Di(G) \subseteq \mathbb{S}_G$. However, notice that
$\Di (G) \leq \mathrm{Hol}(G) \simeq N_{\Sym_G} (G) \leq \Sym_G$, 
where $\mathrm{Hol}(G)= G \rtimes \mathrm{Aut}(G)$ 
is the holomorph of $G$ and 
$N_{\Sym_G}(G) = \{\sigma\in \Sym_G : \sigma G \sigma^{-1} = G\}$
is the normalizer of $G$ in $\Sym_{G}$. 
Since $i \in \mathrm{Aut}(G)$ we have that $\Di (G)$ is a subgroup of $\mathrm{Hol}(G)$, while the isomorphism 
$\mathrm{Hol}(G) \simeq N_{\Sym_G} (G)$  is obtained via the identification of $G$ with its right regular representation $G_R$. From now on, we will think of $\Di (G)$ as a subgroup of $\Sym_G$ where the inversion 
acts by permuting each element of $G$ with its inverse.

For abelian groups, the following is useful when trying to compute the symmetric group of a metric space $(G,d)$. 
\begin{prop} \label{Partsimetrica}
Let $(G,d)$ be an abelian metric group.
Then, the inversion automorphism on $G$ preserves distances, and hence we have that
\begin{equation} \label{contentions}
\mathbb{D}(G) \le \Gamma(G,d) \le \mathbb{S}_{G}.
\end{equation} 
\end{prop}

\begin{proof}
The metric $d$ is (right) traslation invariant, i.e. $G\leq \Gamma (G,d)$. For a group, the inversion is an automorphism if and only if the group is abelian. So, since $G$ is abelian, we only need to prove that it preserves distances. In fact, we have that
$$d(a,b)=w(ab^{-1})=w(b^{-1}a)=d(b^{-1},a^{-1})=d(i(b),i(a))=d(i(a),i(b)),$$
hence $\mathbb{D}(G) \leq~\Gamma(G,d)$, as we wanted to show.
\end{proof}

\goodbreak  
 
We now recall some classes of automorphisms. 
The group of inner automorphisms of $G$ is the one formed only by conjugations, that is $\operatorname{Inn}(G)=\{\varphi_g : g\in G\}$, where $\varphi_g(x) =gxg^{-1}$ for every $x\in G$. An automorphism $\sigma$ of a group $G$ is called a \textit{class-preserving} automorphism if for every $g\in G$, there exists an element $h\in G$ such that $\sigma(g) = hgh^{-1}$. These automorphisms form the subgroup denoted as $\operatorname{Aut}_{cp}(G)$. Clearly every 
inner automorphism of $G$ is class-preserving.
Similarly, an automorphism $\sigma$ of a group $G$ is called a \textit{class-inverting} automorphism if, for any $g \in G$, there exists an element $h \in G$ such that $\sigma(g) = hg^{-1}h^{-1}$.
The class-preserving and the class-inverting automorphisms form the subgroup of \textit{extended class-preserving} automorphism group, denoted by $\operatorname{Aut}_{ecp}(G)$. That is, the subgroup of all the automorphisms that sends every element either to an element in its conjugacy class or to an element in the conjugacy class of its inverse.
Summing up, we have  
\begin{equation} \label{automorphisms}
\operatorname{Inn}(G) \le \operatorname{Aut}_{cp}(G) \le \operatorname{Aut}_{ecp}(G) \le \operatorname{Aut}(G).
\end{equation}

For general metric groups we can give a similar result as the previous proposition, provided that the metric involved is bi-invariant.
\begin{prop} \label{Gbiinv}
If $(G,d)$ is a bi-invariant metric group then 
\begin{equation} \label{contentions aut}
	G \rtimes \operatorname{Aut}_{ecp}(G) \leq~\Gamma(G,d) \le \mathbb{S}_{G}.
\end{equation} 
In particular if $G$ is abelian then $G\rtimes \operatorname{Aut}_{ecp}(G)=\D(G)$.
\end{prop}

\begin{proof}
The metric $d$ is right translation invariant, i.e.\@ $G\leq \Gamma (G,d)$. Let $\sigma\in \operatorname{Aut}(G)$. If $\sigma$ is class-preserving then 
$$d(a,b)=w(ab^{-1})=w(\sigma(ab^{-1}))=w(\sigma(a)\sigma(b)^{-1})=d(\sigma(a),\sigma(b)),$$
while if $\sigma\in \operatorname{Aut}(G)$ is class-inverting then
$$d(a,b)=w(ab^{-1})=w(ba^{-1})=w(\sigma(ab^{-1}))=w(\sigma(a)\sigma(b)^{-1})=d(\sigma(a),\sigma(b)),$$
hence $\operatorname{Aut}_{ecp}(G) \leq~\Gamma(G,d)$, which implies $G \rtimes \operatorname{Aut}_{ecp}(G) \leq~\Gamma(G,d)$ as we wanted to show. 

If $G$ is abelian, the only class-preserving automorphism is the identity and the only class-inverting automorphism is the inversion. In this way, we have that $\operatorname{Aut}_{ecp}(G) = \langle i \rangle$, and hence $G\rtimes \operatorname{Aut}_{ecp}(G) = \D(G)$, by Definition \ref{Dihedral}, as we wanted to show.
\end{proof}

It is known that 
\begin{equation} \label{Inn}
\operatorname{Inn}(G) \simeq G/Z(G) 
\end{equation}
for any group $G$. Hence, we have that 
$\operatorname{Inn}(\Sym_n)=\Sym_n$ for $n \ge 3$, $\operatorname{Inn}(\D_n)=\D_n$ for $n$ odd and $\operatorname{Inn}(\D_n)=\D_n/\Z_2$ for $n$ even and also $\operatorname{Inn}(\Q_{4n})=\Q_{4n}/\Z_2$ for any $n\ge 2$.

\begin{exam} \label{GDQ}
Let $G = \D_{2n}$ or $G=\Q_{4n}$ for any $n\ge 2$.
By \eqref{automorphisms}, \eqref{contentions aut} and \eqref{Inn} we have
$G \rtimes G/\Z_2 \le \Gamma(G,d)$
where $d$ is any bi-invariant metric on $G$.
Thus, $8n^2 \le |\Gamma(G,d)|$. If also the partition of $d$ has a part of the form $P_i =\{g_i,g_i^{-1}\}$ with $g_i$ of order $2n$ (i.e.\@ the graph of distances --see below-- contains $2$ disjoint cycles $C_{2n}$), as for instance in the case of the Lee metric, then 
$\Gamma(G,d) \le \D_{2n} \wr \Z_2$
(see Corollary \ref{Cay(Cn)}). Thus, $2^3 n^2 \le |\Gamma(G,d)| \le 2^5 n^2$, 
in this case.
\hfill $\lozenge$
\end{exam}

\subsection*{Distance and partition graphs} 
Given a finite metric group $(G,d)$ we associate to it a weighted graph which we denote by $\mathcal{G}{(G,d)}$ and call it the \textit{distance graph} of 
$(G,d)$.
It is the complete graph of size $|G|$ with the weights on the edges given by the distance $d$. 
That is, $\mathcal{G}{(G,d)} = (K_n, \omega_d)$ where $|G|=n$ and $\omega_d(\{x,y\})=d(x,y)$ for every $x,y \in G$. 
In the same manner we can associate to any unitary symmetric partition $\mathcal{P}=\{P_i\}_{i\in\mathcal{I}}$ of $G$, a \textit{partition graph}  $\mathcal{G}{(G,\Prep)} = (K_n, \omega_{\Prep})$ where $\omega_{\Prep}(\{x,y\})=i$ if $xy^{-1}\in P_i$ for every $x,y \in G$.

Next we will relate the symmetry group of $(G,d)$ with the automorphism group of the distance graph of $(G,d)$, which is the same automorphism of the graph $\mathcal{G}{(G,\Prep)}$ . 
We need to recall Cayley graphs. Given a group $G$ and a subset $S$ of $G$ one has a Cayley graph $Cay(G,S)$ where $G$ is the vertex set and two elements $g,h$ form a (directed) edge from $g$ to $h$ if and only if $g-h \in S$. If $G$ is abelian we consider the Cayley graph undirected. If $e \notin S$ the graph $Cay(G,S)$ is loopless and if $S$ is symmetric ($S=S^{-1}$) then it is a simple (undirected with no multiple edges) graph.

\begin{thm} \label{capcays}
Let $(G,d)$ be a metric group and let $\mathcal{P}=\{P_0,P_1,\ldots,P_s\}$ be its unitary symmetric partition. 
Then, we have
\begin{equation} \label{G Aut}
\Gamma(G,d) = \bigcap_{1\le i \le s} \mathrm{Aut} \big( Cay(G,P_i) \big).
\end{equation}
\end{thm}

\begin{proof}
The automorphism group $\mathrm{Aut}(\mathcal{G}(G,d))$ of the distance graph $(G,d)$ is the group of all 
edge-preserving bijections 
$\sigma$ of $G$ also preserving the weights of the graph, i.e.\@ if $xy$ is an edge then 
$\omega_d(\sigma(x)\sigma(y)) = \omega_d(xy)$. 
In this way, we have that 
\begin{equation} \label{Aut Gd}
   \Gamma(G,d) = \mathrm{Aut}(\mathcal{G}(G,d)).
\end{equation} 

Now, as mentioned $\mathrm{Aut}(\mathcal{G}(G,d))= \mathrm{Aut}(\mathcal{G}(G,\Prep))$ and note that each part $P_i$ of the partition induces a simple Cayley graph $Cay(G,P_i)$. In fact, if $0<w_1< \cdots < w_s$ are the different nonzero weights of $d$ and 
$P_i=\{g\in G : w(g)=w_i\}$ for $1\le i \le s$, then 
$$\{gh\} \in E \quad \Leftrightarrow \quad g-h\in P_i \quad \Leftrightarrow \quad d(g,h)=w_i.$$

We can consider the graph $Cay(G,P_i)$ naturally weighted by the distances of elements in $G$. 
In this way, the distance graph of $(G,d)$ is the union of simple Cayley graphs, that is 
\begin{equation} \label{cayleys}
	\mathcal{G}(G,d) =  \bigcup_{1\le i \le s} Cay(G,P_i). 
\end{equation} 
We recall that the union of graphs $G_i=(V_i,E_i)$, $i=1,2$, is the graph $G=(V_1 \cup V_2, E_1 \cup E_2)$.
Then, by \eqref{cayleys} we have that $\mathrm{Aut} \big( \mathcal{G}(G,d) \big) = \mathrm{Aut} (Cay(G,P_1) \cup \cdots \cup Cay(G,P_s))$ and we get 
\begin{equation} \label{Aut}
\mathrm{Aut} \big( \mathcal{G}(G,d) \big) = \bigcap_{1\le i \le s} \mathrm{Aut} \big( Cay(G,P_i) \big).
\end{equation}
In fact, by \eqref{cayleys}, it is enough to show that $\mathrm{Aut}(G_1 \cup G_2) = \mathrm{Aut}(G_1) \cap \mathrm{Aut}(G_2)$ for any pair of graphs 
$G_1, G_2$ with the same nodes. If $\sigma\in\mathrm{Aut}(G_1 \cup G_2)$, then $\sigma$ is a permutation that preserves the weights of 
$G_1 \cup G_2$. That is, it preserves the edges of $G_1$ and $G_2$, so the restriction to both of the graphs $G_1, G_2$ is an automorphism of the corresponding graphs. On the other hand, a permutation that preserves the edges of $G_1$ and $G_2$ preserves the weights of $G_1\cup G_2$.
The result is automatic from \eqref{Aut Gd} and \eqref{Aut}.
\end{proof}

%Since many of the symmetry groups that will appear in the following are wreath products of groups we now 
%recall its definition.
%Let $(G,X)$ and $(H,Y)$ be permutation groups. The product 
%$G^{Y}=\{f:Y \rightarrow G\}$ acts over $X\times Y$ in the first coordinate in the following way
%$f(x,y)=(f(y)x,y)$ for all $x\in X$, $y\in Y$, $f\in G^Y$. On the other hand, $H$ acts on $X\times Y$ in the second coordinate by 
%$h(x,y)=(x,h(y))$ for all  $x\in X$, $y\in Y$, $h\in H$. Furthermore, $H$ acts on $G^Y$ by permuting the coordinates: 
%$(hf)(y)=f(h^{-1}(y))$.
%Thus, the semidirect product $G^{Y} \rtimes H$ acts on $X\times Y$.
%This permutation group is called \textit{wreath product} of $(G,X)$ and $(H,Y)$ and it is denoted by 
%$G\wr H$.
%In the paper, $G$ and $H$ will act on themselves ($X=G$ and $Y=H$). The smallest non trivial wreath product is 
%$\Z_2 \wr \Z_2 \simeq \mathbb{D}_4$.

As usual, we denote by $C_n$ the cycle graph of $n$ vertices ($n$-cycle) and by $mC_n$ the disjoint union of $m$ copies of $C_n$. 
We will also use the convention $C_2=K_2$.
Given two groups $G,H$, if $H$ acts on $\{1,2,\ldots,n\}$ then the semidirect product $G^{n} \rtimes H$ is called the wreath product of $G$ and $H$ and denoted by $G\wr H$.
The smallest non trivial wreath product is $\Z_2 \wr \Z_2 \simeq \mathbb{D}_4$.
We have the following consequence of Proposition \ref{Partsimetrica} and Theorem \ref{capcays}.

\begin{coro} \label{Cay(Cn)}
Consider a metric group $(G,d)$ of order $n$ with associated unitary symmetric partition $\mathcal{P}=\{P_0,P_1,\ldots,P_s\}$. 
If %for some $m \mid n$ we have that 
$Cay((G,P_i))= m \, C_{\frac nm}$ 
for some $m \mid n$ and $i=1,\ldots,s,$ then 
\begin{equation} \label{AutGPi}
\mathrm{Aut}(Cay(G,P_i))=\D_{\frac nm} \wr \Sym_m
\end{equation}
and hence $\Gamma(G,d) \le \D_{\frac nm} \wr \Sym_m$.
In particular, if $m=1$ or $m=n$ then $\Gamma(G,d)=\D_n$ or $\Gamma(G,d)=\Sym_n$, respectively. 
Thus, $\Gamma (G,d_{Ham}) \simeq \Sym_n$ for any $G$ and $\Gamma (\Z_n,d_{Lee}) \simeq \D_n$ for $n\in \N$. %, for example $\Gamma(\Z_n,d_{Lee})=\D_n$. 
\end{coro}

\begin{proof}
 By \eqref{contentions} and \eqref{G Aut}, we have that $\D_n \le \Gamma(G,d) \le \mathrm{Aut}(Cay(G,P_i))$ for every $i=1,\ldots,s$. Also, if $Cay(G,P_i)=C_n$ for some $i$ then $\mathrm{Aut}(Cay(G,P_i))=\D_n$. This immediately implies that $\Gamma(G,d)=\D_n$.	Now, if $Cay((G,P_i))= m \, C_{\frac nm}$ for some $i$ then we have \eqref{AutGPi} 
since we can permute the $m$ different $\frac nm$-cycles to each other and in each $\frac nm$-cycle we get the dihedral group $\D_{\frac nm}$ as permutation group.    
\end{proof}

\section{Invariant metrics for small groups} \label{S4}
We now study the invariant and bi-invariant metrics and the associated partitions, symmetry groups and distance graphs of the smallest groups of order up to $7$. 
More precisely, we find all the $\mathcal{P}$-classes of metrics of 
$G \in \{ \{0\}, \Z_2,\Z_3, \Z_4, \Z_2 \times \Z_2, \Z_5, \Z_6, \Sym_3, \Z_7\}$
by giving their unitary symmetric (conjugate) partitions and distance graphs. With the aid of these graphs, we will also compute the associated symmetry groups.
Furthermore, we will indicate which of the metrics are poset metrics \cite{Brualdi}, homogeneous metrics \cite{Gre}, chain metrics or extended Lee metrics \cite{PV}. 

All the information on the metrics will be given in Tables \ref{tabla Z4}--\ref{tabla Z7}. There, we give the weight functions $w_i$ associated to the metrics $d_i$ (omitting the trivial weights $w_i(0)=0$). In the column labeled $s$ we count the number of non-trivial different weights of each given metric and in the second to last column we indicate some known metrics which fall in the same equivalence class of the given metric. In each row of the table, the weights $w_i$ correspond to the first metric listed in this penultimate column. In the last column we give the symmetry groups (as abstract groups).

\subsubsection*{Some general families of invariant metrics}
We now recall the definitions of some particular classes of invariant metrics: 
 
($a$) \textit{Hamming metric}: Let $G = G_1\times \cdots \times G_n$ with $G_1,\ldots,G_n$ groups. The Hamming metric $d_{Ham}^n$ on $G$ is determined by the Hamming weight given by 
\begin{equation*} \label{ham weight}
w_{Ham}^n (x) = \#\{1 \le i \le n : x_i\ne e_{G_i}\}
\end{equation*}
where 
$x=(x_1,\ldots, x_n) \in G$ and $e_{G_i}$ denotes the identity element in $G_i$.
For $n=1$ we simply write $d_{Ham}$ and $w_{Ham}$.

($b$) \textit{Poset metric} \cite{Brualdi}: Let $G = G_1\times \cdots \times G_n$ with $G_1,\ldots,G_n$ groups,
 and consider the poset $P = ([n],\preceq)$ where $\preceq$ is a partial order on $[n]=\{1,2,\ldots,n\}$. 
 A subset $I\subset [n]$ is an ideal if given $a \in [n]$ and $b \in I$ with $a\preceq b$ then $a\in I$. 
 If $A \subset [n]$ is a subset, then $\left\langle A \right\rangle$ denotes the minimal ideal containing $A$. The weight function on $G$ associated to the poset $P$ is given by
\begin{equation} \label{poset weight}
w_P(g) = |\left\langle supp(g)\right\rangle| \qquad \text{where} \qquad supp(g) = \{i \in [n] : g_i \neq e\}
\end{equation}
denotes the support of $g = (g_1,\ldots, g_n)\in G$. This induces the poset metric $d_P$ on $G$.  
For the $n$-antichain (trivial poset of $[n]$) we get the classic Hamming metric $d_{Ham}^n$ on $n$ coordinates.

($c$) \textit{Chain metric} \cite{PV}: Let $G$ be a group and $\C$ a chain  
$	\langle e \rangle = H_0 \subsetneq H_1 \subsetneq \cdots\subsetneq H_n = G$ of subgroups of $G$.
The chain metric $d_\mathcal{C}$ on $G$ associated to $\mathcal{C}$ is defined by 
\begin{equation*} \label{chain metric}
d_\C(x,y) = i \quad  \Leftrightarrow \quad x-y \in H_i \smallsetminus H_{i-1}
\end{equation*}
for $i=0,\ldots, n$. Here we use the convention $H_{-1}=\varnothing$.
If $\mathcal{C}$ is simply $\{e\} \subsetneq H \subsetneq G$ we denote the chain metric by $d_{H}$. In this case we have
$w_H(e)=0$ and 
\begin{equation} \label{2-w chain}
w_H(x) = \begin{cases}
1 & \qquad \text{if $x \in H\smallsetminus \{e\}$}, \\
2 & \qquad \text{if $x \in G\smallsetminus H$}.
\end{cases}
\end{equation}
Also, for the trivial chain $\{e\} \subset G$ we get the Hamming metric $d_{Ham}$ on $G$.

($d$) \textit{Lee and extended Lee metrics} \cite{PV}: 
The well-known Lee metric $d_{Lee}$ on $\Z_m$ is given by 
$$w_{Lee} (x) =\min \{x,m-x\}.$$
Now, for each $n \mid m$, there is an extended Lee metric $d_{n,Lee}$ on $\Z_m$ with associated weight \cite{PV}:
\begin{equation} \label{ext lee met} 
w_{n, Lee}(x)=
\begin{cases}
x  & \qquad \text{if }  x \leq n, \\
n  & \qquad \text{if }  n \leq x \leq m-n, \\
n-x & \qquad \text{if }  m-n \leq x \leq m-1.\\
\end{cases}
\end{equation}
For even $m$ and $n=\frac m2$ we have $d_{n,Lee}=d_{Lee}$. 

\goodbreak 

($e$) \textit{Homogeneous metric}: 
There is a classic definition of a metric given by the homogeneous weight on the ring $\Z_m$ \cite{Hom}. Greferath and Schmidt \cite{Gre} extended this to any finite ring $R$ with identity $1\ne 0$ (see Definition~1.2). They showed that a weight function $w$ on $R$ (in a weaker sense) is an homogeneous weight on $R$ if and only if there is a constant $\lambda \ge 0$ such that (see Theorem 1.3) 
\begin{equation}  \label{hom weight}
w(x) = \lambda \big\{ 1- \tfrac{\mu (0,Rx)}{|R^\times x|} \big\}
\end{equation}
where $R^\times x=\{y\in R : Ry=Rx\}$ and $\mu$ is the Möbius function on the poset of its principal left ideals $\{Rx : x\in R\}$  ordered by inclusion, that is $\mu(Rx, Rx)=1$ and 
$$\mu(Rx, Rz) = − \sum_{Rx≤Ry<Rz} \mu(Rx,Ry).$$
This weight does not define a metric in general. When this is the case we denote the metric by $d_{hom}$. 
Given an abelian group $G$ one can consider all different ring structures on $G$ and their corresponding homogeneous weights. 
It is known that for the group $\Z_n$ of integers modulo $n$ there is one ring structure for each $k\mid n$ determined by 
$1\cdot 1 = k$ (see \cite{Fle}). However, only $k=1$ gives rise to a ring with identity, which is necessarily the ring of integers modulo $n$.
We denote this metric by $d_{hom}$.

\subsubsection*{Invariant metrics for small groups}
For the trivial group $\{0\}$ anything is trivial. Namely, $\mathcal{P}_1(\{0\}) = \{ \{0\} \}$, hence the metric is the Hamming one and $\Gamma(\{0\},d_{Ham})=\Sym_1=\{0\}$. 
Next, we consider the smallest cyclic groups of prime order $\Z_2$, $\Z_3$ and $\Z_5$, which are also trivial. 
\begin{exam}
For the groups $\Z_2$ and $\Z_3$ there is only one unitary symmetric partition:
$\mathcal{P}_1(\Z_2) = \{ \{0\}, \{1\}\}$ and $\mathcal{P}_1(\Z_3) = \{ \{0\}, \{1,2\}\}$. In both cases the metric is the Hamming metric and the symmetry groups are $\Gamma(\Z_2,d_{Ham})=\Sym_2$ and $\Gamma(\Z_3,d_{Ham})=\Sym_3$. 
For $\Z_5$ there are just two unitary symmetric partitions:
$\mathcal{P}_1(\Z_5) = \{ \{0\}, \{1,2,3,4\}\}$ and $\mathcal{P}_2(\Z_5) = \{ \{0\}, \{1,4\}, \{2,3\}\}$. 
The associated metrics are respectively the Hamming and the Lee metrics and hence the corresponding symmetry groups are $\Gamma(\Z_5,d_{Ham})=\Sym_5$ and $\Gamma(\Z_5,d_{Lee})=\D_5$. 
\hfill $\lozenge$
\end{exam}

In the next two examples we determine all invariant metrics for a group $G$ of order $4$. 

\begin{exam}%[$G=\Z_4$]
Consider the group $\Z_4=\{0,1,2,3\}$. 
It is clear that $\Z_4$ has only the two unitary symmetric partitions:
$\mathcal{P}_1 = \{ \{ 0 \}, \{ 1,2,3 \} \}$ and $\mathcal{P}_2 = \{ \{ 0 \}, \{ 1,3 \}, \{ 2 \} \}$,
which correspond with the Hamming and the Lee metrics ($d_1=d_{Ham}$, $d_2=d_{Lee}$).
The corresponding partition graphs and symmetry groups are given in Figures \ref{fig1} and \ref{fig2}.
	\begin{figure}[H]
		\centering
		\begin{minipage}[t][][t]{0.49\textwidth}
	
	\begin{tikzpicture}[scale=0.6,transform shape,>=stealth',shorten
	>=1pt,auto,node distance=3cm, thick,main node/.style={circle,draw,font=\small}]
	\tikzstyle{every node}=[node distance = 3.25cm,bend angle    = 45,fill          =
	gray!30]
	
	\node[main node] (1) {0};
	\node[main node] (2) [above left of=1] {3};
	\node[main node] (3) [above right of=2] {2} ;
	\node[main node] (4) [above right of=1] {1};
	
	\path[every node/.style={font=\large,circle}]
	
	(1) edge node[right] {1} (4)
	edge node [right, pos=0.8] {1} (3)
	(2) edge node [left] {1} (1)
	edge node [below, pos=0.3] {1} (4)
	(3) edge node [left] {1} (2)
	(4) edge node [right] {1} (3);
	\end{tikzpicture}
	\centering
	\caption{$\mathcal{G}_1 = \mathcal{G}(\Z_4,d_{Ham})$}
	\label{fig1}
\end{minipage}
\hfill
		\begin{minipage}[t][][t]{0.49\textwidth}
			\begin{tikzpicture}[scale=0.6,transform shape,>=stealth',shorten
			>=1pt,auto,node distance=3cm, thick,main node/.style={circle,draw,font=\small}]
			\tikzstyle{every node}=[node distance = 3.25cm,bend angle=45,fill=gray!30]
			
			\node[main node] (1) {0};
			\node[main node] (2) [above left of=1] {3};
			\node[main node] (3) [above right of=2] {2} ;
			\node[main node] (4) [above right of=1] {1};         
			
			\path[every node/.style={font=\large,circle}]
			
			(1) edge node[right] {1} (4)
			edge[blue] node [right, pos=0.8] {2} (3)
			(2) edge node [left] {1} (1)
			edge[blue] node [below, pos=0.3] {2} (4)  
			(3) edge node [left] {1} (2) 
			(4) edge node [right] {1} (3);
			\end{tikzpicture}
			\centering
			\caption{$\mathcal{G}_2 = \mathcal{G}(\Z_4,d_{Lee})$} %, \\ $\Gamma (\Z_4,d_{Lee})\simeq \mathbb{D}_4$} 
			\label{fig2}
		\end{minipage}
\end{figure}

Although we know by Corollary \ref{Cay(Cn)} that the symmetry group for $(\Z_4,d_{Lee})$ is $\D_4$, we explain  
how it can be obtained using the graphs. By \eqref{G Aut} we have 
\begin{equation*} \label{capZ4} 
\Gamma(\Z_4,d_{Lee}) = \mathrm{Aut} \big(Cay(\Z_4,\{1,3\}) \big) \cap  \mathrm{Aut}\big(Cay(\Z_4,\{2\})\big).
\end{equation*}
Also, 
$\mathrm{Aut} (Cay(\Z_4,\{1,3\}))= \mathrm{Aut}(C_4) \simeq \D_4$ and 
$\mathrm{Aut}(Cay(\Z_4,\{2\})) = \mathrm{Aut}(2K_2) \simeq \Z_2 \wr \Z_2 \simeq \D_4$, 
by Corollary \ref{Cay(Cn)}.
The two automorphism groups of $Cay(\Z_4,\{1,3\})$ and $Cay(\Z_4,\{2\})$ are not only isomorphic, they are actually equal by \eqref{contentions}. Hence, we finally get $\Gamma(\Z_4,d_{Lee})=\D_4$. 

From the graphs, we see that we have the following weight functions on $\Z_4$:
\renewcommand{\arraystretch}{1}
\begin{table}[H] 
	\caption{Invariant metrics for $\Z_4$} \label{tabla Z4}
	\begin{tabular}{|c|ccc|c|c|c|c|c|} 
		\hline
		$\Z_4$ & $1$ & $2$ & $3$ & $s$ & chain & hom.\@ & metrics 	& $\Gamma$ \\ \hline
		$w_1$ &  $1$ & $1$ & $1$ & 1   & \ding{51} & \ding{55} & $d_{Ham}$  & $\Sym_4$\\ 
		$w_2$ &  $1$ & $2$ & $1$ & 2   & \ding{51} & \ding{51} & $d_{Lee}=d_{2,Lee}$, $d_{\Z_2}$, $d_{hom}$ &  $\D_4$  \\ 
		\hline
	\end{tabular}
\end{table}
We explain the metrics. Note that by \eqref{2-w chain} the chain metric $d_{\Z_2}$ given by
$\{0\} \subset \Z_2 \subset \Z_4$ is equivalent to $d_{Lee}$ but $d_{Lee} \ne d_{\Z_2}$.   
Also, by \eqref{ext lee met} we have $d_{2,Lee}=d_{Lee}$. 
The ring of integers modulo 4 has poset of principal ideals given by $\{0\} \preceq 2\Z_4 \preceq \Z_4$ where $\mu(\{0\},\{0\})=1$, $\mu(\{0\},2\Z_4)=-1$ and $\mu(\{0\},\Z_4)=0$. 
Now, by \eqref{hom weight}, we have 
$w_{hom}(1)=\lambda$, $w_{hom}(2)=2\lambda$, $w_{hom}(3)=\lambda$.
Finally, note that for $\lambda=1$ we have $d_{Lee}= d_{hom}$. 
\hfill $\lozenge$ 
\end{exam}

\begin{exam}
Consider the group $\Z_2 \times \Z_2$. It is straightforward to check that 
$\mathbb{Z}_2\times \mathbb{Z}_2$ has the following five unitary symmetric partitions: 
\begin{gather*}
\mathcal{P}_1 = \big\{\{ (0,0) \}, \{ (0,1), (1,0) , (1,1) \} \big\}, \qquad 
\mathcal{P}_2 = \big\{ \{ (0,0) \}, \{ (1,0), (0,1) \}, \{ (1,1) \} \big\}, \\
\mathcal{P}_3 = \big\{ \{ (0,0) \}, \{ (0,1), (1,1) \}, \{ (1,0) \} \big\}, \qquad  
\mathcal{P}_4 = \big\{ \{ (0,0) \}, \{ (1,0), (1,1) \}, \{ (0,1) \} \big\}, \\ 
\mathcal{P}_5 = \big\{ \{(0,0) \}, \{ (1,0) \}, \{ (0,1) \}, \{ (1,1) \} \big\}, \qquad 
\end{gather*}
Their associated partition graphs are respectively given by Figures \ref{1}--\ref{5}:
\begin{figure}[h!]
\centering
\begin{minipage}[t][][t]{0.325\textwidth}
\begin{tikzpicture}[scale=0.575,transform shape,>=stealth',shorten
>=1pt,auto,node distance=3cm, thick,main node/.style={circle,draw,font=\Large}]
\tikzstyle{every node}=[node distance = 3.25cm, bend angle=45,fill=gray!30]
  \node[main node,font=\fontsize{20}{144}] (1) {(0,0)};
  \node[main node,font=\fontsize{20}{144}] (2) [above left of=1] {(1,0)};
  \node[main node,font=\fontsize{20}{144}] (3) [above right of=2] {(1,1)} ;
  \node[main node,font=\fontsize{20}{144}] (4) [above right of=1] {(0,1)};
    
  \path[every node/.style={font=\large,circle}]
    (1) edge node[right] {1} (4)
  edge node [right, pos=0.8] {1} (3)
  (2) edge node [left] {1} (1)
  edge node [below, pos=0.3] {1} (4)
  (3) edge node [left] {1} (2)
  (4) edge node [right] {1} (3);
\end{tikzpicture}
\centering
\captionsetup{justification=centering,margin=0.1cm}
\caption{$\mathcal{G}_1=\mathcal{G}(\Z_2^2,d_1)$} 
\label{1}
\end{minipage}
\hfill
\begin{minipage}[t][][t]{0.33\textwidth}%
\begin{tikzpicture}[scale=0.575,transform shape,>=stealth',shorten
>=1pt,auto,node distance=3cm, thick,main node/.style={circle,draw,font=\small}]
\tikzstyle{every node}=[node distance = 3.25cm,bend angle    = 45,fill          =
gray!30]

  \node[main node,font=\fontsize{20}{144}] (1) {(0,0)};
  \node[main node,font=\fontsize{20}{144}] (2) [above left of=1] {(1,0)};
  \node[main node,font=\fontsize{20}{144}] (3) [above right of=2] {(1,1)} ;
  \node[main node,font=\fontsize{20}{144}] (4) [above right of=1] {(0,1)};
       
  \path[every node/.style={font=\large,circle}]
 
    (1) edge node[right] {1} (4)
       edge[blue] node [right, pos=0.8] {2} (3)
    (2) edge node [left] {1} (1)
        edge[blue] node [below, pos=0.3] {2} (4)
    (3) edge node [left] {1} (2)
    (4) edge node [right] {1} (3);
\end{tikzpicture}
\centering
\captionsetup{justification=centering,margin=0.1cm}
\caption{$\mathcal{G}_2=\mathcal{G}(\Z_2^2,d_2)$}
\label{2}
\end{minipage}
\hfill
\begin{minipage}[t][][t]{0.33\textwidth}
\begin{tikzpicture}[scale=0.575,transform shape,>=stealth',shorten
>=1pt,auto,node distance=3cm, thick,main node/.style={circle,draw,font=\small}]
\tikzstyle{every node}=[node distance = 3.25cm,bend angle    = 45,fill          =
gray!30]

  \node[main node,font=\fontsize{20}{144}] (1) {(0,0)};
  \node[main node,font=\fontsize{20}{144}] (2) [above left of=1] {(1,0)};
  \node[main node,font=\fontsize{20}{144}] (3) [above right of=2] {(1,1)} ;
  \node[main node,font=\fontsize{20}{144}] (4) [above right of=1] {(0,1)};
  \path[every node/.style={font=\large,circle}]
  
    (1) edge node[right] {1} (4)
       edge node [right, pos=0.8] {1} (3)
    (2) edge[blue] node [left] {2} (1)
        edge node [below, pos=0.3] {1} (4)
    (3) edge node [left] {1} (2)
    (4) edge[blue]  node [right] {2} (3);
\end{tikzpicture}
\centering
\captionsetup{justification=centering,margin=0.1cm}
\caption{$\mathcal{G}_3=\mathcal{G}(\Z_2^2,d_3)$}
\label{3}
\end{minipage}%
\medskip

\hfill
\begin{minipage}[c]{0.33\textwidth}
\begin{tikzpicture}[scale=0.575,transform shape,>=stealth',shorten
>=1pt,auto,node distance=3cm, thick,main node/.style={circle,draw,font=\small}]
\tikzstyle{every node}=[node distance = 3.25cm,bend angle= 45, fill=gray!30]

  \node[main node,font=\fontsize{20}{144}] (1) {(0,0)};
  \node[main node,font=\fontsize{20}{144}] (2) [above left of=1] {(1,0)};
  \node[main node,font=\fontsize{20}{144}] (3) [above right of=2] {(1,1)} ;
  \node[main node,font=\fontsize{20}{144}] (4) [above right of=1] {(0,1)};
          
  \path[every node/.style={font=\large,circle}]

    (1) edge[blue] node[right] {2} (4)
       edge node [right, pos=0.8] {1} (3)
    (2) edge node [left] {1} (1)
        edge node [below, pos=0.3] {1} (4)
    (3) edge[blue] node [left] {2} (2)
    (4) edge node [right] {1} (3);
\end{tikzpicture}
\centering
\captionsetup{justification=centering,margin=0.1cm}
\caption{$\mathcal{G}_4 = \mathcal{G}(\Z_2^2,d_4)$}
\label{4}
\end{minipage}
\hfill
\begin{minipage}[c]{0.33\textwidth}
\begin{tikzpicture}[scale=0.575,transform shape,>=stealth',shorten
>=1pt,auto,node distance=3cm, thick,main node/.style={circle,draw,font=\small}]
\tikzstyle{every node}=[node distance = 3.25cm, bend angle = 45,fill =gray!30]

  \node[main node,font=\fontsize{20}{144}] (1) {(0,0)};
  \node[main node,font=\fontsize{20}{144}] (2) [above left of=1] {(1,0)};
  \node[main node,font=\fontsize{20}{144}] (3) [above right of=2] {(1,1)} ;
  \node[main node,font=\fontsize{20}{144}] (4) [above right of=1] {(0,1)};
  \path[every node/.style={font=\large,circle}]
  
  (1) edge [red] node[right] {3} (4)
  edge[blue] node [right, pos=0.8] {2} (3)
  (2) edge node [left] {1} (1)
  edge[blue] node [below, pos=0.3] {2} (4)
  (3) edge [red] node [left] {3} (2) 
  (4) edge  node [right] {1} (3);
\end{tikzpicture}
\centering
\captionsetup{justification=centering,margin=0.1cm}
\caption{$\mathcal{G}_5 = \mathcal{G}(\Z_2^2,d_5)$} 
\label{5}
\end{minipage}
\hfill
\hfill
\end{figure}

We first compute the symmetry groups. It is obvious that $\Gamma(\Z_2^2,d_1) \simeq\Sym_4$. Also, proceeding as in the previous example, it is clear that the symmetry groups of $(\Z_2^2,d_2)$, $(\Z_2^2,d_3)$ and $(\Z_2^2,d_4)$ are all isomorphic to $\D_4$.
However, these groups are not equal. In fact, using the cycle notation of symmetric groups and the notations 
$a=(0,0)$, $b=(1,0)$, $c=(1,1)$ and $d=(0,1)$, if we consider the rotation of order four 
$\rho=(abcd)$ and the reflections 
$\tau_2=(bd)$, $\tau_3 = (ab)(cd)$ and $\tau_4 = (ad)(bc)$, then we have
$\Gamma(\Z_2^2,d_2) = \langle \rho, \tau_2 \rangle$, $\Gamma(\Z_2^2,d_3) = \langle \rho, \tau_3 \rangle$, and $\Gamma(\Z_2^2,d_4) = \langle \rho, \tau_4 \rangle$.
Finally, we compute the group associated to the Lee metric $d_5$. We have that 
\begin{equation} \label{capZ2XZ2}
\Gamma(\Z_2^2,d_5) = \mathrm{Aut}(Cay(\Z_2^2,P_1)) \cap \mathrm{Aut}(Cay(\Z_2^2,P_2)) \cap \mathrm{Aut}(Cay(\Z_2^2,P_3)).
\end{equation} 
Each $Cay(\Z_2^2,P_i)$, $i=1,2,3$, is the disjoint union of two copies of $K_2$. Hence, 
$$\mathrm{Aut}(Cay(\Z_2^2,P_i))=\mathrm{Aut}(2K_2) = \Z_2 \wr \Z_2 \simeq \D_4$$ for $i=1,2,3$. 
However, these automorphisms groups are all different and we now look at the intersections in \eqref{capZ2XZ2} more carefully. 
Note that 
\begin{align*}
\mathrm{Aut}(Cay(\Z_2^2,P_1)) = \{ id, (ab), (cd), (ab)(cd), (ac)(bd), (ad)(bc), (acbd), (adbc) \}, \\
\mathrm{Aut}(Cay(\Z_2^2,P_2)) = \{ id, (ac), (bd), (ab)(cd), (ac)(bd), (ad)(bc), (abcd), (adbc) \}, \\ 
\mathrm{Aut}(Cay(\Z_2^2,P_3)) = \{ id, (ad), (bc), (ab)(cd), (ac)(bd), (ad)(bc), (abdc), (acdb) \},
\end{align*}
and hence we clearly have 
$$\Gamma(\Z_2^2,d_5) = \bigcap_{1\le i \le 3} \mathrm{Aut}(Cay(\Z_2^2,P_i)) = \{ id, (ab)(cd), (ac)(bd), (bc)(ad) \}.$$
Therefore, by \eqref{capZ2XZ2} we finally obtain that
$\Gamma(\Z_2^2,d_5) \simeq \Z_2 \times \Z_2$, as we wanted to see.

From the above graphs, we see that we have the following weight functions on $\Z_2^2$: 
\renewcommand{\arraystretch}{1.1}
\begin{table}[H] 
	\caption{Invariant metrics for $\Z_2 \times \Z_2$} 
	\label{tabZ2xZ2}
	\begin{tabular}{|c|ccc|c|c|c|c|c|c|} 
		\hline
$\Z_2^2$ & $(1,0)$ & $(1,1)$ & $(0,1)$ & $s$ & poset & chain & hom.\@  & metrics & $\Gamma$ \\ \hline
$w_1$ &  1 & 1 & 1 & $1$ & \ding{51} & \ding{51} & \ding{51} & $d_{Ham}$, $d_{hom}^{\ff_4}$ & $\Sym_4$ \\ 
$w_2$ &  1 & 2 & 2 & $2$ & \ding{51} & \ding{51} & \ding{51} & $d_{1 \preceq 2}$, $d_{\Z_2, i_1}$, $d_{hom}^{R_1}$ & $\D_4$ \\ 
$w_3$ &  2 & 2 & 1 & $2$ & \ding{51} & \ding{51} & \ding{51} & $d_{2 \preceq 1}$, $d_{\Z_2, i_2}$, $d_{hom}^{R_2}$ & $\D_4$ \\ 
$w_4$ &  2 & 1 & 2 & $2$ & \ding{51} & \ding{51} & \ding{51} & $d_{Ham}^2$,$d_{\Z_2, i_1 \times i_2}$, $d_{hom}^{R_3}$ & $\D_4$ \\ 
$w_5$ &  1 & 2 & 3 & $3$ & \ding{55} & \ding{55} & \ding{55} & $d_{Lee}$ & $\D_4$ \\ 
		\hline
	\end{tabular}
\end{table}
The metric $d_1$ is the Hamming metric. We use the notation $d_{Lee}$ for $d_5$ since it is associated to the Lee (i.e.\@ finest unitary symmetric) partition. 
We have that $d_2$ and $d_3$ are poset metrics with respect to the posets on the set of coordinates of $\Z_2\times\Z_2$, namely $\{1,2\}$, determined by $1 \preceq 2$ and  $2 \preceq 1$, respectively. On the other hand, the metric $d_4$ is the poset metric with respect to the antichain on $\{1,2\}$, that gives rise to the product Hamming metric $d_{Ham}^2$. 
The metrics $d_2, d_3$ and $d_4$ are chain metrics with respect to the chains $\{0\} \subset \Z_2 \subset \Z_2 \times \Z_2$
given by the embeddings of $\Z_2 \hookrightarrow \Z_2 \times \Z_2$ respectively defined by 
$$i_1: x \mapsto (x,0), \qquad i_2 : x \mapsto (0,x), \qquad \text{and} \qquad i_1 \times i_2 : x \mapsto (x,x).$$ 
Finally, note that the Hamming metric $d_{Ham}$ is equivalent to the homegeneous metric $d_{hom}^{\ff_4}$ on the finite field of four elements $\ff_4$, while the metrics $d_2, d_3$ and $d_4$ can be respectively obtained as homogeneous weights on the finite rings $R_1, R_2$ and $R_3$ given by the conditions $a \cdot a=a$, $b \cdot b=0$ and $a \cdot b=b$ where $a=(1,0)$, $b=(0,1)$ for $R_1$, $a=(0,1)$, $b=(1,0)$ for $R_2$ and $a=(1,0)$, $b=(1,1)$ for $R_3$.
\hfill $\lozenge$
\end{exam}

\goodbreak 

In the next two examples we study the invariant metrics for the groups of order $6$. 

\begin{exam}
\label{MetricasZ_6}
Consider now the group $\Z_6$. It is easy to check that 
\begin{gather*}
\mathcal{P}_1 = \big\{ \{0\}, \{1,2,3,4,5\}\big\}, \qquad 
\mathcal{P}_2 = \big\{ \{0\}, \{1,2,4,5\}, \{3\}\big\}, \qquad 
\mathcal{P}_3 = \big\{ \{0\}, \{1,3,5\}, \{2,4\}\big\}, \\ 
\mathcal{P}_4 = \big\{ \{0\}, \{1,5\}, \{2,4,3\}\big\}, \qquad 
\mathcal{P}_5 = \big\{ \{0\}, \{1,5\}, \{2,4\}, \{3\} \big\},
\end{gather*}
are all the unitary symmetric partitions of $\mathbb{Z}_6$.
From these partitions, in Figures \ref{8}--\ref{12} we give the corresponding partition graphs: 
%(black=1, blue=2, red=3):
\begin{figure}[H]
\centering
\begin{minipage}[t][][t]{0.33\textwidth}

\begin{tikzpicture}[scale=.575,transform shape,>=stealth',shorten
>=1pt,auto,node distance=3cm, thick,main node/.style={circle,draw,font=\small}]
\tikzstyle{every node}=[node distance = 3cm,bend angle    = 45,fill = gray!30]

  \node[main node] (1) at (-2,1.5) {1};
  \node[main node] (2) at (-2,4) {2};
  \node[main node] (3) at (0,5.5) {3};
  \node[main node] (4) at (2,4) {4};
  \node[main node] (5) at (2,1.5) {5};
  \node[main node] (0) at (0,0) {0};

  \path[every node/.style={font=\large,circle}]
    (1) edge node[above,left,pos=0.4] {1} (4)
       edge node [left, pos=0.6] {1} (3)
       edge  node[above,pos=0.4] {1} (5)
    (2) edge node [left] {1} (1)
        edge node [below, pos=0.6] {1} (4)
        edge  node [left,pos=0.6] {1} (0)
        edge node [left] {1} (3)
        edge node[right,pos=0.6] {1} (5)
    (3) edge node [right] {1} (4) 
    edge  node [right, pos=0.4] {1} (5)
    edge  node [left,pos=0.2] {1} (0)
    (4) edge  node [right] {1} (5)
    edge  node [right,pos=0.6] {1} (0)
    (5) edge  node [right] {1} (0)
    (0) edge  node [left] {1} (1);
\end{tikzpicture}
\centering
\captionsetup{justification=centering,margin=0.1cm}
\caption{$\mathcal{G}_1 = \mathcal{G}(\Z_6,d_1)$} \label{8}
\end{minipage}
\hfill
\begin{minipage}[t][][t]{0.33\textwidth}
\begin{tikzpicture}[scale=.575,transform shape,>=stealth',shorten
>=1pt,auto,node distance=3cm, thick,main node/.style={circle,draw,font=\small}]
\tikzstyle{every node}=[node distance = 3cm,bend angle    = 45,fill          = gray!30]

  \node[main node] (1) at (-2,1.5) {1};
  \node[main node] (2) at (-2,4) {2};
  \node[main node] (3) at (0,5.5) {3};
  \node[main node] (4) at (2,4) {4};
  \node[main node] (5) at (2,1.5) {5};
  \node[main node] (0) at (0,0) {0};

  \path[every node/.style={font=\large,circle}]
    (1) edge [blue] node[left,pos=0.4] {2} (4)
       edge node [left, pos=0.6] {1} (3)
       edge node[above,pos=0.4] {1} (5)
    (2) edge node [left] {1} (1)
        edge node [below, pos=0.6] {1} (4)
        edge node [left,pos=0.6] {1} (0)
        edge node [left] {1} (3)
        edge [blue] node[right,pos=0.6] {2} (5)
    (3) edge node [right] {1} (4) 
    edge node [right, pos=0.4] {1} (5)
    edge [blue] node [left,pos=0.2] {2} (0)
    (4) edge  node [right] {1} (5)
    edge node [right,pos=0.6] {1} (0)
    (5) edge  node [right] {1} (0)
    (0) edge  node [left] {1} (1);
\end{tikzpicture}
\centering
\captionsetup{justification=centering,margin=0.1cm}
\caption{$\mathcal{G}_2 = \mathcal{G}(\Z_6,d_2)$} \label{9}
\end{minipage}%
\hfill
\begin{minipage}[t][][t]{0.33\textwidth}
\centering
\begin{tikzpicture}[scale=.575,transform shape,>=stealth',shorten
>=1pt,auto,node distance=3cm, thick,main node/.style={circle,draw,font=\small}]
\tikzstyle{every node}=[node distance = 3cm,bend angle    = 45,fill          = gray!30]

  \node[main node] (1) at (-2,1.5) {1};
  \node[main node] (2) at (-2,4) {2};
  \node[main node] (3) at (0,5.5) {3};
  \node[main node] (4) at (2,4) {4};
  \node[main node] (5) at (2,1.5) {5};
  \node[main node] (0) at (0,0) {0};
      
 \path[every node/.style={font=\large,circle}]
    (1) edge node[left,pos=0.4] {1} (4)
       edge [blue]  node [left, pos=0.6] {2} (3)
       edge [blue]  node[above,pos=0.4] {2} (5)
    (2) edge node [left] {1} (1)
        edge [blue]  node [below, pos=0.6] {2} (4)
        edge [blue]  node [left,pos=0.6] {2} (0)
        edge node [left] {1} (3)
        edge node[right,pos=0.6] {1} (5)
    (3) edge node [right] {1} (4) 
    edge [blue]  node [right, pos=0.4] {2} (5)
    edge node [left,pos=0.2] {1} (0)
    (4) edge  node [right] {1} (5)
    edge [blue]  node [right,pos=0.6] {2} (0)
    (5) edge  node [right] {1} (0)
    (0) edge  node [left] {1} (1);
\end{tikzpicture}
\centering
\captionsetup{justification=centering,margin=0.1cm}
\caption{$\mathcal{G}_3 = \mathcal{G}(\Z_6,d_3)$} \label{10}
\end{minipage}%
\hfill
\end{figure}

\begin{figure}[H]
\centering
\begin{minipage}[t][][t]{0.45\textwidth}%
\begin{tikzpicture}[scale=.575,transform shape,>=stealth',shorten
>=1pt,auto,node distance=3cm, thick,main node/.style={circle,draw,font=\small}]
\tikzstyle{every node}=[node distance = 3cm,bend angle    = 45,fill          = gray!30]

  \node[main node] (1) at (-2,1.5) {1};
  \node[main node] (2) at (-2,4) {2};
  \node[main node] (3) at (0,5.5) {3};
  \node[main node] (4) at (2,4) {4};
  \node[main node] (5) at (2,1.5) {5};
  \node[main node] (0) at (0,0) {0};
             
  \path[every node/.style={font=\large,circle}]
   (1) edge [blue]  node[left,pos=0.4] {2} (4)
       edge [blue]  node [left, pos=0.6] {2} (3)
       edge [blue]  node[above,pos=0.4] {2} (5)
    (2) edge node [left] {1} (1)
        edge [blue]  node [below, pos=0.6] {2} (4)
        edge [blue]  node [left,pos=0.6] {2} (0)
        edge node [left] {1} (3)
        edge [blue]  node[right,pos=0.6] {2} (5)
    (3) edge node [right] {1} (4) 
    edge [blue]  node [right, pos=0.4] {2} (5)
    edge [blue]  node [left,pos=0.2] {2} (0)
    (4) edge  node [right] {1} (5)
    edge [blue]  node [right,pos=0.6] {2} (0)
    (5) edge  node [right] {1} (0)
    (0) edge  node [left] {1} (1);
\end{tikzpicture}
\centering 
\captionsetup{justification=centering,margin=0.1cm}
\caption{$\mathcal{G}_4 = \mathcal{G}(\Z_6,d_4)$} \label{11}
\end{minipage} 
\begin{minipage}[t][][t]{0.45\textwidth}
\begin{tikzpicture}[scale=.575,transform shape,>=stealth',shorten
>=1pt,auto,node distance=3cm, thick,main node/.style={circle,draw,font=\small}]
\tikzstyle{every node}=[node distance = 3cm,bend angle    = 45,fill          = gray!30]

  \node[main node] (1) at (-2,1.5) {1};
  \node[main node] (2) at (-2,4) {2};
  \node[main node] (3) at (0,5.5) {3};
  \node[main node] (4) at (2,4) {4};
  \node[main node] (5) at (2,1.5) {5};
  \node[main node] (0) at (0,0) {0};

  \path[every node/.style={font=\large,circle}]
    (1) edge [red] node[left,pos=0.4] {3} (4)
       edge [blue]  node [left, pos=0.6] {2} (3)
       edge [blue]  node[above,pos=0.4] {2} (5)
    (2) edge node [left] {1} (1)
        edge[blue]  node [below, pos=0.6] {2} (4)
        edge [blue]  node [left,pos=0.6] {2} (0)
        edge node [left] {1} (3)
        edge [red] node[right,pos=0.6] {3} (5)
    (3) edge node [right] {1} (4) 
    edge [blue]  node [right, pos=0.4] {2} (5)
    edge [red] node [left,pos=0.2] {3} (0)
    (4) edge  node [right] {1} (5)
    edge [blue]  node [right,pos=0.6] {2} (0)
    (5) edge  node [right] {1} (0)
    (0) edge  node [left] {1} (1);
\end{tikzpicture}
%}
\centering
\captionsetup{justification=centering,margin=0.1cm}
\caption{$\mathcal{G}_5 = \mathcal{G}(\Z_6,d_5)$} \label{12}
\end{minipage}%
\hfill
\end{figure}
\vspace{-1em}

We know compute the symmetry groups. We will use the previous graphs together with Theorem \ref{capcays} and Corollary \ref{Cay(Cn)}. One must interpret the graphs as unions of the subgraphs given by each part of the partition (i.e.\@ each color); for instance for the Lee metric $d_5$ we have the following decomposition of the graph of distances into simple Cayley graphs:

\begin{center}
\begin{tikzpicture}[scale=.55,transform shape,>=stealth',shorten
>=1pt,auto,node distance=3cm, thick,main node/.style={circle,draw,font=\tiny}]
\tikzstyle{every node}=[node distance = 3cm,bend angle    = 45,fill          = gray!30]

\node[main node] (1) at (-2,1.5) {1};
\node[main node] (2) at (-2,4) {2};
\node[main node] (3) at (0,5.5) {3};
\node[main node] (4) at (2,4) {4};
\node[main node] (5) at (2,1.5) {5};
\node[main node] (0) at (0,0) {0};
\node[main node,draw=none,fill=none,scale=2.75] (G) at (0,-1.4)
{$Cay(\Z_6,P_{Lee})$};

\node[main node,draw=none,fill=none,scale=5] (=) at (3.2,2.8) {$=$};

\node[main node,right = 6 of 1] (1a)  {1};
\node[main node,right = 6 of 2] (2a)  {2};
\node[main node,right = 6 of 3] (3a)  {3};
\node[main node,right = 6 of 4] (4a)  {4};
\node[main node,right = 6 of 5] (5a)  {5};
\node[main node,right = 6 of 0] (0a)  {0};
\node[main node,draw=none,fill=none,scale=2.75] (G1) at (6.6,-1.4)
{$Cay(\Z_6,\{1,5\})$};

\node[main node,draw=none,fill=none,scale=5,right = 3.4 of {=}] (u1)
{$\bigcup$};

\node[main node,right = 6 of 1a] (1b)  {1};
\node[main node,right = 6 of 2a] (2b)  {2};
\node[main node,right = 6 of 3a] (3b)  {3};
\node[main node,right = 6 of 4a] (4b)  {4};
\node[main node,right = 6 of 5a] (5b)  {5};
\node[main node,right = 6 of 0a] (0b)  {0};
\node[main node,draw=none,fill=none,scale=2.75] (G2) at (13.3,-1.4)
{$Cay(\Z_6,\{2,4\})$};

\node[main node,draw=none,fill=none,scale=5,right = 3.2 of u1] (u2)
{$\bigcup$};

\node[main node,right = 6 of 1b] (1c)  {1};
\node[main node,right = 6 of 2b] (2c)  {2};
\node[main node,right = 6 of 3b] (3c)  {3};
\node[main node,right = 6 of 4b] (4c)  {4};
\node[main node,right = 6 of 5b] (5c)  {5};
\node[main node,right = 6 of 0b] (0c)  {0};
\node[main node,draw=none,fill=none,scale=2.75] (G3) at (19.6,-1.4)
{$Cay(\Z_6,\{3\})$};

\path[every node/.style={font=\large,circle}]
(1) edge [red] node[left=5pt,pos=0.4] {3} (4)
edge [blue]  node [left, pos=0.55] {2} (3)
edge [blue]  node[above,pos=0.4] {2} (5)
(2) edge node [left] {1} (1)
edge[blue]  node [below, pos=0.6] {2} (4)
edge [blue]  node [left,pos=0.55] {2} (0)
edge node [left] {1} (3)
edge [red] node[right=5pt,pos=0.6] {3} (5)
(3) edge node [right] {1} (4) 
edge [blue]  node [right, pos=0.45] {2} (5)
edge [red] node [left,pos=0.2] {3} (0)

(4) edge  node [right] {1} (5)
edge [blue]  node [right,pos=0.55] {2} (0)
(5) edge  node [right] {1} (0)
(0) edge  node [left] {1} (1)

(2a) edge node [left] {1} (1a)
edge node [left] {1} (3a)
(3a) edge node [right] {1} (4a) 
(4a) edge  node [right] {1} (5a)
(5a) edge  node [right] {1} (0a)
(0a) edge  node [left] {1} (1a)

(1b) edge [blue]  node [left, pos=0.55] {2} (3b)
edge [blue]  node[above,pos=0.4] {2} (5b)
(2b) edge[blue]  node [below, pos=0.6] {2} (4b)
edge [blue]  node [left,pos=0.55] {2} (0b)
(3b)  edge [blue]  node [right, pos=0.45] {2} (5b)
(4b) edge [blue]  node [right,pos=0.55] {2} (0b)
(1c) edge [red] node[left=5pt,pos=0.4] {3} (4c)
(2c) edge [red] node[right=5pt,pos=0.6] {3} (5c)
(3c) edge [red] node [left,pos=0.2] {3} (0c);
\end{tikzpicture}
\end{center}

Now, by \eqref{G Aut} and \eqref{AutGPi}
we have that  
$$\begin{aligned}
& \Gamma(\Z_6,d_1) = \mathrm{Aut} (K_6) \simeq \Sym_6, \\
& \Gamma(\Z_6,d_2) = \mathrm{Aut} (C_6 \cup 2C_3) \cap \mathrm{Aut} (3K_2) = \mathrm{Aut} (3K_2) \simeq \Sym_2 \wr \Sym_3, \\
& \Gamma(\Z_6,d_3) = \mathrm{Aut} (C_6 \cup 3K_2) \cap \mathrm{Aut} (2C_3) = \mathrm{Aut} (2C_3) \simeq \Sym_3 \wr \Sym_2, \\
& \Gamma(\Z_6,d_4) = \mathrm{Aut} (C_6) \cap \mathrm{Aut} (C_6 \cup 3K_2) = \mathrm{Aut} (C_6) \simeq \D_6. \\
& \Gamma(\Z_6,d_5) = \mathrm{Aut} (C_6) \cap \mathrm{Aut} (2C_3) \cap \mathrm{Aut} (3K_2) \simeq \Di_6 \cap (\Sym_3 \wr \Sym_2) \cap 
(\Sym_2 \wr \Sym_3) \simeq \D_6.
\end{aligned}$$
We now explain these isomorphisms. The first one is clear. For the second, third and fourth, note that the group of symmetries is the intersection of the automorphism group of a graph with the auromorphism group of the complementary graph (see Figures \ref{9}--\ref{11}), which are equal. Hence, it is enough to compute the easiest case (second identities). For the last group we cannot use this trick, since we have 3 graphs.   
However, one can check that every automorphism of $Cay(\Z_6,\{1,5\}) \simeq \D_6$ is also an automorphism of
$Cay(\Z_6,\{2,4\}) \simeq \Sym_3 \wr \Sym_2$ and of $Cay(\Z_6,\{3\})\simeq \Sym_2 \wr \Sym_3$, and hence we finally get
$ \Gamma(\Z_6,d_5)\simeq\Di_6$.

From the graphs, we see that we have the following weight functions on $\Z_6$: 
\begin{table}[H]
\caption{Invariant metrics for $\Z_6$} \label{tabla Z6}
\begin{equation*}  \label{tabZ6}
\begin{tabular}{|c|ccccc|c|c|c|c|c|c|c|}
\hline
$\Z_6$& $1$ & $2$ & $3$ & $4$ & $5$ & $s$ & poset & chain & hom.\@ & known metrics  & $\Gamma$ & order \\ \hline
 $w_1$ & $1$ & $1$ & $1$ & $1$ & $1$ & 1 & \ding{51} & \ding{51} & \ding{51} & $d_{Ham}$ & $\Sym_6$& 720 \\
 $w_2$ & $1$ & $1$ & $2$ & $1$ & $1$ & 2 & \ding{51} & \ding{51} & \ding{55} & $d_{1\preceq 2}$, $d_{\Z_3}$ & $\Sym_2 \wr \Sym_3$ &48\\ 
 $w_3$ & $1$ & $2$ & $1$ & $2$ & $1$ & 2 & \ding{51} & \ding{51} & \ding{55} & $d_{2\preceq 1}$, $d_{\Z_2}$ & $\Sym_3 \wr \Sym_2$ & 72 \\ 
 $w_4$ & $1$ & $2$ & $2$ & $2$ & $1$ & 2 & \ding{51} & \ding{55} & \ding{55} & $d_{3,Lee}$, $d_{Ham}^2$ & $\D_6$ & 12\\
 $w_5$ & $1$ & $2$ & $3$ & $2$ & $1$ & 3 & \ding{55} & \ding{55} & \ding{51} & $d_{Lee}=d_{2,Lee}$, $d_{hom}$ & $\D_6$ & 12 \\ 
\hline
\end{tabular}
\end{equation*}
\end{table}

The metrics $d_1$ and $d_5$ are Lee and Hamming respectively. 
For poset metrics we think $\Z_6$ as the product $\Z_2 \times \Z_3$ (i.e.\@ 2 coordinates) via the isomorphism $1\mapsto (1,1)$, since otherwise we get the Hamming metric. We have that $d_2, d_3, d_4$ are poset metrics with respect to the posets on the set of coordinates $\{1,2\}$, given by $1 \preceq 2$ and $2 \preceq 1$ and the trivial poset (2-antichain) respectively. 
The metrics $d_3$ and $d_4$ are chain metrics with respect to the chains $\{0\}\subset \Z_2 \subset \Z_6$ and $\{0\}\subset \Z_3 \subset \Z_6$ given by the embeddings $x \mapsto 2x$ and $x \mapsto 3x$, respectively.
The ring of integers modulo $6$ has the poset of principal ideals given by $\{0\} \preceq 2\Z_6 \preceq \Z_6$ and $\{0\} \preceq 3\Z_6 \preceq \Z_6$ where $\mu(\{0\},\{0\})=1$, $\mu(\{0\},2\Z_6)=-1$, $\mu(\{0\},3\Z_6)=-1$ and $\mu(\{0\},\Z_6)=1$. 
By \eqref{hom weight}, we have 
$w_{hom}(1)=w_{hom}(5)=\tfrac 12 \lambda$, $w_{hom}(2)=w_{hom}(4)=\frac 32 \lambda$ and $w_{hom}(3)=2\lambda$.
\hfill $\lozenge$ 
\end{exam}

\begin{exam} \label{Sym3}
Consider the group $\Sym_3=\{id, (12), (13), (23), (123), (132)\}$.
For simplicity we will write $ij$ and $ijk$ instead of $(ij)$ and $(ijk)$.
By inspection one can see that $\Sym_3$ has 15 unitary symmetric partitions given by 
{\small \begin{alignat*}{3}
&\mathcal{P}_1 : \{12, 13, 23, 123, 132\}, 	\quad && \mathcal{P}_2: \{12, 13, 23\}, \{123, 132\}, \quad && 
\mathcal{P}_3 : \{12\}, \{13, 23, 123, 132\}, \\
& \mathcal{P}_4 : \{13\}, \{12, 23, 123, 132\}, \quad && \mathcal{P}_5 : \{23\}, \{12, 13, 123, 132\},	
\quad && \mathcal{P}_6 : \{12, 13\}, \{23, 123, 132\}, \\
&\mathcal{P}_7 : \{12, 23\}, \{13, 123, 132\}, \quad && \mathcal{P}_8 : \{13, 23\}, \{12, 123, 132\}, 
\quad && \mathcal{P}_9 : \{12\}, \{13, 23\}, \{123, 132\}, \\
& \mathcal{P}_{10} : \{13\}, \{12, 23\}, \{123, 132\}, \quad && \mathcal{P}_{11} : \{23\}, \{12, 13\},  \{123, 132\}, 
\quad && \mathcal{P}_{12} : \{12\}, \{13\}, \{23, 123, 132\}, \\	
&\mathcal{P}_{13} : \{12\}, \{23\}, \{13, 123, 132\}, \quad && \mathcal{P}_{14} : \{13\}, \{23\}, \{12, 123, 132\}, 	
\quad && \mathcal{P}_{15} :\{12\}, \{13\}, \{23\}, \{123, 132\}.  
\end{alignat*}}
It is clear from this list that there are only two conjugate partitions (the first two ones) and hence $\Sym_3$ has only 2 bi-invariant metrics. The partition graphs of the metrics (except for the Hamming one) are given in Figures \ref{13}--\ref{18}. %{\blue(black=1, blue=2, red=3, green-dashed=4)}. 

\begin{figure}[H]
	\centering
	\begin{minipage}[t][][t]{0.33\textwidth}
		
		\begin{tikzpicture}[scale=.575,transform shape,>=stealth',shorten
			>=1pt,auto,node distance=3cm, thick,main node/.style={circle,draw,font=\small}]
			\tikzstyle{every node}=[node distance = 3cm,minimum size=9mm,bend angle    = 45,fill = gray!30]
			
			\node[main node] (1) at (-2,1.5) {12};
			\node[main node] (2) at (-2,4) {123};
			\node[main node] (3) at (0,5.5) {13};
			\node[main node] (4) at (2,4) {132};
			\node[main node] (5) at (2,1.5) {23};
			\node[main node] (0) at (0,0) {id};
			
			\path[every node/.style={font=\large,circle}]
			(1) edge node[left,pos=0.4] {1} (4)
			edge[blue] node [left, pos=0.6] {2} (3)
			edge[blue]  node[above, pos=0.4] {2} (5)
			(2) edge node [left] {1} (1)
			edge[blue] node [below, pos=0.6] {2} (4)
			edge[blue]  node [left,pos=0.6] {2} (0)
			edge node [left] {1} (3)
			edge node[right,pos=0.6] {1} (5)
			(3) edge node [right] {1} (4) 
			edge[blue]  node [right, pos=0.4] {2} (5)
			edge  node [left,pos=0.2] {1} (0)
			(4) edge  node [right] {1} (5)
			edge[blue]  node [right,pos=0.6] {2} (0)
			(5) edge  node [right] {1} (0)
			(0) edge  node [left] {1} (1);
		\end{tikzpicture}
		\centering
		\captionsetup{justification=centering,margin=0.1cm}
		\caption{$\mathcal{G}_1 = \mathcal{G}(\Sym_3,d_2)$} \label{13}
	\end{minipage}
	\hfill
	\begin{minipage}[t][][t]{0.33\textwidth}
		\begin{tikzpicture}[scale=.575,transform shape,>=stealth',shorten
			>=1pt,auto,node distance=3cm, thick,main node/.style={circle,draw,font=\small}]
			\tikzstyle{every node}=[node distance = 3cm,minimum size=9mm,bend angle    = 45,fill          = gray!30]
			
			\node[main node] (1) at (-2,1.5) {12};
			\node[main node] (2) at (-2,4) {123};
			\node[main node] (3) at (0,5.5) {13};
			\node[main node] (4) at (2,4) {132};
			\node[main node] (5) at (2,1.5) {23};
			\node[main node] (0) at (0,0) {id};
			
			\path[every node/.style={font=\large,circle}]
			(1) edge [blue] node[left,pos=0.4] {2} (4)
			edge [blue]  node [left, pos=0.6] {2} (3)
			edge [blue]  node[above,pos=0.4] {2} (5)
			(2) edge [blue] node [left] {2} (1)
			edge [blue]  node [below, pos=0.6] {2} (4)
			edge [blue]  node [left,pos=0.6] {2} (0)
			edge [blue] node [left] {2} (3)
			edge node[right,pos=0.6] {1} (5)
			(3) edge node [right] {1} (4) 
			edge [blue]  node [right, pos=0.4] {2} (5)
			edge [blue] node [left,pos=0.2] {2} (0)
			(4) edge [blue]  node [right] {2} (5)
			edge [blue]  node [right,pos=0.6] {2} (0)
			(5) edge node [right] {2} (0)
			(0) edge  node [left] {1} (1);
		\end{tikzpicture}
		\centering
		\captionsetup{justification=centering,margin=0.1cm}
		\caption{$\mathcal{G}_2 = \mathcal{G}(\Sym_3,d_3)$} \label{14}
	\end{minipage}%
	\hfill
	\begin{minipage}[t][][t]{0.33\textwidth}
		\centering
		\begin{tikzpicture}[scale=.575,transform shape,>=stealth',shorten
			>=1pt,auto,node distance=3cm, thick,main node/.style={circle,draw,font=\small}]
			\tikzstyle{every node}=[node distance = 3cm,minimum size=9mm,bend angle    = 45,fill          = gray!30]
			
			\node[main node] (1) at (-2,1.5) {12};
			\node[main node] (2) at (-2,4) {123};
			\node[main node] (3) at (0,5.5) {13};
			\node[main node] (4) at (2,4) {132};
			\node[main node] (5) at (2,1.5) {23};
			\node[main node] (0) at (0,0) {id};
			
			\path[every node/.style={font=\large,circle}]
			(1) edge node[left,pos=0.4] {1} (4)
			edge [blue]  node [left, pos=0.6] {2} (3)
			edge   node[above,pos=0.4] {2} (5)
			(2) edge [blue] node [left] {2} (1)
			edge [blue]  node [below, pos=0.6] {2} (4)
			edge [blue]  node [left,pos=0.6] {2} (0)
			edge  node [left] {1} (3)
			edge node[right,pos=0.6] {1} (5)
			(3) edge node [right] {1} (4) 
			edge [blue]  node [right, pos=0.4] {2} (5)
			edge [blue] node [left,pos=0.2] {2} (0)
			(4) edge [blue]  node [right] {2} (5)
			edge [blue]  node [right,pos=0.6] {2} (0)
			(5) edge node [right] {1} (0)
			(0) edge  node [left] {1} (1);
		\end{tikzpicture}
		\centering
		\captionsetup{justification=centering,margin=0.1cm}
		\caption{$\mathcal{G}_3 = \mathcal{G}(\Sym_3,d_7)$} \label{15}
	\end{minipage}%
	\hfill
\end{figure}
\begin{figure}[H]
	\centering
	\begin{minipage}[t][][t]{0.33\textwidth}%
		\begin{tikzpicture}[scale=.575,transform shape,>=stealth',shorten
			>=1pt,auto,node distance=3cm, thick,main node/.style={circle,draw,font=\small}]
			\tikzstyle{every node}=[node distance = 3cm,minimum size=9mm,bend angle    = 45,fill          = gray!30]
			
			\node[main node] (1) at (-2,1.5) {12};
			\node[main node] (2) at (-2,4) {123};
			\node[main node] (3) at (0,5.5) {13};
			\node[main node] (4) at (2,4) {132};
			\node[main node] (5) at (2,1.5) {23};
			\node[main node] (0) at (0,0) {id};
			
			\path[every node/.style={font=\large,circle}]
			(1) edge [blue] node[left,pos=0.4] {2} (4)
			edge [red]  node [left, pos=0.6] {3} (3)
			edge [red]  node[above,pos=0.4] {3} (5)
			(2) edge node [left] {1} (1)
			edge [red]  node [below, pos=0.6] {3} (4)
			edge [red]  node [left,pos=0.6] {3} (0)
			edge [blue] node [left] {2} (3)
			edge [blue] node[right,pos=0.6] {2} (5)
			(3) edge [blue] node [right] {2} (4) 
			edge [red]  node [right, pos=0.4] {3} (5)
			edge  node [left,pos=0.2] {1} (0)
			(4) edge  node [right] {1} (5)
			edge [red]  node [right,pos=0.6] {3} (0)
			(5) edge [blue] node [right] {2} (0)
			(0) edge [blue] node [left] {2} (1);
		\end{tikzpicture}
		\centering 
		\captionsetup{justification=centering,margin=0.1cm}
		\caption{$\mathcal{G}_4 = \mathcal{G}(\Sym_3,d_{10})$} \label{16}
	\end{minipage} 
	\begin{minipage}[t][][t]{0.33\textwidth}
		\begin{tikzpicture}[scale=.575,transform shape,>=stealth',shorten
			>=1pt,auto,node distance=3cm, thick,main node/.style={circle,draw,font=\small}]
			\tikzstyle{every node}=[node distance = 3cm,minimum size=9mm,bend angle    = 45,fill          = gray!30]
			
			\node[main node] (1) at (-2,1.5) {12};
			\node[main node] (2) at (-2,4) {123};
			\node[main node] (3) at (0,5.5) {13};
			\node[main node] (4) at (2,4) {132};
			\node[main node] (5) at (2,1.5) {23};
			\node[main node] (0) at (0,0) {id};
			
			\path[every node/.style={font=\large,circle}]
			(1) edge [blue] node[left,pos=0.4] {2} (4) 
			edge [red]  node [left, pos=0.6] {3} (3)
			edge [red]  node[above,pos=0.4] {3} (5)
			(2) edge [red] node [left] {1} (1)
			edge [red]  node [below, pos=0.6] {3} (4)
			edge [red]  node [left,pos=0.6] {3} (0)
			edge [blue] node [left] {2} (3) 
			edge node[right,pos=0.6] {1} (5)
			(3) edge node [right] {1} (4) 
			edge [red]  node [right, pos=0.4] {3} (5)
			edge [red] node [left,pos=0.2] {3} (0)
			(4) edge [red]  node [right] {3} (5)
			edge [red]  node [right,pos=0.6] {3} (0)
			(5) edge [blue] node [right] {2} (0) 
			(0) edge  node [left] {1} (1);
		\end{tikzpicture}
		%}
	\centering
	\captionsetup{justification=centering,margin=0.1cm}
	\caption{$\mathcal{G}_5 = \mathcal{G}(\Sym_3,d_{13})$} \label{17}
\end{minipage}%
\begin{minipage}[t][][t]{0.33\textwidth}
	\begin{tikzpicture}[scale=.575,transform shape,>=stealth',shorten
		>=1pt,auto,node distance=3cm, thick,main node/.style={circle,draw,font=\small}]
		\tikzstyle{every node}=[node distance = 3cm,minimum size=9mm,bend angle    = 45,fill          = gray!30]
		
		\node[main node] (1) at (-2,1.5) {12};
		\node[main node] (2) at (-2,4) {123};
		\node[main node] (3) at (0,5.5) {13};
		\node[main node] (4) at (2,4) {132};
		\node[main node] (5) at (2,1.5) {23};
		\node[main node] (0) at (0,0) {id};
		
		\path[every node/.style={font=\large,circle}]
			(1) edge [blue] node[left,pos=0.4] {2} (4)
			edge [black!50!green,dashed]  node [left, pos=0.6] {4} (3)
			edge [black!50!green,dashed]  node[above,pos=0.4] {4} (5)
			(2) edge [red] node [left] {3} (1)
			edge [black!50!green,dashed]  node [below, pos=0.6] {4} (4)
			edge [black!50!green,dashed]  node [left,pos=0.6] {4} (0)
			edge [blue] node [left] {2} (3)
			edge  node[right,pos=0.6] {1} (5)
			(3) edge  node [right] {1} (4) 
			edge [black!50!green,dashed]  node [right, pos=0.4] {4} (5)
			edge [red] node [left,pos=0.2] {3} (0)
			(4) edge [red] node [right] {3} (5)
			edge [black!50!green,dashed]  node [right,pos=0.6] {4} (0)
			(5) edge [blue] node [right] {2} (0)
			(0) edge  node [left] {1} (1);
	\end{tikzpicture}
	%}
\centering
\captionsetup{justification=centering,margin=0.1cm}
\caption{$\mathcal{G}_5 = \mathcal{G}(\Sym_3,d_{15})$} \label{18}
\end{minipage}%
\hfill
\end{figure}

Applying automorphisms of $\Sym_3$ one can obtain the isomorphic graphs corresponding to the rest of the metrics. Thus, $d_3, d_4, d_5$ are given by Figure \ref{14}, $d_6,d_7,d_8$ by Figure \ref{15}, $d_9,d_{10}, d_{11}$ by Figure \ref{16} and $d_{12},d_{13},d_{14}$ by Figure \ref{17}.

Proceeding as in the previous examples we can obtain all the symmetry groups and metrics. We leave the details and summarize the information in the following table. 
\renewcommand{\arraystretch}{.95}
\begin{table}[H] 
\caption{Invariant metrics for $\Sym_3$} \label{tabla S3}
\centering
$$\begin{tabular}{|c|ccccc|c|c|c|c|c|} \hline
$\#$ & 12 & 13 & 23 & 123 & 132 & $s$ & chain & metrics & $\Gamma(\Sym_3,d)$ & order \\ \hline
1 & 1&1&1&1&1			& 1 & \ding{55} & $d_{Ham}$ & $\Sym_6$ & 720 \\ \hline
2 & 1&1&1&2&2  			& 2 & \ding{51} & $d_{\langle (123) \rangle}$& $\Sym_3 \wr \Sym_2$ & 72 \\ \hline
3 & 1&2&2&2&2	  		& 2 & \ding{51} & $d_{\langle (12) \rangle}$& $\Sym_2\wr \Sym_3$ & 48 \\ \hline
4 & 2&1&2&2&2			& 2 & \ding{51} & $d_{\langle (13) \rangle}$& $\Sym_2\wr \Sym_3$ & 48 \\ \hline
5 & 2&2&1&2&2 			& 2 & \ding{51} & $d_{\langle (23) \rangle}$ & $\Sym_2\wr \Sym_3$ & 48 \\ \hline
6 & 1&1&2&2&2	  		& 2 & \ding{55} & & $\D_6$ & 12 \\ \hline
7 & 1&2&1&2&2  			& 2 & \ding{55} & & $\D_6'$ & 12 \\ \hline
8 & 2&1&1&2&2 			& 2 & \ding{55} & & $\D_6''$ & 12 \\ \hline
9 & 1&2&2&3&3 			& 3 & \ding{55} & & $\D_6''$ & 12 \\ \hline
10& 2&1&2&3&3		 	& 3 & \ding{55} & & $\D_6'$ & 12 \\ \hline
11& 2&2&1&3&3		 	& 3 & \ding{55} & & $\D_6$ & 12 \\ \hline
12& 1&2&3&3&3		 	& 3 & \ding{55} & & $\Sym_3$ & 6\\ \hline
13& 1&3&2&3&3			& 3 & \ding{55} & & $\Sym_3$ & 6 \\ \hline
14& 3&1&2&3&3	 		& 3 & \ding{55} & & $\Sym_3$ & 6 \\ \hline
15& 1&2&3&4&4		    & 4 & \ding{51} & $d_{Lee}$ & $\Sym_3$ & 6 \\
\hline
\end{tabular}$$
\end{table}
The metrics $d_4, d_6$ and $d_{10}$ are chain metrics with respect to the chains $\Z_2 \subset \Sym_3$ 
where $\Z_2$ is respectively given by $\langle (23) \rangle$, $\langle (12) \rangle$ and $\langle (13) \rangle$;
while $d_2$ is a chain metric with respect to $\Z_3 \subset \Sym_3$ with $\Z_3=\langle (123) \rangle$.
\hfill $\lozenge$	
\end{exam}

Finally, we study the invariant metrics of the group $\Z_7$.
\begin{exam}
The group $\Z_7$ has 5 unitary symmetric partitions given by  
\begin{gather*}
\mathcal{P}_1 = \big\{ \{0\}, \{1,2,3,4,5,6\}\big\},
\mathcal{P}_2 = \big\{ \{0\}, \{1,6\}, \{2,3,4,5\}\big\}, \\
\mathcal{P}_3 = \big\{ \{0\}, \{2,5\}, \{1,3,4,6\}\big\}, \quad 
\mathcal{P}_4 = \big\{ \{0\}, \{3,4\}, \{1,2,5,6\}\big\}, \quad 
\mathcal{P}_5 = \big\{ \{0\}, \{1,6\}, \{2,5\}, \{3,4\} \big\}.
\end{gather*}
Notice that in this case, partitions $\mathcal{P}_2$, $\mathcal{P}_3$ and $\mathcal{P}_4$ are isomorphic since there are 
automorphisms of $\Z_7$ sending the parts of any partition to the parts of the another ones. For instance, one can pass from  $\mathcal{P}_2$ to $\mathcal{P}_3$ by multiplying by $2$ and from $\mathcal{P}_2$ to $\mathcal{P}_4$ by multiplying by $3$.
The partition graphs of these metrics are as follows: % {\blue(black=1, blue=2, red=3)}
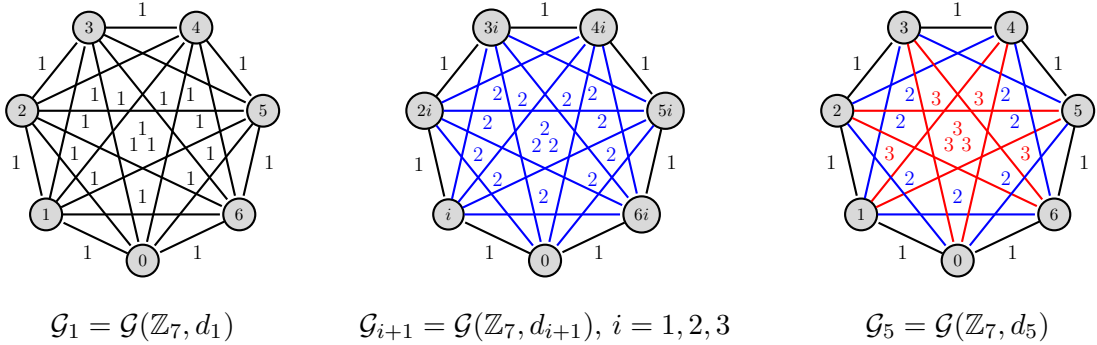
\begin{figure}[h!]
	\centering
	\begin{minipage}[t][][t]{0.33\textwidth}
	\centering
\begin{tikzpicture}[scale=.65,transform shape,>=stealth',shorten
	>=1pt,auto,node distance=3cm, thick,main node/.style={circle,draw,font=\small}]
	\tikzstyle{every node}=[node distance = 3cm,bend angle    = 45,fill          = gray!30]
	
	\node[main node] (1) at (6*360/7-90:2.5) {$1$};
	\node[main node] (2) at (5*360/7-90:2.5) {$2$};
	\node[main node] (3) at (4*360/7-90:2.5) {$3$};
	\node[main node] (4) at (3*360/7-90:2.5) {$4$};
	\node[main node] (5) at (2*360/7-90:2.5) {$5$};
	\node[main node] (6) at (1*360/7-90:2.5) {$6$};
	\node[main node] (0) at (0*360/7-90:2.5) {$0$};
		
	\path[every node/.style={font=\large,circle}]
	
	(1) edge node[left, pos=0.64] {1} (4)
			edge node[right, pos=0.5] {1} (3)
			edge node[below,pos=0.8] {1} (5)
			edge node[above,pos=0.5] {1} (6)
	(2) edge node[left] {1} (1)
			edge node[below, pos=0.4] {1} (4)
			edge node[right,pos=0.44] {1} (0)
			edge node[left,pos=0.6] {1} (3)
			edge node[below,pos=0.5] {1} (5)
			edge node[below,pos=0.2] {1} (6)
	(3) edge node[above] {1} (4) 
			edge node[below, pos=0.6] {1} (5)
			edge node[right,pos=0.36] {1} (6)
			edge node[right,pos=0.5] {1} (0)
	(4) edge node[right,pos=0.4] {1} (5)
			edge node[left] {1} (6)
			edge node[left,pos=0.5] {1} (0)
	(5) edge node[right,pos=0.5] {1} (6)
			edge node[left,pos=0.44] {1} (0)
	(6) edge  node [below,pos=0.4] {1} (0)
	(0) edge  node [below,pos=0.6] {1} (1);
	\end{tikzpicture}
	\centering
	\captionsetup{justification=centering,margin=0.1cm}
	\caption*{$\mathcal{G}_1 = \mathcal{G}(\Z_7,d_1)$} \label{19}
\end{minipage}%
\hfill
\begin{minipage}[t][][t]{0.33\textwidth}%
	
\begin{tikzpicture}[scale=.65,transform shape,>=stealth',shorten
	>=1pt,auto,node distance=3cm, thick,main node/.style={circle,draw,font=\small}]
	\tikzstyle{every node}=[node distance = 3cm,bend angle    = 45,fill          = gray!30]
	
	\node[main node] (1) at (6*360/7-90:2.5) {$i$};
	\node[main node] (2) at (5*360/7-90:2.5) {$2i$};
	\node[main node] (3) at (4*360/7-90:2.5) {$3i$};
	\node[main node] (4) at (3*360/7-90:2.5) {$4i$};
	\node[main node] (5) at (2*360/7-90:2.5) {$5i$};
	\node[main node] (6) at (1*360/7-90:2.5) {$6i$};
	\node[main node] (0) at (0*360/7-90:2.5) {$0$};
	
   	\path[every node/.style={font=\large,circle}]

   	(1) edge [blue] node[left, pos=0.64] {2} (4)
			edge [blue] node[right, pos=0.5] {2} (3)
			edge [blue] node[below,pos=0.8] {2} (5)
			edge [blue] node[above,pos=0.5] {2} (6)
	(2) edge node[left] {1} (1)
			edge [blue] node[below, pos=0.4] {2} (4)
			edge [blue] node[right,pos=0.44] {2} (0)
			edge node[left,pos=0.6] {1} (3)
			edge [blue] node[below,pos=0.5] {2} (5)
			edge [blue] node[below,pos=0.2] {2} (6)
	(3) edge node[above] {1} (4) 
			edge [blue] node[below, pos=0.6] {2} (5)
			edge [blue] node[right,pos=0.36] {2} (6)
			edge [blue] node[right,pos=0.5] {2} (0)
	(4) edge node[right,pos=0.4] {1} (5)
			edge [blue] node[left] {2} (6)
			edge [blue] node[left,pos=0.5] {2} (0)
	(5) edge node[right,pos=0.5] {1} (6)
			edge [blue] node[left,pos=0.44] {2} (0)
	(6) edge  node [below,pos=0.4] {1} (0)
	(0) edge  node [below,pos=0.6] {1} (1);
	\end{tikzpicture}
	\centering 
	\captionsetup{justification=centering,margin=0.1cm}
	\caption*{$\mathcal{G}_{i+1} = \mathcal{G}(\Z_7,d_{i+1})$, $i=1,2,3$} \label{20}
\end{minipage} 
\hfill
\begin{minipage}[t][][t]{0.33\textwidth}
		\begin{tikzpicture}[scale=.65,transform shape,>=stealth',shorten
	>=1pt,auto,node distance=3cm, thick,main node/.style={circle,draw,font=\small}]
	\tikzstyle{every node}=[node distance = 3cm,bend angle    = 45,fill          = gray!30]
	
	\node[main node] (1) at (6*360/7-90:2.5) {$1$};
	\node[main node] (2) at (5*360/7-90:2.5) {$2$};
	\node[main node] (3) at (4*360/7-90:2.5) {$3$};
	\node[main node] (4) at (3*360/7-90:2.5) {$4$};
	\node[main node] (5) at (2*360/7-90:2.5) {$5$};
	\node[main node] (6) at (1*360/7-90:2.5) {$6$};
	\node[main node] (0) at (0*360/7-90:2.5) {$0$};
	
    \path[every node/.style={font=\large,circle}]
    
    (1) edge [red] node[left, pos=0.64] {3} (4)
			edge [blue] node[right, pos=0.5] {2} (3)
			edge [red] node[below,pos=0.8] {3} (5)
			edge [blue] node[above,pos=0.5] {2} (6)
	(2) edge node[left] {1} (1)
			edge [blue] node[below, pos=0.4] {2} (4)
			edge [blue] node[right,pos=0.44] {2} (0)
			edge node[left,pos=0.6] {1} (3)
			edge [red] node[below,pos=0.5] {3} (5)
			edge [red] node[below,pos=0.2] {3} (6)
	(3) edge node[above] {1} (4) 
			edge [blue] node[below, pos=0.6] {2} (5)
			edge [red] node[right,pos=0.36] {3} (6)
			edge [red] node[right,pos=0.5] {3} (0)
	(4) edge node[right,pos=0.4] {1} (5)
			edge [blue] node[left] {2} (6)
			edge [red] node[left,pos=0.5] {3} (0)
	(5) edge node[right,pos=0.5] {1} (6)
			edge [blue] node[left,pos=0.44] {2} (0)
	(6) edge  node [below,pos=0.4] {1} (0)
	(0) edge  node [below,pos=0.6] {1} (1);
	
	(1) edge [red]  node[left, pos=0.4] {3} (4)
			edge [blue]  node[left, pos=0.6] {2} (3)
			edge [red]  node[above,pos=0.4] {3} (5)
			edge [blue]  node[above,pos=0.4] {2} (6)
	(2) edge node [left] {1} (1)
			edge [blue]  node [below, pos=0.6] {2} (4)
			edge [blue]  node [left,pos=0.6] {2} (0)
			edge node [left] {1} (3)
			edge [red]  node[right,pos=0.6] {3} (5)
			edge [red]  node[right,pos=0.6] {3} (6)
	(3) edge node [above] {1} (4) 
			edge [blue]  node [right, pos=0.4] {2} (5)
			edge [red]  node [left,pos=0.2] {3} (6)
			edge [red]  node [left,pos=0.2] {3} (0)
	(4) edge  node [right] {1} (5)
			edge [blue] node [right] {2} (6)
			edge [red] node [right,pos=0.6] {3} (0)
	(5) edge  node [right] {1} (6)
			edge [blue]  node [right,pos=0.6] {2} (0)
	(6) edge  node [below,pos=0.4] {1} (0)
	(0) edge  node [below,pos=0.6] {1} (1);
	\end{tikzpicture}
	\centering
		\captionsetup{justification=centering,margin=0.1cm}
		\caption*{$\mathcal{G}_5 = \mathcal{G}(\Z_7,d_5)$} \label{21}
	\end{minipage}
\hfill
\end{figure}

From the above graphs, we see that we have the following $s$-weight functions on $\Z_7$: 
\renewcommand{\arraystretch}{1.1}
\begin{table}[H]
\caption{Invariant metrics for $\Z_7$} \label{tabla Z7}	
\begin{equation*}  \label{tabZ7}
\begin{tabular}{|c|cccccc|c|c|c|}
\hline
$\Z_7$& $1$ & $2$ & $3$ & $4$ & $5$ & $6$ & $s$ & metrics & $\Gamma$ \\ \hline
$w_1$ & $1$ & $1$ & $1$ & $1$ & $1$ & $1$ & $1$ & $d_{Ham}$, $d_{hom}^{\ff_7}$ & $\Sym_7$ \\
$w_2$ & $2$ & $2$ & $1$ & $1$ & $2$ & $2$ & $2$ &  & $\D_7$ \\  
$w_3$ & $2$ & $1$ & $2$ & $2$ & $1$ & $2$ & $2$ &  & $\D_7$ \\  
$w_4$ & $1$ & $2$ & $2$ & $2$ & $2$ & $1$ & $2$ &  & $\D_7$ \\
$w_5$ & $1$ & $2$ & $3$ & $3$ & $2$ & $1$ & $3$ & $d_{Lee}$ & $\D_7$ \\ 
\hline
\end{tabular}
\end{equation*}
\end{table}

We know that $\Gamma(\Z_7,d_1)\simeq \D_7$ and $\Gamma(\Z_7,d_5)\simeq \Sym_7$. For the other groups, we have 
$$\D_7 \subset \Gamma(\Z_7,d_i) = \mathrm{Aut}(Cay(\Z_7,\{a,-a\})) \cap \mathrm{Aut}(Cay(\Z_7,\{b,-b, c, -c\}))$$ 
with $a,b,c \in \{1,2,3\}$ for $i=2,3,4$.
Since 
$\mathrm{Aut}(Cay(\Z_7,\{a,-a\})) = \mathrm{Aut}(C_7) \simeq \D_7$, 
we have that $\Gamma(\Z_7,d_i)\simeq \D_7$ for $i=2,3,4$.
\hfill $\lozenge$ 
\end{exam}

For groups or order $\ge 8$, the number of invariant metrics grows rapidly (as we will see in the next section) and the computation of the metrics and the symmetry groups gets more and more difficult. For this reason, we now study the symmetry groups for the non-abelian groups of order 8, $\D_4$ and $\Q_8$, 
only with the `Lee' metrics, that is those associated to the finest unitary symmetric partition.

\begin{exam} \label{Q8D4}
Consider the quaternion and the dihedral groups $\Q_8=\{\pm 1, \pm i, \pm j, \pm k\}$ with $i^2=j^2=k^2=ijk=1$ and $\D_4= \{ e, \rho, \rho^2, \rho^3, \tau, \tau \rho, \tau \rho^2, \tau \rho^3\}$ where $\rho^4=\tau^2=e$ and $\rho \tau \rho=\tau$.  
Their Lee partitions %(finest unitary symmetric) 
and conjugate Lee partitions are given by 
\begin{align*}
& \Prep_{Lee}(\Q_8) = \{ \{1\}, \{-1\}, \{i,-i\}, \{j,-j\}, \{k,-k\}\} = \Prep_{Lee}^{c}(\Q_8), \\ 
& \Prep_{Lee}(\D_4) = \{\{e\}, \{\rho,\rho^3\}, \{\rho^2\}, \{\tau\}, \{\rho^2\tau\}, \{\rho\tau\}, \{\rho^3\tau\}\}, \\
& \Prep_{Lee}^{c}(\D_4) = \{\{e\},\{\rho,\rho^3\},\{\rho^2\},\{\tau,\rho^2\tau\},\{\rho\tau,\rho^3\tau\}\}. 
\end{align*}
 Note that the fact that the Lee partition and the conjugate Lee partition of $\Q_8$ coincide  implies that any invariant metric of $\Q_8$ is also bi-invariant.

The partition graphs of $\D_4$ and $\Q_8$ for these metrics are respectively given by (black=1, blue=2, red=3, green-dashed=4, magenta-dotted=5, cyan=6): 
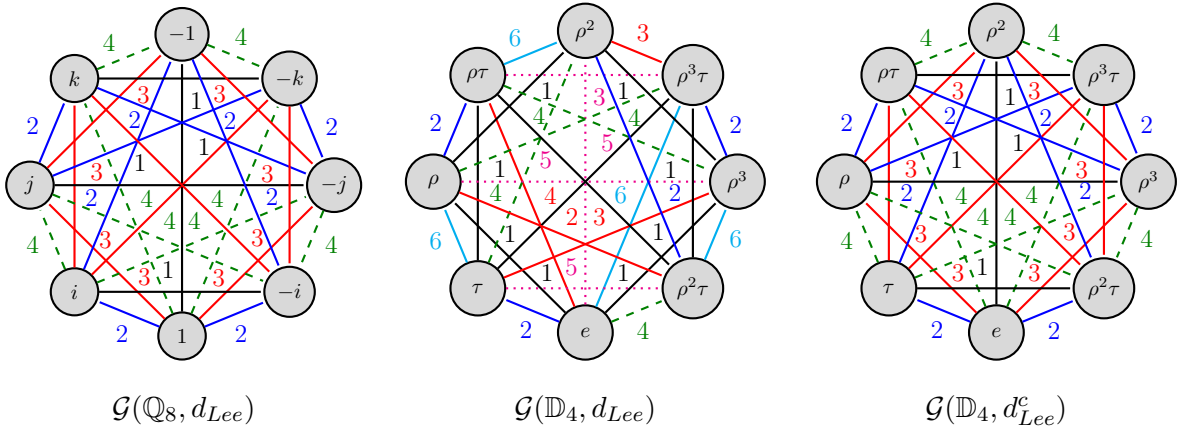
\begin{figure}[H]
	\centering
	\begin{minipage}[t][][t]{0.33\textwidth}
		\begin{tikzpicture}[scale=.8,transform shape,>=stealth',shorten
		>=1pt,auto,node distance=3cm, thick,main node/.style={circle,draw,font=\small}]
		\tikzstyle{every node}=[node distance = 3cm,minimum size=7.8mm,bend angle    = 45,fill          = gray!30]
		
		\node[main node] (1) at (7*360/8-90:2.5) {$i$};
		\node[main node] (2) at (6*360/8-90:2.5) {$j$};
		\node[main node] (3) at (5*360/8-90:2.5) {$k$};
		\node[main node] (4) at (4*360/8-90:2.5) {$-1$};
		\node[main node] (5) at (3*360/8-90:2.5) {$-k$};
		\node[main node] (6) at (2*360/8-90:2.5) {$-j$};
		\node[main node] (7) at (1*360/8-90:2.5) {$-i$};
		\node[main node] (0) at (0*360/8-90:2.5) {$1$};
				
		\path[every node/.style={font=\large,circle}]
		
		(0) edge  node [right,pos=0.66] {1} (4)
		edge [blue]  node[below, pos=0.6] {2} (1)
		edge [blue]  node[below,pos=0.6] {2} (7)
		edge [red]  node[right, pos=0.36] {3} (2)
		edge [red]  node[left, pos=0.36] {3} (6)
		edge [black!50!green,dashed]  node[right, pos=0.54] {4} (3)
		edge [black!50!green,dashed]  node[left, pos=0.55] {4} (5)
		
		(1) edge  node [above,pos=0.42] {1} (7)
		edge [blue]  node[left, pos=0.34] {2} (4)
		edge [red]  node[right, pos=0.58] {3} (3)
		edge [red]  node[left, pos=0.2] {3} (5)
		edge [black!50!green,dashed]  node[left,pos=0.4] {4} (2)
		edge [black!50!green,dashed]  node[above,pos=0.46] {4} (6)
		(2) edge node [above,pos=0.34] {1} (6)
		edge [blue]  node [left, pos=0.6] {2} (3)
		edge [blue]  node [above,pos=0.39] {2} (5)
		edge [red]  node[right,pos=0.64] {3} (4)
		edge [black!50!green,dashed]  node[above,pos=0.55] {4} (7)
		(3) edge node [below,pos=0.6] {1} (5) 
		edge [blue]  node [above, pos=0.63] {2} (6)
		edge [red]  node [right,pos=0.8] {3} (7)
		edge [black!50!green,dashed]  node [above,pos=0.2] {4} (4)
		(4) edge [black!50!green,dashed] node [above,pos=0.6] {4} (5)
		edge [blue] node [right,pos=0.66] {2} (7)
		edge [red] node [left,pos=0.36] {3} (6)
		(5) edge [blue]  node [right,pos=0.4] {2} (6)
		edge [red] node [left,pos=0.42] {3} (7)
		(6) edge [black!50!green,dashed]  node [right,pos=0.6] {4} (7);
	
		\end{tikzpicture}
		\centering
		\captionsetup{justification=centering,margin=0.1cm}
		\caption*{$\mathcal{G}(\Q_8, d_{Lee})$}
	\end{minipage}%
%\hspace{1cm} 
	\centering
	\begin{minipage}[t][][t]{0.33\textwidth}
		\begin{tikzpicture}[scale=.8,transform shape,>=stealth',shorten
		>=1pt,auto,node distance=3cm, thick,main node/.style={circle,draw,font=\small}]
		\tikzstyle{every node}=[node distance = 3cm, minimum size=9mm, bend angle = 45, fill = gray!30]
		
		\node[main node] (1) at (7*360/8-90:2.5) {$\tau$};
		\node[main node] (2) at (6*360/8-90:2.5) {$\rho$};
		\node[main node] (3) at (5*360/8-90:2.5) {$\rho\tau$};
		\node[main node] (4) at (4*360/8-90:2.5) {$\rho^2$};
		\node[main node] (5) at (1*360/8-90:2.5) {$\rho^2\tau$};
		\node[main node] (6) at (2*360/8-90:2.5) {$\rho^3$};
		\node[main node] (7) at (3*360/8-90:2.5) {$\rho^3\tau$};
		\node[main node] (0) at (0*360/8-90:2.5) {$e$};
		
		\path[every node/.style={font=\large,circle}]
		
		(0) edge  node [right,pos=0.36] {1} (2)
		edge  node [left,pos=0.36] {1} (6)
		edge [blue]  node[below, pos=0.6] {2} (1)
		edge [red]  node[right, pos=0.54] {4} (3)
		edge [black!50!green,dashed]  node[below,pos=0.6] {4} (5)
		edge [magenta,dotted]  node[right, pos=0.66] {5} (4)
		edge [cyan]  node[left, pos=0.55] {6} (7)
		(1) edge  node [right, pos=0.58] {1} (3)
		edge  node [left,pos=0.2] {1} (7)
		edge [black!50!green,dashed]  node[left, pos=0.34] {4} (4)
		edge [magenta,dotted]  node[above,pos=0.42] {5} (5)
		edge [cyan]  node[left,pos=0.4] {6} (2)
		edge [red]  node[above,pos=0.46] {3} (6)
		(2) edge  node [right,pos=0.64] {1} (4)
		edge [blue]  node [left, pos=0.6] {2} (3)
		edge [red]  node [above,pos=0.55] {2} (5)
		edge [black!50!green,dashed]  node[above,pos=0.39] {4} (7)
		edge [magenta,dotted]  node[above,pos=0.34] {5} (6)
		(3) edge node [right,pos=0.8] {1} (5) 
		edge [black!50!green,dashed]  node [above, pos=0.63] {4} (6)
		edge [magenta,dotted]  node [below,pos=0.6] {3} (7)
		edge [cyan]  node [above,pos=0.2] {6} (4)
		(4) edge node [left,pos=0.36] {1} (6)
		edge [blue] node [right,pos=0.67] {2} (5)
		edge [red] node [above,pos=0.6] {3} (7)
		(5) edge node [left,pos=0.58] {1} (7)
		edge [cyan]  node [right,pos=0.4] {6} (6)
		(6) edge [blue]  node [right,pos=0.6] {2} (7);
	
		\end{tikzpicture}
		\centering
		\captionsetup{justification=centering,margin=0.1cm}
		\caption*{$\mathcal{G}(\D_4, d_{Lee})$}
	\end{minipage} 
%\hspace{1cm} 
	%\hfill
	\begin{minipage}[t][][t]{0.33\textwidth}
		\begin{tikzpicture}[scale=.8,transform shape,>=stealth',shorten
		>=1pt,auto,node distance=3cm, thick,main node/.style={circle,draw,font=\small}]
		\tikzstyle{every node}=[node distance = 3cm,minimum size=9mm,bend angle    = 45,fill          = gray!30]
		
		\node[main node] (1) at (7*360/8-90:2.5) {$\tau$};
		\node[main node] (2) at (6*360/8-90:2.5) {$\rho$};
		\node[main node] (3) at (5*360/8-90:2.5) {$\rho\tau$};
		\node[main node] (4) at (4*360/8-90:2.5) {$\rho^2$};
		\node[main node] (7) at (1*360/8-90:2.5) {$\rho^2\tau$};
		\node[main node] (6) at (2*360/8-90:2.5) {$\rho^3$};
		\node[main node] (5) at (3*360/8-90:2.5) {$\rho^3\tau$};
		\node[main node] (0) at (0*360/8-90:2.5) {$e$};
		
		\path[every node/.style={font=\large,circle}]
		
		(0) edge  node [right,pos=0.66] {1} (4)
		edge [blue]  node[below, pos=0.6] {2} (1)
		edge [blue]  node[below,pos=0.6] {2} (7)
		edge [red]  node[right, pos=0.36] {3} (2)
		edge [red]  node[left, pos=0.36] {3} (6)
		edge [black!50!green,dashed]  node[right, pos=0.54] {4} (3)
		edge [black!50!green,dashed]  node[left, pos=0.55] {4} (5)
		
		(1) edge  node [above,pos=0.42] {1} (7)
		edge [blue]  node[left, pos=0.34] {2} (4)
		edge [red]  node[right, pos=0.58] {3} (3)
		edge [red]  node[left, pos=0.2] {3} (5)
		edge [black!50!green,dashed]  node[left,pos=0.4] {4} (2)
		edge [black!50!green,dashed]  node[above,pos=0.46] {4} (6)
		(2) edge node [above,pos=0.34] {1} (6)
		edge [blue]  node [left, pos=0.6] {2} (3)
		edge [blue]  node [above,pos=0.39] {2} (5)
		edge [red]  node[right,pos=0.64] {3} (4)
		edge [black!50!green,dashed]  node[above,pos=0.55] {4} (7)
		(3) edge node [below,pos=0.6] {1} (5) 
		edge [blue]  node [above, pos=0.63] {2} (6)
		edge [red]  node [right,pos=0.8] {3} (7)
		edge [black!50!green,dashed]  node [above,pos=0.2] {4} (4)
		(4) edge [black!50!green,dashed] node [above,pos=0.6] {4} (5)
		edge [blue] node [right,pos=0.66] {2} (7)
		edge [red] node [left,pos=0.36] {3} (6)
		(5) edge [blue]  node [right,pos=0.4] {2} (6)
		edge [red] node [left,pos=0.42] {3} (7)
		(6) edge [black!50!green,dashed]  node [right,pos=0.6] {4} (7);
	
		\end{tikzpicture}
		\centering
		\captionsetup{justification=centering,margin=0.1cm}
		\caption*{$\mathcal{G}(\D_4, d_{Lee}^c)$}
	\end{minipage} 
\end{figure}
\noindent
That is, $\mathcal{G}(\Q_8, d_{Lee}^c) = \mathcal{G}(\D_4, d_{Lee}^c)  = 4K_2 \cup 3(2C_4)$ and 
$\mathcal{G}(\D_4, d_{Lee}) = 6(4K_2) \cup 2C_4$.  By Corollary~\ref{Cay(Cn)}, the groups $\mathcal{G}(G, d_{Lee}^c)$ are intersections of groups isomorphic to $\Z_2 \wr \Sym_4$ and $\D_4 \wr \Z_2$.
Hence, by Example \ref{GDQ} with $n=2$,
for $G=\D_4, \Q_8$ we have 
$$G \rtimes (\Z_2 \times \Z_2) \le \Gamma(G, d_{Lee}^c) \le  \D_4 \wr \Z_2$$
since $G/\Z_2=\Z_2 \times \Z_2$.
Thus, $2^5 \le \#\Gamma(G,d_{Lee}^c) \le 2^7$ for  $G=\D_4, \Q_8$.

We now compute the group $\Gamma(G, d_{Lee}^c)$ for $G=\D_4,\Q_8$. Notice that $\Gamma(G,d_{Lee}) \le \Sym_8$. 
First we find the order. Using that $\Gamma(G, d_{Lee}^c)$ acts on $G$ then we have
$$|\Gamma(G, d_{Lee}^c)|=|G||\Gamma_e|,$$ 
where $\Gamma_e$ is the stabilizer of $\Gamma(G, d_{Lee}^c)$. 
From the graphs, one can see that $\Gamma_e= \Z_2 \times \Z_2 \times \Z_2$. In fact, the group $\Gamma_e$ for $G=\Q_8$ is determined by the 3 independent permutations sending $i \mapsto -i$, $j \mapsto -j$ and $k \mapsto -k$. Similarly, the group $\Gamma_e$ for $G=\D_4$ is given by  $\tau \mapsto \rho^2 \tau$, $\rho\mapsto \rho^3$ and $\rho \tau \mapsto \rho^3 \tau$.
Thus, we get that $|\Gamma(G, d_{Lee})^c|=2^6$.  

There are 267 groups of order 64, out of which 256 are non-abelian ones. 
Intuitively, this symmetry group should be a wreath product. But there are only 2 wreath products of order 64, $\Z_2 \wr \Z_4$ and $\Z_2 \wr (\Z_2 \times \Z_2)$. Consider the group 
$$\Z_2 \wr (\Z_2 \times \Z_2) = \Z_2^4 \rtimes_\varphi (\Z_2 \times \Z_2)$$
with $\Z_2^4 = \langle \varphi_1, \varphi_i, \varphi_j, \varphi_k \rangle$ and $\Z_2^2 = \langle \varphi_{ij}, \varphi_{ik} \rangle$ where the involution $\varphi_a$ is determined by $a \leftrightarrow -a$ for $a\in \{1,i,j,k\}$, $\varphi_{ij}$ is determined by $i \leftrightarrow j$, $-i \leftrightarrow -j$, $1 \leftrightarrow -1$, $k \leftrightarrow -k$ and $\varphi_{ik}$ is determined by $i \leftrightarrow k$, $-i \leftrightarrow -k$, $1 \leftrightarrow -1$, $j \leftrightarrow -j$.  
We have that $\Z_2^4 \rtimes (\Z_2 \times \Z_2)$ is a subgroup of $\Gamma(G, d_{Lee}^c)$, which has the same cardinality, and hence $\Gamma(G, d_{Lee}^c) = \Z_2 \wr (\Z_2 \times \Z_2)$.
\hfill $\lozenge$
\end{exam}

\section{The number of invariant metrics on finite groups} \label{S5}
How many equivalence $\Prep$-classes of metrics are there on a finite group $G$?
We will denote the number of non-equivalent invariant metrics of $G$ by $M(G)$, that is 
$$M(G)=\#(\widetilde{\mathcal{M}}(G)).$$ 
To answer this question we use Bell numbers. We recall that the $n$-th Bell number $B_n$ counts the number of partitions of a set of $n$ elements. 
The first Bell numbers are 
\begin{equation*}
\label{bells}
B_0=1, \: B_1=1, \: B_2=2, \: B_3=5, \: B_4=15, \: B_5=52, \: B_6=203, \: B_7=877, 
\: B_8=4.140.
\end{equation*}

We now give the number of invariant metrics on a group $G$. We denote by 
$k_2(G)$ the number of elements of order $\le 2$ of $G$, 
that is 
\begin{equation} \label{k2G}
k_2(G)=\#\{ x\in G : x^2 = e \}.
\end{equation}

\begin{prop} \label{CardMG}
	If $G$ is a finite group of order $n$ then $M(G)= B_{k(G)}$
	where 
	\begin{equation} \label{kG2} 
	k(G) = \tfrac 12 \{ n+k_2(G)\}-1.
	\end{equation}
	In particular, if $n$ is odd then $M(G) = B_{\frac{n-1}2}$.
\end{prop}

\begin{proof}
	In view of the correspondence \eqref{correspondencia}, the number of $\Prep$-classes of metrics of $G$ is given by the number of unitary symmetric partitions of $G$. The identity and the elements of order 2 in $G$ are their own inverses. However, if the element $g\in G$ has order $>2$ then $g$ and $g^{-1}$ must be in the same part, since any partition is symmetric. Thus, it is clear that the number of symmetric partitions is given by the Bell number $B_{k(G)}$ with 
	%associated to the number %$k(G)$ defined in \eqref{kG}. 
	\begin{equation} \label{kG} 
	k(G) = \tfrac{1}{2} \# \big\{ g\in G : ord(g)>2 \big\} + \#\big\{ g\in G : ord(g)= 2 \big\}.
	\end{equation}
	Since $n = k_2(G) + \# \big\{ g\in G : ord(g)>2 \big\}$, we clearly get that 
	$k(G)=\tfrac 12 \{ n-k_2(G) \} + k_2(G)-1$, from where \eqref{kG2} follows. 
	
	Now, if $G$ is a group of odd order then it has no elements of order 2, hence $k_2(G)=1$ and $k(G)=\frac{n-1}2$.
\end{proof}

Notice that if $G$ has even order then it has an odd number of elements of order 2, since
we have the disjoint union $G=\{e\} \cup \{ \text{order $2$ elements} \} \cup \{g,g^{-1}, h,h^{-1}, \ldots \}$. 
Hence, $k(G)$ is a well defined positive integer.

For a group $G$ with (not so) big cardinality $n$ it could be difficult to compute $k(G)$ by hand. In this case, we have the obvious (not so sharp) bound 
\begin{equation} \label{MGBn-1}
M(G) \le B_{n-1}
\end{equation}
since $1\le k(G) \le n-1$ by definition. That is, every group has at least one invariant metric, the Hamming metric. 

\begin{exam} \label{coro k(G)}
	We now give the number of invariant metrics in some easy cases.
	
	\noindent ($i$) 
	If $\Z_n$ denote the cyclic group of order $n$ then $M(\Z_n) = B_{[\frac n2]}$. 
	In fact, we clearly have 
	$k(\Z_{2n}) = \frac12 (2n+2)-1=n$, since there is only one element of order $2$, and $k(\Z_{2n+1}) = \frac 12((2n+1)+1)-1=n$, since there is no elements of order 2. 
	
	\noindent ($ii$) 
	For any $k\ge 1$, since every nonzero element is of order 2, we have $k(\Z_2^k)=k_2(\Z_2^k)$ and thus $M(\Z_2^k) = B_{2^k-1}$.
	Moreover, the equality in \eqref{MGBn-1} holds if and only if $G=\Z_2^k$ for some $k$. 
	\hfill $\lozenge$
\end{exam}

For an abelian group $G$ we can give the precise number of invariant metrics. To this end we will need the following basic observation. 
\begin{lema} \label{lemHxK}
	If $G,K$ are groups then 
	\begin{equation} \label{HxK}
	k(G \times K) = \tfrac12 \{|G||K| + k_2(G)k_2(K)\}-1.
	\end{equation}
In particular, $k(G \times \Z_2^r) = 2^r (k(G)+1)-1$ for every $r\in \N$.
\end{lema}

\begin{proof}
	Just note that $k_2(G \times K) = \# \{(x,y) \in G \times K : (x,y)^2=(e,e) \} = k_2(G) k_2(K)$.
The remaining assertion is clear since $|\Z_2^r|=k_2(\Z_2^r)= 2^r$.
\end{proof}

Using the previous proposition we will obtain closed expressions for the number of metrics of certain families of groups. 
For abelian groups we have the following result. 
\begin{prop} \label{kGab}
	Let $G$ be an abelian group of order $n$. Then, the number of invariant metrics on $G$ is given by 
	$M(G) = B_{\frac n2+2^{s-1}-1}$ 
	where $s$ is the number of cyclic factors of order some power of 2 in the canonical product decomposition of $G$ into cyclic groups. 
\end{prop}

\begin{proof}
Since $G$ is abelian it has a decomposition as product of cyclic groups 
$$G \simeq \Z_{2^{m_1}} \times \cdots \times \Z_{2^{m_s}} \times \Z_{p_1^{r_1}}^{e_1}
\times \cdots \times \Z_{p_s^{r_t}}^{e_t}$$
where $m_1,\ldots,m_s, r_1,\ldots, r_t, e_1, \ldots, e_t \in \N$
and $p_1, \ldots,p_t$ are odd prime numbers (not necessarily distinct). 
Only the factors of even order have elements of order 2 (one per factor). Hence, using \eqref{HxK}, we see that 
$$k_2(G)=k_2(\Z_{2^{m_1}} \times \cdots \times \Z_{2^{m_s}})  = k_2(\Z_{2^{m_1}}) \cdots k_2(\Z_{2^{m_s}}) = 2^s.$$ 
We have $k(G)=\frac{n+2^s}2-1= \frac n2+2^{s-1}-1$, by \eqref{kG2}, and the result follows by Proposition~\ref{CardMG}.
\end{proof}
%Notice that if $n$ is odd, the result is in coincidence with Proposition \ref{}.

We now give the number of invariant metrics for some families of non-abelian groups. We first consider dihedral, dicyclic and quasidihedral groups and later symmetric and alternating groups. 

\begin{prop} \label{prop Ms}
The number of invariant metrics of the dihedral groups $\D_n$, dicyclic groups $\Q_{4n}$ and quasidihedral groups $Q\D_n^\pm$ are given by: 
	\begin{enumerate}[$(a)$]
		\item $M(\D_{2k+1}) = B_{3k+1}$ for $k\ge 1$ and $M(\D_{2k}) = B_{3k}$ for $k\ge 2$. \sk
		
		\item $M(\Q_{4n})=B_{2n}$ for $n\ge 2$. \sk 
		
		\item $M(Q\D_n^-)=B_{5 \cdot 2^{n-3}}$ and $M(Q\D_n^+)=B_{2^{n-1}+1}$ for $n\ge 4$.
	\end{enumerate}
\end{prop}

\begin{proof} By Proposition \ref{CardMG}, $M(G)=B_{k(G)}$, so we will compute the numbers $k(G)$ in each case.
	
	\noindent ($a$) \label{exam Dn}
	For $n\ge 3$, the dihedral group $\D_n$ is the group of symmetries of a regular $n$-gon and hence has order $2n$.
	One familiar presentation is 
	\begin{equation} \label{Dn def}
	\mathbb{D}_n = \langle a,b : a^n=b^2=e, bab=a^{-1} \rangle.
	\end{equation}
	Notice that if $n$ is odd then $\D_n$ has $n$ elements of order 2 (the reflections) while if $n$ is even $\D_n$ has $n+1$ elements of order 2, the $n$ reflections plus the half turn rotation. Thus, if $n=2k$ we have  
	$k(\D_{n})=\frac{2n+n+2}2 -1=3k$ and if $n=2k+1$ we have $k(\D_{n})=\frac{2n+n+1}2-1 =3k+1$.
	
	\noindent ($b$) \label{exam Dic}
	The dicyclic groups $\Q_{4n}$ of order $4n$, also denoted $\mathrm{Dic}_n$, are defined for every $n\ge 2$ by  
	\begin{equation} \label{dycs rels} 
	\Q_{4n} =  \langle a,b : a^{2n}=e, b^2=a^n, b^{-1}ab=a^{-1} \rangle.
	\end{equation}	
	When $n$ is a power of $2$ these groups are the generalized quaternions $\Q_{2^m}$ for $m\ge 4$, in particular the quaternion group $\Q_8$ for $n=2$.  
	Notice that $\Q_{4n}$ has only one element of order 2, namely $a^n$. It is easy to see that every element in the group can be uniquely written in the form $a^kb^j$, where $0 \le k < 2n$ and 
	$j = 0,1$. It is clear that $a^k a^k=e$ if and only if $k=n$. Also, in general $(a^k b)(a^m b)=a^{k-m+n}$, so we have $a^kba^kb=a^n\ne e$. Thus, $k(\Q_{4n}) = \tfrac 12 (2n+2)-1=2n$. 
	
	\noindent ($c$) \label{QDn}
	The quasidihedral groups $Q\D_n^{\pm}$ of order $2^n$ (also called semidihedral and denoted $S\D_{2^n}^{\pm}$) are defined for every $n\ge 4$ by  
	\begin{equation} \label{qdns rels} 
	Q\D_n^\pm =  \langle x,y : x^{2^{n-1}}=y^{2}=e, yxy=x^{2^{n-2} \pm 1}\rangle.
	\end{equation}
	The group $Q\D_n^-$ has $2^{n-2}+1$ elements of order 2 while the group $Q\D_n^+$ has 3 elements of order 2.
	Therefore, $k(Q\D_n^-)=\tfrac 12(2^n+2^{n-2}+2)-1= 5 \cdot 2^{n-3}$ and 
	$k(Q\D_n^+)=\tfrac 12(2^n+4)-1=2^{n-1}+1$, and ($c$) holds.
	
	Let us check the number of elements of order 2. First note that since $y$ has order 2, the elements of these groups are explicitly given by 
	\begin{equation} \label{qdns els}
	Q\D_n^\pm = \{e,x,x^2,\ldots,x^{2^{n-2}},\ldots,x^{2^{n-1}-1}, y, yx, yx^2, \ldots,yx^{2^{n-2}},\ldots,yx^{2^{n-1}-1} \} \\ 
	\end{equation}
	Clearly, in both groups, $y$ is an element of order 2 and $x^{2^{n-2}}$ is the only element of order 2 in the cyclic subgroup $\langle x \rangle$. We now study how many elements of order 2 are there in the coset $y\langle x \rangle$ in each quasidihedral group. Suppose $yx^k$, with $k=1,\ldots,2^{n-1}-1$, is of order 2. Then, by \eqref{qdns rels} we have 
	$$e=yx^k y x^k=(yx^ky^{-1})x^k= (yxy)^kx^k = x^{k(2^{n-2} \pm 1)} x^k.$$
	Thus, $x^{k2^{n-2}}=e$ in $Q\D_n^-$ and $x^{k(2^{n-2}+2)}=e$ in $Q\D_n^+$. Now, $x^{k2^{n-2}}$ is of order 2 in $Q\D_n^-$ if and only if $k2^{n-2}\equiv 0 \pmod{2^{n-1}}$, and this happens for ever $k$ even. There are $2^{n-2}-1$ such elements and hence $2^{n-2}-1+2=2^{n-2}+1$ elements of order $2$ in $Q\D_n^-$. Similarly, $x^{k(2^{n-2}+2)}$ is of order 2 in $Q\D_n^+$ if and only if $k(2^{n-2}+2) \equiv 0 \pmod{2^{n-1}}$, and this happens only for $k=2^{n-1}$. Thus, there are only 3 elements of order 2 in $Q\D_n^+$.
\end{proof}

We now study the number of invariant metrics in the symmetric group $\Sym_n$ for $n\ge 3$ and the alternating subgroup $\mathbb{A}_n$ for $n\ge 4$ (these groups for smaller values of $n$ are abelian).
In this case, we will not have close expressions, but rather expressions involving the following numbers
\begin{equation} \label{sn}
s_n = \sum_{0 \le k \le \lfloor \frac n2 \rfloor} \tfrac{1}{2^k \, k! \, (n-2k)!} 
\qquad \text{and} \qquad a_n = \sum_{0 \le k \text{ even } \le \lfloor \frac n2 \rfloor} \tfrac{1}{2^k \, k! \, (n-2k)!}.
\end{equation}

\begin{prop} \label{Sn An}
	The number of invariant metrics for $\Sym_n$ and $\mathbb{A}_n$ are given by
	$$M(\Sym_n)= B_{\frac 12 n!(s_n+1)-1} \quad \text{for $n\ge 3$} \qquad \text{and} \qquad 
	M(\mathbb{A}_n) = B_{\frac 12 n!(a_n+\frac 12)-1} \quad \text{for $n\ge 4$,}$$
	where $s_n$ and $a_n$ are as in \eqref{sn}.
\end{prop}

\begin{proof} \label{exam Sn}
Any element in $\Sym_n$ is a product of transpositions. The elements of order 2 are just the products of disjoint transpositions, which in this case commute with each other. We have to count all possible products of disjoint transpositions. 
For instance, the elements of order 2 in $\Sym_6$ are of the form $(i_1 i_2)$, $(i_1 i_2)(i_3 i_4)$ or $(i_1 i_2)(i_3 i_4)(i_5 i_6)$ with 
$i_1, i_2, i_3, i_4, i_5, i_6$ all different elements in $\{1,2,3,4,5,6\}$; and there are $\tbinom 62 + \tfrac{1}{2!}\tbinom 62 \tbinom 42 + \tfrac{1}{3!}\tbinom 62 \tbinom 42 \tbinom 22$ of them. Thus, one can check that the number of elements of order $\le 2$ in $\Sym_n$ for $n\ge 3$ is given by
\begin{equation} \label{k2Sn}
k_2(\Sym_n) = 1 + \sum_{0 \le k \le t-1}  \tfrac{1}{(k+1)!} \tbinom{n}{2} \tbinom{n-2}{2} \tbinom{n-4}{2} \cdots \tbinom{n-2k}{2}, 
\qquad n=2t, 2t+1,
\end{equation}
which can be written more succinctly as 
$k_2(\Sym_n) = n! s_n$. Hence, $M(\Sym_n)=B_{k(\Sym_n)}$ where we have
$k(\Sym_n) = \tfrac 12 \{ n!+k_2(\Sym_n) \}-1= \tfrac 12 n!(s_n+1)-1$, as desired.

Assuming $n\ge 4$, the elements of order 2 in $\mathbb{A}_n$ are the products of an even number of disjoint transpositions. 
It is clear that $k_2(\mathbb{A}_n)$ is obtained from expression \eqref{k2Sn} but summing only over odd indices. That is,   
as before $k_2(\mathbb{A}_n) = n! a_n$. 
Hence, $M(\mathbb{A}_n)=B_{k(\mathbb{A}_n)}$ where $k(\mathbb{A}_n)=\tfrac 12 (\tfrac{n!}2+k_2(\mathbb{A}_n))-1= \frac{n!}2(a_n+\frac 12)-1$, and the result follows.
\end{proof}

\subsubsection*{Semidirect products}
Recall that if $G$ and $H$ are groups and $\varphi: H \rightarrow \mathrm{Aut}(G)$ is a homomorphism we have the semidirect product 
$G \rtimes_\varphi H$ which is the set $G\times H$ with the product given by $(g_1,h_1)(g_2,h_2)= (g_1 \varphi_h(g_2), h_1 h_2)$, where $\varphi_h$ stands for the morphism $\varphi(h)$.
When $\varphi$ is natural or the only one we omit it from the notation.

It is difficult to count the number of metrics on semidirect products $G \rtimes_\varphi H$, unless we know the function $\varphi$ explicitly. In general we can only estimate this number, although in some particular cases we can compute it explicitly. By \eqref{kG2} we have 
\begin{equation} \label{k2GsxH}
k(G \rtimes_\varphi H) = \tfrac 12 \{ |G||H| + k_2(G \rtimes_\varphi H) \} -1.
\end{equation}
Now, since $(g,h)^2 = (g\varphi_h(g),h^2)$ we obtain 
\begin{align*}
k_2 (G \rtimes_\varphi H) &= \#\{(g,h) \in G \times H : (g \varphi_h(g), h^2) =(e_G,e_H) \}    \\
						 &= \sum_{h \in H, \,h^2=e_H} \#\{ g\in G : g\varphi_h(g) = e_G \}.
\end{align*}
Separating the contribution of the identity element $e_H$, since $\varphi_{e_H}= id$, we have
\begin{equation} \label{k semid}
k_2(G \rtimes_\varphi H) = k_2(G) + \sum_{h\in H^\times, \, h^2=e_H} 
\#\{ g\in G : g\varphi_h(g) = e_G \}.
\end{equation}
Since $0\le \#\{ g\in G : \varphi_h(g) = g^{-1} \} \le |G|$ we have that
$$k_2(G) \le k_2(G \rtimes_\varphi H) \le k_2(G) + |G|(k_2(H)-1).$$

Applying these inequalities in \eqref{k2GsxH}, after some trivial computations we obtain
\begin{equation} \label{kGsxH bounds}
\tfrac 12 |G|(|H|-1) + k(G) \le k(G \rtimes_\varphi H) \le k(H)|G| + k(G).
\end{equation}
Notice that both equalities hold if and only if $\tfrac 12 (|H|-1)=k(H)$, that is if and only if $k_2(H)=1$ which happens if and only if $H$ is of odd order. Thus, if $H$ is of odd order then we have 
\begin{equation} \label{kGsxH eq}
k(G \rtimes H )= \tfrac 12 |G|(|H|-1) +k(G).
\end{equation}
In particular, if $G$ is also of odd order, then 
$k(G \rtimes H) = \tfrac 12 (|G||H|)-1 = k(G \times H)$.

We now consider the case of $G \rtimes_\varphi \Z_2$ where $G$ is an abelian group. By \eqref{k semid}, we have 
\begin{equation} \label{k semidZ2}
	k_2(G \rtimes_\varphi \Z_2) = k_2(G) + \#\{ g\in G : g\varphi_1(g) = e_G \}.
\end{equation}
In the particular case of the generalized dihedral group $\D(G)=G\rtimes \Z_2$ of an abelian group $G$, with $G \ne \Z_2^k$
since otherwise $\D(G)=G$, where $\varphi_1$ is the inversion (see Definition \ref{Dihedral}) we have  
$k_2(G\rtimes \Z_2)= k_2(G)+|G|$. Thus, 
\begin{equation} \label{DG}
k(\D(G)) = \tfrac 12 (3|G|+k_2(G))-1
\end{equation}
for any abelian group $G \ne \Z_2^k$. If $G=\Z_n$ we get the dihedral group $\D_n$ and \eqref{DG} coincides with ($a$) in Proposition~\ref{prop Ms}.

Notice also that if we take $G=\Z_{2^{n-1}}$ and consider the semidirect products $G\rtimes_{\varphi^\pm} \Z_2$ where the morphisms $\varphi^\pm : \Z_2 \rightarrow \mathrm{Aut}(\Z_{2^{n-1}})$ are given by $x\mapsto x^{2^{n-2}\pm 1}$ we get the quasidihedral groups $Q\D_n^\pm$ and that \eqref{DG} in this case coincides with ($c$) in Proposition~\ref{prop Ms}.

\section{Bi-invariant metrics on groups} \label{S5b}
We now study the number of bi-invariant metrics on groups. 
Let $G$ be a group, not necessarily finite. We will denote the number (if it is finite) of bi-invariant metrics on $G$ by $M^*(G)$. Clearly, we have
\begin{equation} \label{MG<=M*G}
M^*(G) \le M(G).
\end{equation}

Let $C(G) = \{C_g:g\in G\}$ be the set of conjugacy classes of $G$ where 
$$C_g=\{hgh^{-1}: h \in G\}$$ 
is the conjugacy class of an element $g \in G$.
An element $g$ of $G$ is called \textit{real} if it is conjugate to its inverse, i.e.\@ if there is some $h\in G$ such that 
$hgh^{-1}=g^{-1}$. 
A conjugacy class $C_g$ is \textit{real} if $C_g=C_{g^{-1}}$.
We denote by $c(G)=\#C(G)$ the number of conjugacy classes of $G$ and by $c_2(G)$ the number of real conjugacy classes in $G$.
From now on, when we write $c(G)$ or $c_2(G)$ we will assume that $G$ has a finite number of conjugacy classes. If $P_0,P_1,\ldots,P_s$ is a conjugate partition of $G$ then $g\in P_i$ implies that $C_g \subset P_i$. 

We are now in a position to state the next result.
\begin{prop} \label{No bi-inv}
	If $G$ is a finite group, $M^*(G)=B_{k^*(G)}$ with
	\begin{equation} \label{kbiG}
	k^*(G) = \tfrac 12 \{ c(G) + c_2(G)\} -1. 
	\end{equation}
\end{prop}

\begin{proof}
	Consider the finest unitary symmetric conjugate partition $\mathcal{P}=\{P_0, P_1, \ldots,P_s\}$ of $G$. Every unitary symmetric conjugate partition of $G$ is obtained by taking the unions of some parts of $\mathcal{P}$. Hence, the number of unitary symmetric conjugate partitions of $G$ is given by the number of partitions of $\{1,2,\ldots,s\}$. We proceed similarly as in the proof of Proposition \ref{CardMG}. We have 
	\begin{equation} \label{k*G} 
	k^*(G) = \tfrac{1}{2} \# \big\{ g\in G : C_g \ne C_{g^{-1}} \big\} + \#\big\{ g\in G : C_g = C_{g^{-1}} \big\}-1.
	\end{equation}
	Since $c(G) = c_2(G) + \#\{ g\in G : C_g \ne C_{g^{-1}} \}$, we have that 
	$k^*(G) = \tfrac 12 \{ c(G)-c_2(G) \} + c_2(G) -1$, from where \eqref{k*G} follows.
	Thus, we finally obtain that $M^*(G)=B_{k^*(G)}$. 
\end{proof}

We know that if $G$ is a finite abelian group then right-invariant, left-invariant and bi-invariant metrics are the same. We now check that 
$M(G)=M^*(G)$, i.e.\@ equality holds in \eqref{MG<=M*G}.
If $G$ is abelian of order $n$, then $c(G)=n$ (since $C_g(G)=\{g\}$) for every $g\in G$) and $c_2(G)=1+k_2(G)$ since the only real elements are the identity and the involutions (order 2 elements). Thus, from \eqref{kG2} and \eqref{kbiG} we have 
$$k^*(G)=\tfrac{n+c_2(G)}2-1 = \tfrac{n-1+k_2(G)}2 = k(G)$$ 
as we wanted to show.

\begin{rem} \label{rem bi odd}
	If $|G|$ is odd, then it has no real elements other than the identity. In fact, if $x\in G$ is real then $yxy^{-1}=x^{-1}$ for some $y\in G$. Hence $yx^ky^{-1}=x^{-k}$ for every $k$. Since $x$ has finite order we have $x^{-1}=x^k$ for some $k$ and thus we have $yx^{-1}y^{-1}=x$. This implies that $y^2xy^{-2}=yx^{-1}y^{-1} =x$, so $x$ and $y^2$ commute, but since both $x,y$ have odd order then they must also commute.
	Thus, in this case we have that 
	$$k^*(G)=\tfrac 12 (c(G)-1).$$
	For instance, the smallest non-abelian group of odd order is $G=\Z_7 \rtimes \Z_3$, which has 5 conjugacy classes. Then $M(G)=B_{10}=115.975$ and $M^*(G)=B_2=2$. 
\end{rem}

\subsubsection*{Ambivalent groups}
If every element of $G$ is real, the group $G$ is called \textit{ambivalent}.
In this case $c(G)=c_2(G)$ and, hence 
\begin{equation} \label{kGcG}
k^*(G)= c(G)-1
\end{equation}
for ambivalent groups $G$.
By the previous remark, $G$ must have even order. 
Some families of groups, such as dihedral, generalized quaternions or symmetric, are known to be ambivalent.

Now we give the analogous of Proposition \ref{prop Ms} for bi-invariant metrics. 
\begin{prop} \label{prop M*s}
	The number of bi-invariant metrics of the dihedral groups $\D_n$, dicyclic groups $\Q_{4n}$ and quasidihedral groups $Q\D_n^\pm$ are given by: 
	\begin{enumerate}[$(a)$]
		\item $M^*(\D_{2k+1}) = B_{k+1}$ for $k\ge 1$ and $M(\D_{2k}) = B_{k+2}$ for $k\ge 2$. \sk
		
		\item $M^*(\Q_{4n}) = B_{n+2}$ for $n$ even and $M^*(\Q_{4n}) =B_{n+1}$ for $n$ odd. \sk 
		
		\item $M^*(Q\D_n^-)=B_{3 \cdot 2^{n-4}+2}$ and $M(Q\D_n^+)=B_{5\cdot 2^{n-4}+1}$ for $n\ge 4$.
	\end{enumerate}
\end{prop}

\begin{proof} \label{ambi}
($a$) 
	It is known that dihedral groups are ambivalent and that $c(\D_n)=\frac{n+3}2$ if $n$ is odd and $c(\D_n)=\frac{n+6}2$ if $n$ is even. In fact, the conjugacy classes are given by $\{e\}$, $\{a,a^{-1}\}$, $\{a^2,a^{-2}\},\ldots, \{a^{\frac{n-1}2},a^{-\frac{n-1}2} \}$ and $\{b,ab, a^2b, \ldots, a^{n-1}b \}$ for $n$ odd and  
	$\{e\}$, $\{a^{\frac n2}\}$, $\{a,a^{-1}\}$, \linebreak $\{a^2,a^{-2}\}, \ldots, \{a^{\frac{n-2}2},a^{-\frac{n-2}2} \}$, 
	$\{b,a^2b, a^4b, \ldots, a^{n-2}b \}$ and $\{ab,a^3b, a^5b, \ldots, a^{n-1}b \}$ for $n$ even.
	Hence, by \eqref{kGcG} we have 
	\begin{equation} \label{kGDn}
	k^*(\D_n) = \begin{cases} 
	\frac{n+1}2 & \qquad \text{if $n$ is odd}, \\[1mm]
	\frac{n+4}2 & \qquad \text{if $n$ is even}. 
	\end{cases}
	\end{equation}
	Thus, $M^*(\D_n)=B_{\frac{n+1}2}$ for $n$ odd and $M^*(\D_n)=B_{\frac{n+4}2}$ for $n$ even.

\noindent ($b$)	
For every $n\ge 2$ we have the dicyclic groups $\Q_{4n}$ of order $4n$ as defined in \eqref{dycs rels}. 
We have that $c(\Q_{4n}) = n+3$. In fact, the conjugacy classes of $\Q_{4n}$ are 
$C_e=\{e\}$, $C_{a^n}=\{a^n\}$, $C_{a^k}=\{a^k, a^{2n-k}\}$ for $k=1,\ldots, n-1$, $C_b=\{b,a^2b, a^4b, \ldots, a^{2n-2}b\}$ and $C_{ab}=\{ab,a^3b, a^5b, \ldots,a^{2n-1}b\}$.
The $n+1$ classes $C_e$, $C_{a}$, $C_{a^2}, \ldots, C_{a^n}$ are clearly real.
The class $C_{b}$ and $C_{ab}$ are real if and only if $n$ is even. 
Indeed, $(a^kb)(a^mb)=a^{k-m+n}$ hence $(a^kb)^{-1}=a^{k-n}b$. Therefore, if $n$ is even, $k-n$ has the same parity as $n$, and thus the inverse $(a^kb)^{-1}$ is in the same class as $a^kb$. In this way, $c_2(\Q_{4n})=n+3$ if $n$ is even and $c_2(\Q_{4n})=n+1$ if $n$ is odd.
If $n$ is even, then $\Q_{4n}$ is ambivalent and $k^*(\Q_{4n})=c(\Q_{4n})-1$. If $n$ is odd then 
we have $k^*(\Q_{4n})=\tfrac 12\{ c(\Q_{4n})+c_2(\Q_{4n})\}-1$.
Finally, we have obtained that 
\begin{equation} \label{kGQn}
	k^*(\Q_{4n})= \begin{cases} 
	n+2 & \qquad \text{if $n$ is even}, \\[.5mm] n+1 & \qquad \text{if $n$ is odd}, 
	\end{cases}
\end{equation} 
for every $n\ge 1$. 

\goodbreak 

\noindent ($c$) We study the conjugacy classes of the groups $Q\D_n^{\pm}$ simultaneously, using the presentations \eqref{qdns rels}. 
By \eqref{qdns els}, the elements of $Q\D_n^{\pm}$ are of the form $x^k$ or $yx^k$ with $k=0,\ldots, 2^{n-1}-1$. 
The conjugates of the elements $x^k$ are $x^\ell x^k x^{-\ell}=x^k$ or $(yx^\ell)x^k (yx^\ell)^{-1}= yx^\ell x^k x^{-\ell} y = x^{k(2^{n-2}\pm 1)}$ for any $\ell$, where $\pm$ corresponds to the group $Q\D_n^\pm$. 
Hence, we have the conjugacy classes $C_e=\{e\}$ and $C_{x^k} = \{x^k, x^{k(2^{n-2}\pm 1)}\}$ for $k=0,\ldots, 2^{n-1}-1$.
The conjugates of $yx^k$ are of the form
\begin{gather*}
x^\ell (yx^k) x^{-\ell} = y(yx^\ell y) x^{k-\ell} = y x^{\ell (2^{n-2}\pm 1)} x^{k-\ell} = (yx^k) x^{\ell(2^{n-2}\pm 1-1)}, \\
(yx^\ell) (yx^k) (yx^{\ell})^{-1} = yx^\ell (y x^{k-\ell}y^{-1}) = 
y x^\ell x^{(k-\ell) (2^{n-2}\pm 1)}  = (yx^k) x^{(k-\ell)(2^{n-2}\pm 1-1)}.
\end{gather*}
%\end{eqnarray*}
Thus, $C_{yx^k} = \{ yx^k x^{\ell(2^{n-2}\pm 1-1)}, yx^k x^{(k-\ell)(2^{n-2}\pm 1-1)} \}_{\ell\ge 0}$ for $k=0,\ldots, 2^{n-1}-1$.
However, note that $C_{yx^k} = C_y \cdot x^k$ for every $k$ and that 
$C_y = \{ yx^{\ell(2^{n-2}\pm 1-1)}, yx^{k(2^{n-2}\pm 1-1)} \}_{\ell \ge 0}$. 

Summing up, the conjugacy classes of $Q\D_n^-$ (1st row) and $Q\D_n^+$ (2nd row) are %respectively 
given by
\begin{alignat*}{3}
& C_e=\{e\}, \qquad && C_{x^k} = \{x^k, x^{k(2^{n-2} - 1)}\}, \qquad && 
C_{yx^k} = \{ yx^k x^{\ell(2^{n-2}-2)}, yx^k x^{(k-\ell)(2^{n-2}-2)} \}_{\ell \ge 0}, \\  
& C_e=\{e\}, \qquad && C_{x^k} = \{x^k, x^{k(2^{n-2} + 1)}\}, \qquad && 
C_{yx^k} = \{ yx^{\ell(2^{n-2})}, yx^{k2^{n-2}} \}_{\ell \ge 0},
\end{alignat*}
for $k=0,\ldots, 2^{n-1}-1$ (there are some repetitions).

For instance, for $n=4$ the conjugacy classes (real classes are underlined) 
of $Q\D_4^-$ are $\underline{\{e\}}$, $\{x,x^3\}$, \underline{$\{x^2, x^6\}$}, \underline{$\{x^4\}$}, $\{x^5, x^7\}$, \underline{$\{y,yx^2,yx^4, yx^6\}$} and  \underline{$\{yx, yx^3, yx^5, yx^7\}$} while the conjugacy classes of $Q\D_4^+$ are 
\underline{$\{e\}$}, $\{x,x^5\}$, $\{x^2\}$, $\{x^3, x^7\}$, \underline{$\{x^4\}$}, $\{x^6\}$, \underline{$\{y,yx^4\}$}, $\{yx, yx^5\}$, \underline{$\{yx^2,yx^6\}$} and $\{yx^3, yx^7\}$. Thus $Q\D_4^-$ has 7 conjugacy classes (5 real) and $Q\D_4^+$ has 10 conjugacy classes (4 real). 
For $n=5$ we have  $\underline{\{e\}}$, $\{x,x^7\}$, \underline{$\{x^2, x^{14}\}$}, $\{x^3,x^5\}$, \underline{$\{x^4, x^{12}\}$},  \underline{$\{x^6, x^{10}\}$}, \underline{$\{x^8\}$}, $\{x^9, x^{15}\}$, $\{x^{11}, x^{13}\}$, 
\underline{$\{y, yx^2, yx^4, yx^6, yx^8, yx^{10}, yx^{12}, yx^{14}\}$} and  \underline{$\{yx, yx^3, yx^5, yx^7, yx^9, yx^{11}, yx^{13}, yx^{15}\}$} for $Q\D_5^-$ while  
\underline{$\{e\}$}, $\{x,x^9\}$, $\{x^2\}$, $\{x^3, x^{11}\}$, $\{x^4\}$, $\{x^5, x^{13}\}$, $\{x^7, x^{15}\}$, $\{x^6\}$, \underline{$\{x^8\}$}, $\{x^{10}\}$, $\{x^{12}\}$, $\{x^{14}\}$, \underline{$\{y,yx^8\}$}, $\{yx, yx^9\}$, $\{yx^2,yx^{10}\}$, $\{yx^3, yx^{11}\}$, \underline{$\{yx^4,yx^{12}\}$}, $\{yx^7, yx^{15}\}$ for $Q\D_5^+$. Thus $Q\D_5^-$ has 11 conjugacy classes (7 real) and $Q\D_5^+$ has 20 conjugacy classes (8 real). 

One can check that in general, $Q\D_n^-$ has $1+2^{n-2}+2$ conjugacy classes given by $C_e$, all the $C_{x^k}$'s with $k=1,\ldots,2^{n-1}-1$ (which are of the form $\{x^i,x^j\}$ with $8\mid i+j$ except for $\{x^{2^{n-2}}\}$), $C_y$ and $C_{yx}$ out of which $1+2^{n-3}+2$ are real, namely $C_e$, those of the form $C_{x^{2k}}$ and $C_y$, $C_{yx}$.   
On the other hand, $Q\D_n^+$ has $(2^{n-2}+2^{n-3})+2^{n-2}$ conjugacy classes given by $C_{x^k}$ with $k\ge 0$ and $C_y, C_{yx}, \ldots, C_{yx^{2^{n-2}-1}}$, out of which $4$ are real, namely $C_e$, $C_{x^{2^{n-2}}}$, $C_y$ and $C_{yx^{2^{n-3}}}$.  
		
Finally, we have that $c(Q\D_n^-)=2^{n-2}+3$ and $c_2(Q\D_n^-) =2^{n-3}+3$ and thus we get 
$$k^*(Q\D_n^-) = \tfrac 12 (2^{n-2}+3 + 2^{n-3}+3)-1 = 3\cdot 2^{n-4}+2.$$
Also, $c(Q\D_n^+)=(2^{n-2}+2^{n-3})+2^{n-2}=5\cdot 2^{n-3}$ and $c_2(Q\D_n^-) = 4$, and hence 
we have $k^*(Q\D_n^-) = \frac{5\cdot 2^{n-3}+4}{2}-1 = 5\cdot 2^{n-4}+1$, and the proof is complete.
\end{proof}

\begin{exam} \label{ambi3}
	The symmetric group $\Sym_n$ is ambivalent for any $n$. 
	It is a classic result that the number of conjugacy classes of $\Sym_n$ is given by the number $p(n)$ of unordered integer partitions.  
	Thus $k^*(\Sym_n) = p(n)-1$ and hence
	$$M^*(\Sym_n)=B_{p(n)-1}.$$
	Recall that $p(1)=1$, $p(2)=2$, $p(3)=3$, $p(4)=5$, $p(5)=7$ and $p(6)=11$. 
	So, for instance, we have $M^*(\Sym_3)=B_2=2$, $M^*(\Sym_4)=B_4=15$, $M^*(\Sym_5)=B_6=203$ and 
	$M^*(\Sym_6)=B_{10}=115.975$. 
	
	The conjugacy classes of the alternating groups are a bit more involved. It is known that 
	$$c(\mathbb{A}_n)= 2 f_n + \tfrac{p(n)-f_n}2$$ 
	where $f_n$ is the number of self-conjugate integer partitions, that is the number of self-conjugate Ferrer diagrams of size $n$.
	The alternating groups $\mathbb{A}_n$ are ambivalent only for $n=1,2,5,6,10,14$.
	Thus, $k^*(\mathbb{A}_n)= 2 f_n + \tfrac{p(n)-f_n}2 -1$ 
	for $n=1,2,5,6,10,14$.
	\hfill $\lozenge$
\end{exam}

\begin{exam} \label{SL2q}
	Consider the special group $\mathrm{SL}_2(\ff_q)$ of matrices $2 \times 2$ over a finite field $\ff_q$ of $q$ elements. 
	That is $\mathrm{SL}_2(\ff_q)=\{ (\begin{smallmatrix} a& b \\ c & d \end{smallmatrix}) : a,b,c,d \in \ff_q, ad-bc=1\}$.
	The group $\mathrm{SL}_2(\ff_q)$ has order $q^3-q$. If $q$ is odd, this group has only one element of order $2$, $-I$, and thus 
	$k(\mathrm{SL}_2(\ff_q))=\frac 12(q^3-q)$ and 
	$$M(\mathrm{SL}_2(\ff_q)) = B_{\frac 12 (q^3-q)}.$$ 
	For instance, $M(\mathrm{SL}_2(\ff_3))=B_{12}= 4.213.597$, $M(\mathrm{SL}_2(\ff_5))=B_{60}$ and $M(\mathrm{SL}_2(\ff_{3^2}))=B_{360}$. 
	
	The number of conjugacy classes of $\mathrm{SL}_2(\ff_q)$ is given by $q+4$ for $q$ odd and by $q+1$ if $q$ is even.
	It is known that $\mathrm{SL}_2(\ff_q)$ is ambivalent if and only if $-1$ is a square in $\ff_q$, which in turn happens if and only if $q$ is even or if $q\equiv 1 \pmod 4$. 
	Suppose that $q$ is odd with $q\equiv 1 \pmod 4$. Then, $k^*(\mathrm{SL}_2(\ff_q)) = \tfrac 12(q+4-1)$ by \eqref{kGcG} and hence 
	$$M^*(\mathrm{SL}_2(\ff_q)) = B_{\frac{q+3}2}.$$ 
	For instance, $M^*(\mathrm{SL}_2(\ff_5)))=B_4=52$ and $M^*(\mathrm{SL}_2(\ff_{3^2}))=B_6=203$. 
	\hfill $\lozenge$  
\end{exam}

Note that $k^*(G)$ is the cardinal of the finest bi-invariant (unitary, symmetric and conjugate) partition $\mathcal{P}^*$, 
and every bi-invariant partition is some union of parts of this partition. 
In particular, we have that  
$$k^*(G) \leq k(G) \leq n-1,$$ 
where $n=|G|$, with both equalities holding only for $G=\Z_2^k$ for $k\ge 1$ (see Example~\ref{coro k(G)}~($ii$)).
One also have that $k^*(G) \leq c(G) -1$ and every bi-invariant metric satisfies 
$$|\mathcal{P}^*| \le k^*(G) \leq c(G)-1.$$ 
Hence, the number of bi-invariant metrics $M^*(G)$ on $G$ is bounded by the Bell number $B_{c(G)-1}$.

We will need the following results.
\begin{lema} \label{cc2prod}
	If $H,K$ are groups then $c(H \times K)  = c(H)c(K)$, $c_2(H \times K)  = c_2(H)c_2(K)$ and %hence
	\begin{equation} \label{k*HxK}
	k^*(H \times K)  = \tfrac 12 \{ c(H)c(K) + c_2(H)c_2(K)\} -1.
	\end{equation}
	As a consequence, the product of ambivalent groups is ambivalent.
\end{lema}

\begin{proof}
	Clearly, we have the identity $C_{(h,k)}= C_h \times C_k$ of conjugacy classes for every $(h,k) \in H \times K$ and hence 
	$c(H \times K)  = c(H)c(K)$. 
	Let us see that a real conjugacy class in the product $H \times K$ is the product of real conjugacy classes in $H$ and $K$. 
	The conjugacy class $C_{(h,k)}$ is real if and only if 
	$C_{(h,k)} = C_{(h,k)^{-1}}= C_{(h^{-1}, k^{-1})}$. 
	On the other hand, 
$C_h \times C_k = C_{(h,k)} =  C_{(h^{-1}, k^{-1})} = C_{h^{-1}} \times C_{k^{-1}}$.
	This implies that $C_h = C_{h^{-1}}$ and $C_k = C_{k^{-1}}$. The converse implication is analogous and thus we have that
	$c_2(H \times K)  = c_2(H)c_2(K)$. Equation \eqref{k*HxK} follows from \eqref{kbiG} and these product formulas.  
\end{proof}

\begin{lema} \label{kGxZ2r}
Let $G$ be a finite group. If $A$ is any ambivalent finite group then 
\begin{equation} \label{eq kGxkA}
k^*(G \times A)= k^*(A) (k^*(G) + 1)-1.
\end{equation}
In particular, we have
$k^*(G \times \Z_2^r)= 2^r (k^*(G) + 1)-1$
for any $r\in \N_0$.
\end{lema}

\begin{proof} 
	The result is a direct consequence of Lemma \ref{cc2prod} and the definition \eqref{kbiG} of $k^*(G)$.  
\end{proof}

\subsubsection*{Bi-invariance degree} 
We know that for abelian groups, the number of invariant and bi-invariant metrics is the same. However, from our previous propositions and examples, this seems difficult for non-abelian groups to hold. In Example \ref{Q8D4} we saw that the quaternions $\Q_8$ has the same number of invariant and bi-invariant metrics. Moreover, from Propositions \ref{prop Ms} and \ref{prop M*s} we see that this is the only dicyclic group $\Q_{4n}$ having this property.

 To measure the ratio of bi-invariant metrics on a finite group $G$ out of the invariant ones we define the 
\textit{bi-invariance degree} of $G$ by 
	\begin{equation} \label{bG} 
		b(G)=\frac{k^{*}(G)}{k(G)}.
	\end{equation}
It is clear that $0<b(G) \le 1$ and that $b(G)=1$ if and only if every invariant metric is also bi-invariant. 
By \eqref{kG2} and \eqref{kbiG} we have the expression 
	$$b(G)=\frac{c(G)+c_2(G)-2}{|G|+k_2(G)-2}.$$ 
By Propositions \ref{prop Ms} and \ref{prop M*s} we have 
\begin{gather*}
b(\D_{2k+1})=\tfrac{k+1}{3k+1}, \quad b(\D_{2k})=\tfrac{k+2}{3k}, \quad b(\Q_{4n})=\tfrac{n+2}{2n}, \quad 
b(Q\D_n^-) = \tfrac{3 \cdot 2^{n-4}+2}{5\cdot 2^{n-3}}, \quad b(Q\D_n^+) = \tfrac{5 \cdot 2^{n-4}+1}{2^{n-1}+1}.
\end{gather*}
Therefore, asymptotically we have 
	$b(\D_n) \simeq \frac 13$, $b(\Q_{4n}) \simeq \frac 12$, $b(Q\D_n^-) \simeq \frac{3}{10}$ and $b(Q\D_n^+) \simeq \frac 58$.

The following, one of the main results in the paper, characterizes all groups $G$ (finite or not) in which the invariant metrics are also bi-invariant (i.e.\@ with $b(G)=1$ for the finite case).

\goodbreak 
\begin{thm} \label{char bi}
	Every right (left) invariant metric on a group $G$ is also bi-invariant if and only if $G$ is abelian or $G=\Q_8 \times H$ where $H$ is an elementary abelian $2$-group. In particular, if $G$ is finite and non-abelian then $G=\Q_8 \times \Z_2^k$ for some $k \in \Z_{\ge 0}$ and $M(G)=M^*(G)=B_{5\cdot 2^k-1}$. 
\end{thm}

\begin{proof}
	If $G$ is abelian we know that any invariant metric is also bi-invariant. 
	Let  $G=\Q_8 \times H$, where $H$ is an elementary abelian 2-group 
	(i.e.\@ every non trivial element has order $2$). 
	We want to show that every invariant metric $d$ on $G$ is also bi-invariant. That is, any unitary symmetric partition of $G$ is also conjugate.
	This happens if and only if the partition given by the conjugacy classes of $G$ is finer than the partition associated with the metric $d$, that is $C(G) \preceq P(G,d)$. 
	The conjugacy classes of $G$ are of two forms, namely 
	$$C(G) = \big\{ \{x\} : x \in Z(G) \big\} \cup \big \{ \{x,x^{-1}\} : x \notin Z(G) \big \},$$
	where $Z(G)$ denotes the center of $G$. In this way, we have that 
	$$C(G)\preceq \mathcal{P}_{Lee}(G) \preceq P(G,d)$$ 
	where $\mathcal{P}_{Lee}(G)$ is the finest unitary symmetric partition of $G$, i.e.\@ $\mathcal{P}_{Lee}(G)=\big\{ \{x,x^{-1} \} : x\in G \big\}$.
	
	\sk 
	We now prove the converse. Suppose that every invariant metric on $G$ is bi-invariant.
	For every $x \in G$ consider the metric given by the unitary symmetric partition 
	$$\Prep_x = \{\{e\},\{x,x^{-1}\},G \smallsetminus \{e,x,x^{-1}\}\}.$$ 
	Since $\Prep_x$ is conjugate, by hypothesis, we have that $C_x \subset \{x,x^{-1}\}$. 
		
	Notice that every subgroup $N$ of $G$ is normal, i.e.\@ $gNg^{-1}=N$ for every $g\in G$.  
	This holds since if $x\in N$ then $C_x \subset \{x,x^{-1}\} \subset N$ and $C_x \subset N$ for every $x\in N$ clearly implies the normality of $N$.
	Groups in which all the subgroups are normal are characterized (see \cite{Baer}), that is $G$ is either abelian or 
	$$G = \Q_8 \times H \times K$$ 
	where $H$ is an elementary abelian $2$-group and $K$ is an abelian group in which every element has odd order. To finish the proof we show that $K$ is the trivial group. 
	Take an element $x=(q,h,k) \in G$ with $ord(q)=4$. 
	On the one hand we have $C_{(q,h,k)}= \{ (q,h,k), (-q,h,k) \}$. 
	On the other hand, 
	$$C_{(q,h,k)} \subset \{(q,h,k), (q,h,k)^{-1} \} = \{ (q,h,k), (-q,h,k^{-1})\}.$$
	Hence $k=k^{-1}$, which implies that $k=e$ since $k$ has odd order and hence $K=\{e\}$. 
	
	\goodbreak  
	Finally, suppose that $G$ is finite. If $G$ is non-abelian then $G=\Q_8 \times \Z_2^k$ for some $k\in \N$
	and $G$ is ambivalent since $\Q_8$ and $\Z_2^{k}$ are ambivalent by Lemma \ref{cc2prod}. 
	Thus, by \eqref{kGcG} we have 
	$$k^*(G)= c(G)-1= c(\Q_8)c(\Z_2^k)-1= 5\cdot 2^k-1,$$
	where we used \eqref{kGQn} 
	with $n=2$ (one can also use Lemma \ref{kGxZ2r}).
	On the other hand, by Proposition~\ref{kGxZ2r} we obtain 
	$k(G)= 2^k(k(\Q_8)+1)-1 = 5\cdot 2^k-1$
	where we used Proposition \ref{prop Ms} ($b$) with $n=2$.
	Thus, $k^*(G)= k(G)$ and hence $M^*(G)= M(G)=B_{5\cdot 2^k-1}$, as desired. 
\end{proof}

To measure the commutativity of elements in non-abelian groups, Erdös and Turán defined in 1968 (\cite{ET}) the \textit{commutativity degree} of a group $G$ as 
$$d(G)= \frac{\#\{(x,y) \in G^2 : xy=yx \}}{|G|^2},$$
which is the probability that two elements of $G$ commute, and they proved that 
\begin{equation} \label{dG} 
d(G)= \frac{c(G)}{|G|}.
\end{equation}
Later, in 1973, Gustafson \cite{Gu} showed that if $G$ is non-abelian then $c(G) \le \tfrac 58 |G|$, hence $d(G) \le \frac 58$.

Note that for the groups $G=\Q_8 \times \Z_2^k$ it holds $c(G)=\frac 58 |G|$. 
In this way, for the groups in the previous theorem we have  $b(G)=1$ and also that $d(G)=1$ if $G$ is abelian or $d(G)=\frac 58$ if $G$ is non-abelian. In other words, if $d(G)$ is less that $\frac 58$ then $G$ has some invariant metric which is not bi-invariant, that is 
$$ d(G) < \tfrac 58 \quad \Rightarrow \quad b(G)<1.$$ 

However, for example, if $d(G)=\frac 58$ we cannot assure that every invariant metric is bi-invariant. This is the case for the groups $\D_4 \times \Z_2^k$ (see Tables \ref{tablita}--\ref{tablita5} for more examples). In fact, $c(\D_4 \times \Z_2^k) = c(\D_4)c(\Z_2^k) = 5\cdot 2^k$ and $|\D_4 \times \Z_2^k|=8\cdot 2^k$, hence $d(\D_4 \times \Z_2^k) = \frac 58$.
However, one checks that $b(\D_4 \times \Z_2^k)=\frac{5 \cdot 2^k-1}{7 \cdot 2^k-1} \sim \frac 57 < 1$.
Finally, we notice that all groups $G$ having $d(G)=\frac 58$ must be isoclinic to $\Q_8$ (see \cite{Le}).

%\subsubsection*{Ordered partitions}
%We now consider metrics associated to ordered partitions, although only in this subsection.
%
%If $G$ is a finite group of cardinality $n$ then ordered $\#(\widetilde{\mathcal{M}}(G)) = a_{k(G)}$, were $a(n)$ is the $n$-th ordered Bell number, or Fubini number given by
%$$	a(n)= \sum_{k=0}^n k! \{\begin{smallmatrix} n \\ k \end{smallmatrix}\}.$$
%	
%
%
%\begin{exam}
%	{\blue In $\Z_6$ there is an ordered partition whose associated metric is integral with gaps.
%	Consider the ordered partition $P_0=\{ 0\},P_1=\{ 1,5\},P_2=\{ 3\},P_3=\{ 2,4\}$ then this is not an interval metric cause if $w(P_i)=i$ then $w(2)=3>w(1)+w(1) = 1 +1 =2$.}
%\end{exam}
%
%
%{\blue Decir que no las estudiaremos pues todas las ordered metrics con la misma particion tienen igual grupo de simetria.}
%
%\begin{rem}
%For any $k\in \Z_{\ge 0}$, every metric in $\Q_8 \times \Z_2^k$ is interval.
%In fact, every metric in $\Q_8$ is interval by Proposition \ref{2-weights}. {\blue FALTA...}
%To see that every metric in $\Q_8 \times \Z_2$ is interval notice that any unitary symetric partition is of the form $\mathcal{P}(\Q_8 \times \Z_2) = \mathcal{P}(\Q_8) \times \mathcal{P}(\Z_2)$ since... 
%The associated weight is given by $w(q,h) = 2 w_{\Q_8}(q) + w_{\Z_2}(h)$ (see \cite{Ba95}), which is clearly interval since $w_{\Q_8}$ and $w_{\Z_2}$ are interval. The general case follows by induction.
%\end{rem}

\section{The number of invariant metrics on groups of order up to 32} \label{S6}
In this final section we compute the number of non-equivalent invariant and bi-invariant metrics for all the groups of order less than or equal to 32.

\begin{thm} \label{allmetrics}
Let $G$ be a finite group of order $n\le 32$.
The number of non-equivalent invariant and bi-invariant metrics of $G$ are given in Tables \ref{tablita}--\ref{tablita5}.   
\end{thm}

\begin{proof}
By Propositions \ref{CardMG} and \ref{No bi-inv}, the number of invariant and bi-invariant metrics of $G$ are given respectively by $M(G)=B_{k(G)}$ and $M^*(G)=B_{k^*(G)}$ where $B_m$ is the $m$-th Bell number and $k(G)$ and $k^*(G)$ are given in \eqref{kG} and \eqref{k*G}.

We know some special cases in general. For instance, we know that $M(G)=B_{(n-1)/2}$ if $n$ is odd (for any $G$) and that $M(G)=M^*(G)$ if $G$ is abelian. For $G$ abelian of even order $n=2m$ it is enough to use Corollary \ref{kGab} asserting that $M(G)=B_{m+2^{s-1}-1}$ where $s$ is the number of factors of the form $\Z_{2^t}$ in the prime decomposition of $G$.
For some families of non-abelian groups, we also know these numbers in general. Dihedral, dicyclic (including generalized quaternions) and quasidihedral %(semidihedral and modular) 
groups are covered by Propositions~\ref{prop Ms} and \ref{prop M*s} while symmetric, alternating and special linear groups are treated in Proposition \ref{Sn An} and Examples \ref{ambi3} and \ref{SL2q}.

We will do a case by case study, proceeding by increasing order of the groups.
The invariant and bi-invariant metrics for groups of order up to 7 were studied in Section \ref{S4} by inspection. 
With the mentioned results for abelian groups and dihedral or symmetric groups we check that the number of metrics are indeed correct. 
So we analyze the cases $n\ge 8$. We will do it in detail for $8 \le n\le 16$, i.e.\@ the groups in Table \ref{tablita}. For the cases $17\le n\le 32$ we will just give some comments, since it is clear how to proceed in each case.

\noindent \textit{Order 8.} There are 5 groups of order 8, the abelian ones $\Z_8$, $\Z_4 \times \Z_2$ and $\Z_2^3$ 
and the non-abelian ones $\D_4$ and $\Q_8$. It is clear that $M(\Z_8)=B_5$, $M(\Z_4 \times \Z_2)=B_6$ and 
$M(\Z_2^3)=B_7$. For the dihedral group we have $M(\D_4)=B_6$ and $M^*(\D_4) = B_4$ and for the quaternion group 
$\Q_8=\{\pm 1, \pm i, \pm j, \pm k\}$ 
we have $M(\Q_8)=M^*(\Q_8)=B_4$.

\noindent \textit{Order 9.} There are 2 groups of order 9, $\Z_9$ and $\Z_3 \times \Z_3$, hence $M(G)=B_4$ in these two cases.  

\noindent \textit{Order 10.} There are two groups of order 10, the cyclic group $\Z_{10}$ and the dihedral group $\D_{5}$. Since $\Z_{10} = \Z_5 \times \Z_2$ we have that $M(\Z_{10})=B_6$. Also, $M(\D_5)=B_7$ and $M^*(\D_5)=B_3$.

\noindent \textit{Order 11.} 
There is only one group of order 11, $\Z_{11}$, hence $M(\Z_{11})=B_5$.

\noindent \textit{Order 12.} There are five groups of order $12$, two abelian ones $\Z_{12}=\Z_4 \times \Z_3$ and $\Z_6 \times \Z_2 = \Z_3 \times \Z_2^2$, and three non-abelian ones $\D_6$, $\Q_{12}$ and $\mathbb{A}_4$. For the abelian groups we have 
$M(\Z_{12})=B_6$ and $M(\Z_6 \times \Z_2)=B_7$. 
For the dihedral group we have $M(\D_6)= B_9$ and $M^*(\D_6)=B_5$ while for the dyciclic group $\Q_{12}$ we have that $M(\Q_{12})=B_6$ and $M^*(\Q_{12})=B_4$.
For the alternating group $\mathbb{A}_4$ we have $M(\mathbb{A}_4)= B_7$ by Proposition \ref{Sn An}. 
To compute the number of bi-invariant metrics on $\mathbb{A}_4$, notice that $\mathbb{A}_4$ has 4 conjugacy classes given by $\{ id \}$, $\{ (12)(34), (13)(24), (14)(23)\}$, $\{ (123), (243), (142), (134) \}$ and $\{(132), (124), (143), (234)\}$, the first two of which are real classes. Hence, $k^*(\mathbb{A}_4) = \frac{4+2}2 - 1= 2 $ and thus 
$M^*(\mathbb{A}_4)=B_2$. 

\noindent \textit{Order 13.} There is only one group of order 13, $\Z_{13}$, hence $M(\Z_{13})=B_6$.

\noindent \textit{Order 14.}  There are two groups of order 14, the cyclic group $\Z_{14}$ and the dihedral group $\D_{7}$. Since $\Z_{14}=\Z_7 \times \Z_2$ we have that $M(\Z_{14})=B_8$. Also, $M(\D_7)=B_{10}$ and $M^*(\D_7)=B_4$.

\noindent \textit{Order 15.} There is only one group of order 15, $\Z_{15}$, hence $M(\Z_{15})=B_7$.

\noindent \textit{Order 16.} There are 14 groups of order 16; the 5 abelian groups $\Z_{16}$, $\Z_8 \times \Z_2$, $\Z_4^2$, $\Z_4 \times \Z_2^2$ and $\Z_2^4$, and the 9 non-abelian groups   
$\D_{8}$, $\Q_{16}$, $Q\D_{4}^+$, $Q\D_{4}^-$, $\D_4 \times \Z_2$, $\Q_8 \times \Z_2$, $\Z_4 \rtimes \Z_4$, $\Z_2^2 \rtimes \Z_4$, and $(\Z_4 \times \Z_2) \rtimes \Z_2$.
We have $M(\Z_{16})=B_8$, $M(\Z_8 \times \Z_2)=M(\Z_4^2)=B_{9}$, 
$M(\Z_4 \times \Z_2^2)=B_{11}$ and $M(\Z_2^4)=B_{15}$.
For the dihedral groups we know that $M(\D_8)=B_{12}$ and $M^*(\D_8)=B_6$, for the quaternionic group we have $M(\Q_{16})=B_8$ and $M^*(\Q_{16})=B_6$ and for the semidihedral groups we have $M(Q\D_4^-)=B_{20}$, $M(Q\D_{4}^+)=B_9$ and $M^*(Q\D_4^-)=B_5$, $M^*(Q\D_4^+)=B_6$. 

For the direct products $G \times \Z_2$ with $G$ non-abelian of order 8, using Lemmas \ref{lemHxK} and \ref{kGxZ2r} we get that $M(\D_4 \times \Z_2) = B_{13}$ and $M^*(\D_4 \times \Z_2)=B_9$ and that $M(\Q_8 \times \Z_2)=B_9$ and $M^*(\Q_8 \times \Z_2)=B_9$.

It remains to study the three semidirect products $\Z_4 \rtimes_\varphi \Z_4$, $\Z_2^2 \rtimes_\varphi \Z_4$ and $(\Z_4 \times \Z_2) \rtimes_\varphi \Z_2$. For the first semidirect product, $\varphi:\Z_4 \rightarrow \mathrm{Aut}(\Z_4)$ is determined by $\varphi(1)$ and $\varphi_1(a) = (-1)^a$.
By \eqref{k semid} we have that 
\begin{equation} \label{sd1}
	k_2(\Z_4 \rtimes_\varphi \Z_4)=k_2(\Z_4) + \sum_{h \in \Z_4^*, 2h=0} \#\{g\in \Z_4: \varphi_h(g)=-g\}= 2+2=4,
\end{equation} 
where $\#\{g\in \Z_4: \varphi_2(g)=-g\}=2$ since $\varphi_2=id$.

For the second semidirect product, $\varphi : \Z_4 \rightarrow \mathrm{Aut}(\Z_2^2) \simeq \mathrm{GL}_2(\Z_2)$ is determined by $\varphi(1)$ which must be an element of order 2. Since all elements of order 2 in $\mathrm{GL}_2(\Z_2)$ are conjugate, all posible automorphism give rise to the same semidirect product. Thus, we can take $\varphi_1$ as multiplication by the matrix $J=(\begin{smallmatrix} 0&1 \\ 1&0 \end{smallmatrix}$). By \eqref{k semid} we have that $\varphi_2=id$ and 
\begin{equation} \label{sd2}
k_2(\Z_2^2 \rtimes_\varphi \Z_4) = k_2(\Z_2^2) + \#\{(x,y)\in \Z_2^2: 2(x,y)=(0,0)\}= 4+4=8.
\end{equation}

For the third semidirect product, we use that $(\Z_4 \times \Z_2) \rtimes_\varphi \Z_2 \simeq \Q_8 \rtimes_{\varphi'} \Z_2$, where 
the morphism $\varphi' : \Z_2 \rightarrow \mathrm{Aut}(\Q_8)$ is given by conjugation by $i$, that is $\varphi_1'(x) = ixi^{-1}$ for any quaternion $x$.
By \eqref{k semidZ2} we have that 
\begin{equation} \label{sd3}
k_2(\Q_8 \rtimes_{\varphi'} \Z_2) = k_2(\Q_8) + \#\{ x\in \Q_8 : xixi^{-1}=1\}= 2+6=8,
\end{equation}
since $xix(-i)=1$ holds for every quaternion different from $\pm i$.

Thus, we have that 
$k(\Z_4 \rtimes \Z_4)=\frac{16+4}2-1$ and $k(\Z_2^2 \rtimes_\varphi \Z_4) = k((\Z_4 \times \Z_2) \rtimes_\varphi \Z_2) = \frac{16+8}2-1$ by \eqref{sd1}--\eqref{sd3}, and hence $M(\Z_4 \rtimes \Z_4)=B_9$ and $M(\Z_2^2 \rtimes_\varphi \Z_4) = M((\Z_4 \times \Z_2) \rtimes_\varphi \Z_2)= B_{11}$.

For the group $$\Z_4 \rtimes_\varphi \Z_4=\langle a,x : a^4=x^4=1, xax^{-1}=a^{-1}\rangle$$
the conjugacy classes are 
$$\{e\}, \{a,a^{3}\}, \{x,a^{2}x\}, \{a^2\}, \{x^2\}, \{ax,a^{3}x\}, \{ax^2,a^{3}x^2\}, \{x^3,a^2x^3\}, \{a^2x^2\}, \{ax^3,a^{3}x^3\}$$ 
while the real conjugacy classes are $\{e\}, \{a,a^{3}\}, \{a^2\}, \{x^2\}, \{ax^2,a^{3}x^2\}, \{a^2x^2\}$.
Thus, we have $k(\Z_4 \rtimes_\varphi \Z_4)=\frac{10+6}2-1$ and hence $M^*(\Z_4 \rtimes_\varphi \Z_4)=B_7$. 

The groups 
\begin{gather*}
\Z_2^2 \rtimes_\varphi \Z_4 = \langle a,b,c : a^2=b^2=c^4=1, cac^{-1}=ab=ba, bc=cb \rangle, \\
(\Z_4 \times \Z_2) \rtimes_\varphi \Z_2 = \langle  a,b,c : a^4=c^2=1, b^2=a^2, ab=ba, ac=ca, cbc=a^2b  \rangle,
\end{gather*}
have character tables given by
$$\begin{array}{c|rrrrrrrrrr}
\Z_2^2 \rtimes \Z_4 & &&&&&&&&& \\ 
\hline
\rho_1 & 1 & 1 & 1 & 1 & 1 & 1 & 1 & 1 & 1 & 1 \\
\rho_2 & 1 & -1 & -1 & 1 & 1 & 1 & -1 & -1 & 1 & 1 \\
\rho_3 & 1 & -1 & 1 & 1 & 1 & -1 & -1 & 1 & 1 & -1 \\
\rho_4 & 1 & 1 & -1 & 1 & 1 & -1 & 1 & -1 & 1 & -1 \\
\rho_5 & 1 & -i & -1 & 1 & -1 & i & i & 1 & -1 & -i \\
\rho_6 & 1 & i & -1 & 1 & -1 & -i & -i & 1 & -1 & i \\
\rho_7 & 1 & -i & 1 & 1 & -1 & -i & i & -1 & -1 & i \\
\rho_8 & 1 & i & 1 & 1 & -1 & i & -i & -1 & -1 & -i \\
\rho_9 & 2 & 0 & 0 & -2 & 2 & 0 & 0 & 0 & -2 & 0 \\
\rho_{10} & 2 & 0 & 0 & -2 & -2 & 0 & 0 & 0 & 2 & 0 \\ 
\end{array}$$  
$$\begin{array}{c|rrrrrrrrrr}  
(\Z_2\times \Z_4)\rtimes\Z_2 &&&&&&&&&& \\ \hline
\rho_1 & 1 & 1 & 1 & 1 & 1 & 1 & 1 & 1 & 1 & 1 \\
\rho_2 & 1 & -1 & -1 & -1 & 1 & 1 & 1 & 1 & -1 & -1 \\
\rho_3 & 1 & -1 & -1 & 1 & 1 & 1 & -1 & -1 & 1 & 1 \\
\rho_4 & 1 & -1 & 1 & -1 & 1 & -1 & 1 & -1 & -1 & 1 \\
\rho_5 & 1 & -1 & 1 & 1 & 1 & -1 & -1 & 1 & 1 & -1 \\
\rho_6 & 1 & 1 & -1 & -1 & 1 & -1 & -1 & 1 & -1 & 1 \\
\rho_7 & 1 & 1 & -1 & 1 & 1 & -1 & 1 & -1 & 1 & -1 \\
\rho_8 & 1 & 1 & 1 & -1 & 1 & 1 & -1 & -1 & -1 & -1 \\
\rho_9 & 2 & 0 & 0 & -2 i & -2 & 0 & 0 & 0 & 2 i & 0 \\
\rho_{10} & 2 & 0 & 0 & 2 i & -2 & 0 & 0 & 0 & -2 i & 0 \\ 
\end{array}$$ 

In general, the number of characters of a group is the same as the number of conjugacy classes while the number of real characters coincides with the number of real conjugacy classes. 
Thus, we have $k(\Z_2^2 \rtimes \Z_4)=\frac{10+6}2-1=7$ and $k((\Z_4\times \Z_2)\rtimes\Z_2)=\frac{10+8}2-1=8$ and hence 
$M^*(\Z_2^2 \rtimes \Z_4)=B_7$ and $M^*((\Z_4\times \Z_2)\rtimes\Z_2)=B_8$.

For the remaining groups of order $17\le n \le 32$ (see Tables \ref{tablita2}--\ref{tablita5}) the results follow in the same way as before from the results of the previous section. 
For the computation of the invariant metrics for the groups of even order, 
the groups that do not fall into the hypothesis of our previous results are the group
$\Z_5 \rtimes \Z_4 = \langle a,b : a^5=b^4=e, bab^{-1}=a^3 \rangle$  ($\#54$)
of order 20, the group 
$\Z_3 \rtimes \Z_8 = \langle a,b : a^3=b^8=e, bab^{-1}=a^{-1} \rangle$  ($\#74$) 
of order 24 and those groups of order 32 numbered \#116--\#121 and \#135--\#142 in Table \ref{tablita5}.
Using the presentations it is easy (though tedious) to check that $k_2(\Z_5 \rtimes \Z_4)=6$ and $k_2(\Z_3 \rtimes \Z_8 )=2$ from which one obtain the number of invariant metrics.   
For the rest of the groups one can also use their presentations with generators and relations (or some mathematical software to count the number of elements of order 2 of these groups). 
To compute the bi-invariant metrics of the non-abelian groups of order $17 \le n\le 32$ which are not dihedral, semidihedral or dicyclic groups, one can use their character tables as before. For instance, one can use the webpage GroupNames (\cite{GN}). However, many of the direct products can be computed by reducing them to easier cases by using Lemma~\ref{kGxZ2r}. 
\end{proof}

\renewcommand{\arraystretch}{1.15}
\begin{table}[H] 
\caption{Invariant metrics on groups of order $1\le n \le 16$}  \label{tablita}
$$\begin{tabular}{|c|c|c|c|c|c|c|c|c|c|} 
	\hline 
\# & order & id & $G$									 & $M(G)$ 			& $M^*(G)$ & $b(G)$ & $d(G)$ & comment \\	\hline
1 & $1$ (1) & 1  & $\{e\}$ 											 & $B_1$ 			& 				 & 1 & & abelian \\ \hline
2 & $2$ (1) & 1  & $\Z_2$ 												 & $B_1$ 			& 				 & 1 & & abelian \\ \hline
3 & $3$ (1) & 1 & $\Z_3$ 												 & $B_1$		& 			 & 1 & & abelian \\ \hline
4 & $4$ (2) & 1 & $\Z_4$ 												 & $B_2$ 			& 				 & 1 & & abelian \\ 
5 &     & 2 & $\Z_2^2					$ 						 & $B_3$ 			& 				 & 1 & & abelian \\ \hline
6 & $5$ (1)& 1 & $\Z_5$ 												 & $B_2$ 			& 			   & 1 & & abelian \\ \hline
7 & $6$ (2)& 2  & $\Z_6$										 & $B_3$ 			& 				 & 1 & & abelian \\    
8 &     & 1  & $\D_3$   			 & $B_4$ 		& $B_2$ & $\frac 12$ & $\frac 12$ & dihedral, $\Sym_3$, $\mathrm{SL}_2(\ff_2)$ \\ \hline
9 & $7$ (1)& 1  & $\Z_7$ 												 & $B_3$ 	& 				 & 1 & & abelian \\ \hline  
10 & $8$ (5)& 1 & $\Z_8$ 				                 & $B_4$ 		& 				 & 1 & & abelian \\  
11 &    & 2 & $\Z_4 \times \Z_2$ 						 & $B_5$ 		& 				 & 1 & & abelian \\ 
12 &    & 5 & $\Z_2^3$											 & $B_7$ 		& 			& 1 & & abelian \\ 
13 &    & 3 & $\D_4$ 												 & $B_6$ 		& $B_4$ 		& $\frac 23$ & $\frac 58$ & dihedral \\
14 & 	& 4 & $\Q_8$ 												 & $B_4$ 		& $B_4$ 		& 1 & $\frac 58$ & quaternionic \\ \hline
15 & $9$ (2)& 1 & $\Z_9$ 											& $B_4$ 		& 			& 1 & & abelian \\  
16 & 	& 2	& $\Z_3^2$ 											 & $B_4$ 		& 				 & 1 & & abelian  \\ \hline  
17 & $10$ (2)& 2 & $\Z_{10}$ 										 & $B_5$ 		& 				 & 1 & & abelian \\ 
18 &	  & 1 & $\D_5$ 												 & $B_7$  	& $B_3$	& $\frac 37$  & $\frac 25$ & dihedral \\ \hline
19 & $11$ (1)& 1 & $\Z_{11}$ 										 & $B_5$ 		& 				 & 1 & & abelian \\ \hline
20 & $12$ (5)& 2 & $\Z_{12}$ 										 & $B_6$ 		&					 & 1 & & abelian \\ 
21 &	& 5	& $\Z_3 \times \Z_2^2$ 					 & $B_7$  	& 				 & 1 & & abelian \\ 
22 &	& 4	& $\D_6$ 												 & $B_9$	& $B_5$ 		 & $\frac 59$  & $\frac 12$  & dihedral \\
23 & 	& 1	& $\Q_{12}$  & $B_6$   	& $B_4$ 		 & $\frac 23$  & $\frac 12$  & dicyclic \\ 
24 & 	& 3 & $\mathbb{A}_{4}$ & $B_7$ &  $B_2$ 	 & $\frac 27$ & $\frac 13$  & alternating \\ \hline
25 & $13$ (1)& 1 & $\Z_{13}$ 										 & $B_6$ 		& 				 & 1 & & abelian \\ \hline
26 & $14$ (2)&	2 & $\Z_{14}$ 										 & $B_7$ 		& 				 & 1 & & abelian \\ 
27 & 	  &	1 & $\D_7$  											 & $B_{10}$& $B_4$ 		 & $\frac 25$ & $\frac{5}{14}$ & dihedral  \\ \hline 
28 & $15$ (1)& 1 & $\Z_{15}$ 										 & $B_7$ 		& 				 & 1 & & abelian \\ \hline
29 & $16$ (14)& 1 & $\Z_{16}$ 										 & $B_8$ 	& 				 & 1 & & abelian \\
30 &	  & 5 & $\Z_{8} \times \Z_2$ 					 & $B_9$ & 				 & 1 & & abelian \\
31 &	  &	2 & $\Z_4^2$ 											 & $B_9$ & 				 & 1 & & abelian \\
32 &	  &	10 & $\Z_{4} \times \Z_2^2$ 				 & $B_{11}$& 				 & 1 & & abelian \\
33 &	  &	14 & $\Z_2^4$ 											 & $B_{15}$ & 	 & 1 & & abelian \\
34 &	  & 7 & $\D_{8}$ 											 & $B_{12}$ & $B_6$ & $\frac 12$ & $\frac{7}{16}$ & dihedral  \\
35 &	  & 9 & $\Q_{16}$ 										 & $B_8$  & $B_6$		 & $\frac 34$& $\frac{7}{16}$ & quaternionic \\ 
36 &	  &	8 & $Q\D_{4}^-$ 		 & $B_{10}$& $B_5$ 		 & $\frac 12$ & $\frac{7}{16}$ & quasidhedral\\ 
37 &	  &	6 & $Q\D_{4}^+$  & $B_9$ & $B_6$ 	 & $\frac 23$ & $\frac 58$ & quasidhedral\\ 
38 &	  &	11 & $\D_4 \times \Z_2$ 						 & $B_{13}$& $B_9$ & $\frac{9}{13}$& $\frac 58$ & product, gen.\@ dihedral \\
39 & 	  &	12 & $\Q_8 \times \Z_2$ 						 & $B_9$ & $B_9$ & 1 & $\frac 58$ & product\\
40 &	  & 4 & $\Z_4 \rtimes \Z_4$   				 & $B_9$ & $B_7$	  & $\frac 79$ & $\frac 58$  &  semidirect product\\ 
41 & 	  & 3 & $\Z_2^2 \rtimes \Z_4$ 				 & $B_{11}$& $B_7$ & $\frac{7}{11}$& $\frac 58$  &  semidirect product \\
42 &	  & 13 & $(\Z_4\times \Z_2)\rtimes \Z_2$ & $B_{11} $ & $B_8$ &$\frac{8}{11}$ & $\frac 58$   & semidirect product %generalized dihedral 
\\ \hline
\end{tabular}$$
\end{table}

\subsubsection*{Comments on Tables \ref{tablita}--\ref{tablita5}} 
In the second and third columns we list the GAP id of the groups. In the second column we also indicate between parenthesis the number of groups of the given order. 
Recall that $b(G)=\frac{k^*(G)}{k(G)}$ and $d(G)=\frac{c(G)}{|G|}$.
For abelian groups $G$ we omit the number of bi-invariant metrics, since $M(G)=M^*(G)$, and $d(G)=1$.
In the last column we indicate the property of group that best describes it. In general there are some of them and we list the ones that can be used to obtain the number of metrics with our results.  
Finally, we mention that to find which groups of order $2n$ are generalized dihedral groups, we perform the semidirect product of non-cyclic abelian groups of order $n$ different from $\Z_2^k$ and compare them with the non-abelian groups of order $2n$ in the tables. We get: 
$\D(\Z_4 \times \Z_2)= \D_4 \times \Z_2$, 
$\D(\Z_3^2) = \Z_3^2 \rtimes \Z_2 $, $\D(\Z_3 \times \Z_2^2) = \D_6 \times \Z_2$, 
$\D(\Z_8 \times \Z_2)= \D_8 \times \Z_2$, $\D(\Z_4^2)= \Z_4 \rtimes \D_4$ and $\D(\Z_4 \times \Z_2^2)= \D_4 \times \Z_2^2$.

\renewcommand{\arraystretch}{1.1125}
\begin{table}[H] 
\caption{Invariant metrics on groups of order $17 \le n \le 24$}  \label{tablita2}
$$\begin{tabular}{|c|c|c|c|c|c|c|c|c|c|} 
	\hline 
\# & order & id & $G$ 						 & $M(G)$ 			& $M^*(G)$ & $b(G)$ & $d(G)$ & comment \\	\hline
43& 17 (1)& 1 & $\Z_{17}$ & $B_8$ &  & $1$ & & abelian \\ \hline
44 & 18 (5)& 2 & $\Z_{18}$ & $B_9$ & & $1$ & & abelian \\
45 &       & 5 &  $\Z_3^2 \times \Z_2 $ & $B_9$ & & $1$ & & abelian \\ 
46 &       & 1 & $ \D_9 $ & $B_{13}$ & $B_5$ & $\frac{5}{13}$& $ \frac 13$ & dihedral \\ 
47 &       & 3 & $ \D_3 \times \Z_3 $ & $B_{10}$ & $B_5$ & $\frac 12$& $ \frac 12$ & product \\
48 &       & 4 & $ \Z_3^2 \rtimes \Z_2 $ & $B_{13}$ & $B_5$ & $\frac{5}{13}$& $ \frac 13$ & generalized dihedral \\ \hline
49 & 19 (1) & 1 & $ \Z_{19}$ & $B_9$ & & $1$ & &  abelian \\ \hline
50 & 20 (5) & 2 &$\Z_{20}$ & $B_{10}$ & & $1$ & & abelian \\
51 & & 5 & $\Z_5 \times \Z_2^2$ & $B_{11}$ & & $1$ & &  abelian \\ 
52 & & 4& $\D_{10}$ & $B_{15}$ & $B_7$ &$\frac{7}{15}$ & $\frac{2}{5}$ &  dihedral \\
53 & & 1 & $\Q_{20}$ & $B_{10}$ & $B_6$ & $\frac{3}{5}$ & $\frac{2}{5}$ &  dicyclic \\ 
54 &  & 3 & $\Z_5 \rtimes \Z_4$ & $B_{12}$ & $B_3$ & $\frac 14$ & $\frac 14$ & semidirect product \\ \hline
55 & 21 (2)& 2&$ \Z_{21}$ & $B_{10}$ &  & $1$ & & abelian \\ 
56 &  &1 & $\Z_7 \rtimes \Z_3$ & $B_{10}$ & $B_2$ & $\frac 15$&  $\frac{5}{21}$ & semidirect product \\ \hline
57 & 22 (2) & 2 & $\Z_{22}$ & $B_{11}$ &  & $1$ & & abelian \\ 
58 &  & 1 & $\D_{11}$ & $B_{16}$ & $B_6$ & $\frac 38$& $\frac{7}{22}$ & dihedral \\ \hline
59 & 23 (1)& 1 & $\Z_{23}$ & $B_{11}$ & & $1$ & & abelian \\ \hline
60 & 24 (15) & 2 & $\Z_{24}$ & $B_{12}$ &  & $1$ & & abelian \\
61 & & 9 & $\Z_4 \times \Z_3 \times \Z_2$ & $B_{13}$ & & $1$ & & abelian \\
62 & & 15 & $\Z_3 \times \Z_2^3$ & $B_{15}$ & & $1$ & &  abelian \\ 
63 & & 6 & $\D_{12}$ & $B_{18}$ & $B_8$ & $\frac 49$ & $\frac{3}{8}$ & dihedral \\
64 & & 4 & $\Q_{24}$ & $B_{12}$ & $B_8$ & $\frac 23$ & $\frac{3}{8}$ & quaternionic \\
65 & & 14 & $\D_6 \times \Z_2$ & $B_{19}$ & $B_{11}$ & $\frac{11}{19}$& $\frac 12$ & product, gen.\@ dihedral  \\
66 & & 10 & $\D_4 \times \Z_3$ & $B_{14}$ & $B_9$ & $\frac{9}{14}$ & $\frac{5}{8}$ &  product  \\
67 & & 5 & $\D_3 \times \Z_4$ & $B_{15}$ & $B_8$ & $\frac{8}{15}$ & $\frac 12$ & product  \\
68 & & 7 & $\Q_{12} \times \Z_2$ & $B_{13}$ & $B_9$ & $\frac{9}{13}$ & $\frac 12$ & product \\
69 & & 11 & $\Q_8 \times \Z_3$ & $B_{12}$ & $B_9$ & $\frac 34$ & $\frac{5}{8}$ & product \\
70 & & 13 & $\mathbb{A}_4 \times \Z_2$ & $B_{15}$ & $B_5$ & $\frac 13$& $\frac 13$ & product \\
71 & & 12 & $\Sym_4$ & $B_{16}$ & $B_4$ & $\frac 34$ & $\frac{5}{24}$ & symmetric \\
72 & & 3 & $\mathrm{SL}_2(\ff_3)$ & $B_{12}$ & $B_4$ & $\frac 13$ & $\frac{7}{24}$ &  special linear \\
73 & & 8 & $(\Z_6 \times \Z_2) \rtimes \Z_2$ & $B_{16}$ & $B_7$ & $\frac{7}{16}$ & $\frac{3}{8}$ & semidirect product \\
74 & & 1 & $\Z_3 \rtimes \Z_8$ & $B_{12}$ & $B_7$ & $\frac{7}{12}$ & $ \frac 12$ & semidirect product \\ \hline
\end{tabular}$$
\end{table}

\renewcommand{\arraystretch}{1.25}
\begin{table}[H] 
	\caption{Invariant metrics on groups of order $25 \le n \le 31$}  \label{tablita3}
	$$\begin{tabular}{|c|c|c|c|c|c|c|c|c|c|} 
	\hline 
\# & order & id & $G$ 				& $M(G)$ 			& $M^*(G)$ & $b(G)$ & $d(G)$ & comment \\	\hline
75 & 25 (2)& 1 & $\Z_{25} $ & $B_{12}$ &  & 1 & & abelian \\ 
76 & 	& 2 & $\Z_5^2$ & $B_{12}$ & & $1$ & & abelian\\ \hline
77 & 26	(2)& 2& $\Z_{26}$ & $B_{13}$ & & $1$ & & abelian \\ 
78 &    & 1 & $\D_{13}$ & $B_{19}$ & $B_7$ & $\frac{7}{19}$ & $\frac{4}{13}$ & dihedral \\ \hline
79 & 	27 (5) & 1 & $\Z_{27}$ & $B_{13}$ & & $1$ & & abelian \\
80 & 	& 2 & $\Z_9 \times \Z_3$ & $B_{13}$ & & $1$ & & abelian \\
81 & 	& 5 & $\Z_3^3$ & $B_{13}$ & & $1$ & & abelian \\ 
82 & 	& 4 & $\Z_9 \rtimes \Z_3$ & $B_{13}$ & $B_5$ & $\frac{5}{13}$ & $\frac{11}{27}$ &  semidirect product \\ 
83 & 	& 3 & $\Z_3^2 \rtimes \Z_3$ & $B_{13}$ & $B_5$ & $\frac{5}{13}$ & $\frac{11}{27}$ &  semidirect product \\ \hline
84 & 28	(4)& 2 & $\Z_{28}$ & $B_{14}$ & & $1$ & & abelian  \\
85 &    & 4 & $\Z_{14} \times \Z_2$ & $B_{15}$ &  & $1$ & & abelian \\ 
86 & 	& 3 & $\D_{14}$ & $B_{21}$ & $B_9$ & $\frac{3}{7}$ & $\frac{5}{14}$ & dihedral \\
87 &    & 1 & $\Q_{28}$ & $B_{14}$ & $B_8$ & $\frac 47$ & $\frac{5}{14}$ & dicyclic \\ \hline
88 & 	29 (1)& 1 & $\Z_{29}$ & $B_{14}$ & & $1$ & & abelian \\ \hline
89 & 30 (4)	& 4 & $\Z_{30}$ & $B_{15}$ & & $1$ & & abelian \\ 
90 & 	& 3& $\D_{15}$ & $B_{22}$ & $B_8$ & $\frac{4}{11}$ & $\frac{3}{10}$ & dihedral \\
91 & 	& 2 & $\D_5 \times \Z_3$ & $B_{17}$ & $B_7$ & $\frac{7}{17}$ & $\frac{2}{5}$ & product \\
92 &    & 1 & $\D_3 \times \Z_5$ & $B_{16}$ & $B_8$ & $\frac 12$ & $\frac 12$ & product \\ \hline
93 & 31 (1)& 1 & $\Z_{31}$ & $B_{15}$ & & $1$ & & abelian \\ \hline
	\end{tabular}$$
\end{table}

\renewcommand{\arraystretch}{1.25}
\begin{table}[h!] 
	\caption{Invariant metrics on groups of order $32$, I: abelian, dihedral, quasidhedral and dicyclic}
	\label{tablita4}
	$$\begin{tabular}{|c|c|c|c|c|c|c|c|} 
	\hline 
	\# & id & $G$ 		 & $M(G)$ 			& $M^*(G)$ & $b(G)$ & $d(G)$ & comment \\	\hline
	94 & 1 & $\Z_{32}$ & $B_{16}$ &  & 1 & & abelian \\ 
	95 & 16& $\Z_{16} \times \Z_2$ & $B_{17}$ & & 1 & & abelian \\
	96 & 3 & $\Z_8 \times \Z_4$ & $B_{17}$ & & $1$ & & abelian \\ 
	97 & 36& $\Z_8 \times \Z_2^2$ & $B_{19}$ & & $1$ & & abelian \\ 
	98 & 21& $\Z_4^2 \times \Z_2$ & $B_{19}$ & & $1$ & & abelian \\
	99 & 45& $\Z_4 \times \Z_2^3$ & $B_{23}$ & & $1$ & & abelian\\
	100 & 51& $\Z_2^5$ & $B_{31}$ & & $1$ & & abelian \\ \hline 
	101 & 18 & $\D_{16}$ & $B_{24}$ & $B_{10}$ & $\frac{5}{12}$ & $\frac{11}{32}$ &  dihedral \\
	102 & 20 & $\Q_{32}$ & $B_{16}$ & $B_{10}$ & $\frac{5}{8}$  & $\frac{11}{32}$ &  quaternionic \\ 
	103 & 19 & $S\D_5^-$ & $B_{20}$ & $B_8$ & $\frac 25$  & $\frac{11}{32}$ &  quasidhedral\\ 
	104 & 17 & $S\D_5^+$ & $B_{17}$ & $B_{11}$ & $\frac{11}{17}$ & $\frac 58$ & quasidhedral \\ \hline
	\end{tabular}$$
\end{table}

\renewcommand{\arraystretch}{1.15}
\begin{table}[H] 
	\caption{Invariant metrics on groups of order $32$, II: products and extensions}
	\label{tablita5}
	$$\begin{tabular}{|c|c|c|c|c|c|c|c|} 
	\hline 
	\# & id & $G$ 						 & $M(G)$ 		& $M^*(G)$  & $b(G)$ & $d(G)$ & comment \\	\hline
	105& 39 & $\D_{8} \times \Z_2$ & $B_{25}$ & $B_{13}$ & $\frac{13}{25}$ & $\frac{7}{16}$ &  direct product, gen.\@ dihedral \\ 
	106& 25 & $\D_{4} \times \Z_4$ & $B_{21}$ & $B_{14}$ & $\frac 23$ & $\frac{5}{8}$ & direct product \\
	107& 46 & $\D_{4} \times \Z_2^2$ & $B_{27}$ & $B_{19}$ & $\frac{19}{27}$ & $\frac{5}{8}$ & direct product, gen.\@ dihedral \\ 
	108& 41 & $\Q_{16} \times \Z_2$ & $B_{17}$ & $B_{13}$ & $\frac{13}{17}$ & $\frac{7}{16}$ & direct product \\ 
	109& 26 & $\Q_{8} \times \Z_4$ & $B_{17}$ & $B_{14}$ & $\frac{14}{17}$ & $\frac{5}{8}$ & direct product \\ 
	110& 47 & $\Q_{8} \times \Z_2^2$ & $B_{19}$ & $B_{19}$ & 1 & $\frac{5}{8}$ & direct product \\
	111& 40 & $S\D_5^- \times \Z_2$ & $B_{21}$ & $B_{11}$ & $\frac{11}{21}$ & $\frac{7}{16}$ & direct product \\ 
	112& 37 & $S\D_5^+ \times \Z_2$ & $B_{19}$ & $B_{13}$ & $\frac{13}{19}$ & $\frac 58$ & direct product \\ 
113& 23 & $(\Z_4 \rtimes \Z_4) \times \Z_2$ & $B_{19}$ & $B_{15}$ & $\frac{15}{19}$ & $\frac 58$ &  direct product \\ 
114& 22 & $(\Z_2^2 \rtimes \Z_4) \times \Z_2$ & $B_{23}$ & $B_{15}$ & $\frac{15}{23}$ & $\frac 58$ &  direct product \\ 
115& 48 & $(\Z_4 \circ \Z_4) \times \Z_2$ & $B_{23}$ & $B_{17}$ & $\frac{17}{23}$ & $\frac 58$ &  direct product \\ \hline
116& 38 & $\Z_8 \circ \D_4$ & $B_{19}$ & $B_{13}$ & $\frac{13}{19}$ & $\frac 58$ &  central product \\ 
117& 42 & $\Z_4 \circ \D_8$ & $B_{21}$ & $B_{11}$ & $\frac{11}{21}$ & $\frac{7}{16}$ &  central product \\ 
118& 49 & $\D_4 \circ \D_4$ & $B_{25}$ & $B_{16}$ & $\frac{16}{25}$ & $\frac{17}{32}$ &  central product \\ 
119& 50 & $\D_4 \circ \Q_8$ & $B_{21}$ & $B_{13}$ & $\frac{13}{21}$ & $\frac{17}{32}$ &  central product \\ 
120& 11 & $\Z_4 \wr \Z_2$ & $B_{19}$ & $B_{9}$ & $\frac{9}{19}$ & $\frac{7}{16}$  &  wreath product \\ 
121& 27 & $\Z_2^2 \wr \Z_2$ & $B_{25}$ &$B_{13}$ & $\frac{13}{25}$ & $\frac{7}{16}$  &  wreath product \\ \hline
	122& 24 &  $\Z_4^2 \rtimes_1 \Z_2$ & $B_{19}$ & $B_{13}$ & $\frac{13}{19}$ & $\frac{5}{8}$ & semidirect product \\ 
	123& 33 & $\Z_4^2 \rtimes_2 \Z_2$ & $B_{33}$ & $B_{10}$ & $\frac{10}{33}$ & $\frac {7}{16}$ &  semidirect product \\ 
	124& 43 & $\Z_8 \rtimes \Z_2^2$ & $B_{23}$ & $B_{10}$ & $\frac{10}{23}$ & $\frac{11}{32}$ &  semidirect product \\ 
	125& 4 &  $\Z_8 \rtimes \Z_4$ & $B_{17}$ & $B_{11}$ & $\frac{11}{17}$ & $\frac 58$ &  semidirect product \\ 
	126& 6 & $\Z_2^2 \rtimes \Z_4$ & $B_{21}$ & $B_{8}$ & $\frac{8}{21}$ & $\frac{11}{32}$ &  semidirect product \\ 
	127& 12 & $\Z_4 \rtimes \Z_8$ & $B_{17}$ & $B_{12}$ & $\frac{12}{17}$ & $\frac 58$ &  semidirect product \\ 
	128& 5 & $\Z_2^2 \rtimes \Z_8$ & $B_{19}$ & $B_{12}$ & $\frac{12}{19}$ & $\frac 58$ &  semidirect product \\ 
	129& 28 & $\Z_4 \rtimes_1 \D_4$ & $B_{23}$ & $B_{12}$ & $\frac{12}{23}$& $\frac{7}{16}$ &  semidirect product \\ 
	130& 34 & $\Z_4 \rtimes_2 \D_4$ & $B_{25}$ & $B_{13}$ & $\frac{13}{25}$& $\frac{7}{16}$ &  semidirect prod.\@, gen.\@ dihedral \\ 
	131& 29 & $\Z_2^2 \rtimes \Q_8$ & $B_{19}$ & $B_{12}$ & $\frac{12}{19}$& $\frac{7}{16}$ &  semidirect product \\ 
	132& 35 & $\Z_4 \rtimes \Q_8$ & $B_{17}$ & $B_{13}$ & $\frac{13}{17}$ & $\frac{7}{16}$ &  semidirect product \\ 
	133& 9 & $\D_4 \rtimes \Z_4$ & $B_{21}$ & $B_{10}$ & $\frac{10}{21}$ & $\frac{7}{16}$ &  semidirect product \\ 
	134& 10 & $\Q_8 \rtimes \Z_4$ & $B_{17}$ & $B_{10}$ & $\frac{10}{17}$ & $\frac{7}{16}$ &  semidirect product \\ \hline
	135& 15 & $\Z_8 \cdot \Z_4$ & $B_{17}$ & $B_{9}$ & $\frac{9}{17}$ & $\frac{7}{16}$ &  non-split extension \\ 
	136& 44 & $\Z_8 \cdot \Z_2^2$ & $B_{19}$ & $B_{10}$ & $\frac{10}{19}$ & $\frac{11}{32}$ &   non-split extension \\ 
	137& 32 & $\Z_4^2 \cdot \Z_2$ & $B_{17}$ & $B_{11}$ & $\frac{11}{17}$ & $\frac{7}{16}$ &   non-split extension \\ 	
	138& 2  & $\Z_2 \cdot \Z_4^2$ & $B_{19}$ & $B_{13}$ & $\frac{13}{19}$ & $\frac{5}{8}$ &   non-split extension \\ 
	139& 7  & $\Z_4 \cdot \D_4$ & $B_{21}$ & $B_{8}$  & $\frac{8}{21}$ & $\frac{11}{32}$ &   non-split extension \\ 
	140& 31 & $\Z_4 \cdot_4 \D_4$ & $B_{21}$ & $B_{11}$ & $\frac{11}{21}$ & $\frac{7}{16}$ &   non-split extension \\ 
	141& 8 & $\Z_4 \cdot_{10} \D_4$ & $B_{17}$ & $B_{8}$ & $\frac{8}{17}$& $\frac{11}{32}$ &   non-split extension \\ 
	142& 13 & $\Z_4 \cdot \Q_8$ & $B_{17}$ & $B_{9}$ & $\frac{9}{17}$ & $\frac{7}{16}$ &   non-split extension \\ 
	143& 30 & $\Z_2^2 \cdot \D_4$ & $B_{21}$ & $B_{11}$ & $\frac{11}{21}$ & $\frac {7}{16}$ &   non-split extension \\ 
	144& 14 & $\Z_2 \cdot \Q_8$ & $B_{17}$ & $B_{11}$ & $\frac{11}{17}$ & $\frac{7}{16}$ &   non-split extension \\ \hline
	\end{tabular}$$
\end{table}

\renewcommand{\arraystretch}{1}
\begin{table}[h!]
	\caption{Bell numbers $B_1,\ldots, B_{32}$}  \label{tabla bells}
	\begin{tabular}{||l|r||l|r||} \hline
$B_1$ & 1 & $B_{17}$ & 82.864.869.804 \\ \hline
$B_2$ & 2 & $B_{18}$ & 682.076.806.159 \\ \hline
$B_3$ & 5 & $B_{19}$ & 5.832.742.205.057 \\ \hline
$B_4$ & 15 & $B_{20}$ & 51.724.158.235.372 \\ \hline
$B_5$ & 52 & $B_{21}$ & 474.869.816.156.751 \\ \hline
$B_6$ & 203 & $B_{22}$ & 4.506.715.738.447.323 \\ \hline
$B_7$ & 877 &$B_{23}$ & 44.152.005.855.084.346 \\ \hline
$B_8$ & 4.140 & $B_{24}$ &445.958.869.294.805.289 \\ \hline
$ B_9$ & 21.147 & $B_{25}$ & 4.638.590.332.229.999.353 \\ \hline
$ B_{10}$ & 115.975 & $B_{26}$ & 49.631.246.523.618.756.274 \\ \hline
$ B_{11}$ &  678.570 & $B_{27}$ & 545.717.047.936.059.989.389 \\ \hline
$ B_{12}$ & 4.213.597 & $B_{28}$ & 6.160.539.404.599.934.652.455 \\ \hline
$ B_{13}$ & 27.644.437 & $B_{29}$ & 71.339.801.938.860.275.191.172 \\ \hline
$ B_{14}$ & 190.899.322 & $B_{30}$ & 846.749.014.511.809.332.450.147 \\ \hline
$ B_{15}$ &  1.382.958.545 & $B_{31}$ & 10.293.358.946.226.376.485.095.653 \\ \hline
 $ B_{16}$ & 10.480.142.147 & $B_{32}$ & 128.064.670.049.908.713.818.925.644 \\ \hline
	\end{tabular}
\end{table}

%\section{Chain metrics and 2-weight metrics}
%We are interested in 2-weight invariant metrics on groups, i.e.\@ the simplest metrics beyond the Hamming metric which is the only invariant metric of weight 1 for a group.
%We will denote by $\widetilde{\mathcal{M}}_2(G)$ the number of 2-weight invariant metrics of $G$.
%
%
%
%\begin{prop}
%Let $(G,d)$ be a metric group and $k=k(G)$ the number defined in \eqref{kG}. Then we have:
%\begin{equation} \label{M2G}
	%\widetilde{\mathcal{M}}_2(G) = \begin{cases} 
	    %\displaystyle \sum_{i=1}^{\lfloor k/2 \rfloor} \binom ki & \qquad \text{if $k$ is odd}, \\[2mm]
	   %\displaystyle \sum_{i=1}^{\frac k2 -1} \binom{k}{i} + \tfrac 12 \binom{k}{\frac k2} & \qquad \text{if $k$ is even}.
	%\end{cases}
%\end{equation}
%\end{prop}
%
%{\gray Me parece que deberia ser $2^{k-1}-1$}
%
%
%
%
%\section{Duality and MacWilliams identities}
%The numbers $A_i=\#P_i$ for $i=0,\ldots,s$ ($A_0=1$) are called the frequencies of the weights. 
%The \textit{spectrum} of the metric group $(G,d)$ is $Spec(G,d)=\{A_0,A_1,\ldots, A_s\}$ and
%the \textit{weight enumerator polynomial} of $(G,d)$ is
%\begin{equation} \label{enumerator}
%W_{G,d}(x) = \sum_{g\in G} X_{w(g)} = \sum_{i=0}^s A_i X_i.
%\end{equation}
%
%
%Agregar enumeradores de peso completo, simetrizados y comentarios de que para Hamming coinciden y para Lee no. 
%
%
%
%\section{Appendix: Tables}

\subsection*{Final remarks}
With the results of Section \ref{S5}, the number of invariant and bi-invariant metrics for several groups of low finite order $n$ can be computed. One needs an explicit presentation of the group $G$ to find the number of elements of order 2 and the table of characters of $G$ (for the information on the conjugacy classes). With the data of the groups given in the webpage GroupNames \cite{GN} one can obtain these numbers for almost all groups of order $n \le 128$.


\begin{thebibliography}{XXX}
\bibitem{Baer} \textsc{Reinhold Baer}.
\textit{Situation der Untergruppen und Struktur der Gruppe}. S.-B. Heidelberg. 
Akad.\@ Math.-Nat.\@ Klasse 2 (1933) 12--17. 

\bibitem{Ba95} \textsc{Vladimir Batagelj}. 
\textit{Norms and distances over finite groups}. 
Journal of Combinatorics, Information and System Sci \textbf{20}, (1995) 243--252.

\bibitem{Brualdi} \textsc{R.A.\@ Brualdi, J.S.\@ Graves, K.M.\@ Lawrence}.
\textit{Codes with a poset metric}.
Discrete Mathematics \textbf{147:1-3}, (1995) 57--72.

\bibitem{Hom} \textsc{I.\@ Constantinescu, W.\@ Heise}. 
\textit{A metric for codes over residue class rings of integers.}
Probl.\@ Pered.\@ Inform.\@ {\bf 33:3}, (1997) 22--28.

\bibitem{Deza} \textsc{M.M.\@ Deza, E.\@ Deza}. 
\textit{Encyclopedia of Distances}. 
Springer, 2nd Edition, 2013. 

\bibitem{ET} \textsc{P.\@ Erdös, P.\@ Turán.} 
\textit{On some problems of a statistical group theory IV}. 
Acta Math.\@ Acad.\@ Sci.\@ Hungary \textbf{19}, (1968) 413--435.

\bibitem{Fle} \textsc{Colin R.\@ Fletcher}. 
\textit{Rings of Small Order}. 
The Mathematical Gazette 64(427), (1980) 9--22. %doi:10.2307/3615885 

\bibitem{Gluesing}
\textsc{Heide Gluesing{-}Luerssen}.
\textit{Fourier-reflexive partitions and MacWilliams identities for additive codes}. 
Des.\@ Codes Cryptography \textbf{75:3}, (2015) 543--563.

\bibitem{Gre} \textsc{M.\@ Greferath, S.\@ E.\@ Schmidt}. 
\textit{Finite ring combinatorics and MacWilliams’ equivalence theorem}. 
J.\@ Combin.\@ Theory Ser.\@ A 92:1, (2000) 17--28.

\bibitem{GN} 
\textsc{Tim Dokchitser}\textit{GroupNames}, webpage \url{people.maths.bris.ac.uk/~matyd/GroupNames/}.

\bibitem{Gu}  \textsc{W.H.\@ Gustafson}. 
\textit{What is the probability that two group elements commute?} 
Amer.\@ Math.\@ Monthly \textbf{80}, (1973) 1031--1034.

\bibitem{HKCSS} \textsc{A.R.\@ Hammons, P.V.\@ Kumar, A.R.\@ Calderbank, N.J.A.\@ Sloane, P.\@ Solé}. 
\textit{The $\Z_4$-linearity of Kerdock, Preparata, Goethals, and related codes}. 
IEEE Trans.\@ Inf.\@ Theory \textbf{40:2}, (1994) 301--319.

\bibitem{Lee} 
\textsc{C.\@ Lee}.
\textit{Some properties of nonbinary error-correcting codes}. 
IRE Transactions on Information Theory \textbf{4:2}, (1958) 77--82. 

\bibitem{Le} \textsc{P.\@ Lescot}.
\textit{Isoclinism classes and commutativity degrees of finite groups}. 
J.\@ Algebra \textbf{177:3}, (1995) \! 847--869.

\bibitem{Ne}
\textsc{A.A.\@ Nechaev}. 
\textit{The Kerdock code in a cyclic form}. 
Diskret.\@ Mat.\@ \textbf{1:4} (1989) 123--139. 
English translation in Discrete Math.\@ Appl.\@ \textbf{1} (1991), 365--384.

\bibitem{PP} \textsc{L.\@ Panek, N.M.P.\@ Panek}. 
\textit{Symmetry group of ordered Hamming block space}. 2017, arXiv:1705.09987. 

\bibitem{PV} \textsc{R.\@ Podest\'a, M.\@ Vides}. 
\textit{Isometries between finite groups}. 
Discrete Mathematics \textbf{343:11} (2020), Article ID 112070. %, arXiv:1705.09987. 

\bibitem{RT-metric}
\textsc{M.Y.\@ Rosenbloom, M.A.E.\@ Tsfasman}. 
\textit{Codes for the $m$-metric}. 
Probl.\@ Pered.\@ Inform.\@ \textbf{33:1}, (1997) 55--63.

\bibitem{Salagean-Mandache99}
\textsc{Ana Salagean-Mandache}.
\textit{On the isometries between $\Z_{p^{k}}$ and $\Z_p^{k}$}.
IEEE Trans.\@ Inf.\@ Theory \textbf{45:6}, (1999) 2146--2148.

\bibitem{YO} 
\textsc{B.\@ Yildiz, Z.\@ \"Odemi\c{s} \"Ozger}.
\textit{Generalization of the Lee weight to $\Z_{p^{k}}$}. 
TWMS J.\@ App.\@ Eng.\@ Math.\@ \textbf{2}, (2012) 145--153. 
\end{thebibliography}
\end{document}